\documentclass[12pt]{amsart}
\usepackage{pifont}
\usepackage{dsfont}
\usepackage{anysize}
\usepackage{mathrsfs}
\marginsize{2.7cm}{2.44cm}{2.0cm}{3.0cm}
\usepackage[small,it]{caption}
\usepackage[OT2,T1]{fontenc}
\DeclareSymbolFont{cyrletters}{OT2}{wncyr}{m}{n}
\DeclareMathSymbol{\Sha}{\mathalpha}{cyrletters}{"58}
 \usepackage[latin1]{inputenc}
 \usepackage[dvips]{graphicx}
 \usepackage{wrapfig}
 \usepackage{amsmath}
 \usepackage{amsthm}
 \usepackage{amsfonts}
 \usepackage{amssymb}
 \usepackage{layout} \usepackage{verbatim}
 \usepackage{alltt}
\usepackage{yfonts}
\usepackage{color}

\definecolor{pAlgae}{RGB}{87,115,135}
\definecolor{airforceblue}{rgb}{0.36, 0.54, 0.66}
	\definecolor{bondiblue}{rgb}{0.0, 0.58, 0.71}
\definecolor{britishracinggreen}{rgb}{0.0, 0.26, 0.15}
\definecolor{camouflagegreen}{rgb}{0.47, 0.53, 0.42}
\definecolor{darkcyan}{rgb}{0.0, 0.55, 0.55}

 \usepackage{hyperref}

\definecolor{pinegreen}{rgb}{0.0, 0.47, 0.44}

\newcommand{\blue}{\color{blue}}

\let\oldFootnote\footnote
\newcommand\nextToken\relax

\renewcommand\footnote[1]{%
    \oldFootnote{#1}\futurelet\nextToken\isFootnote}

\newcommand\isFootnote{%
    \ifx\footnote\nextToken\textsuperscript{,}\fi}

\makeatletter
\newcommand{\mylabel}[2]{#2\def\@currentlabel{#2}\label{#1}}
\makeatother

\usepackage[all]{xy}
\usepackage{setspace}

\newtheorem*{thma}{Theorem A}

\newtheorem*{thmc}{Theorem C}
\newtheorem*{thmd}{Theorem D}

\newtheorem*{core}{Corollary E}

\newtheorem*{cortothmA}{Corollary B}

\newcommand{\LL}{\Lambda}

\newcommand{\QQ}{\mathbb{Q}}

\newcommand{\lra}{\longrightarrow}
\newcommand{\ZZ}{\mathbb{Z}}

\newcommand{\ra}{\rightarrow}

\newcommand{\be}{\begin{equation}}
\newcommand{\ee}{\end{equation}}

\newcommand{\al}{\mathcal{L}}

\newcommand{\RR}{\mathcal{R}}

\newcommand{\cl}{\textup{cl}}

\newcommand{\Dst}{\mathbb D_{\mathrm{st}}}
\newcommand{\F}{\mathrm{Fil}}
\newcommand{\Fr}{\mathrm{Fr}}
\newcommand{\res}{\mathrm{res}}
\newcommand{\pr}{\mathrm{pr}}
\newcommand{\spl}{\mathrm{spl}}
\newcommand{\Ddagrig}{\mathbb{D}^\dagger_{\textup{rig}}}
\newcommand{\DdagrigA}{\mathbb{D}^\dagger_{\textup{rig},A}}
\newcommand{\DdagrigfrakB}{\mathbb{D}^\dagger_{\textup{rig},\frak{B}}}
\newcommand{\CDst}{\mathcal D_{\mathrm{st}}}
\newcommand{\CDcris}{\mathcal D_{\mathrm{cris}}}
\newcommand{\Dcris}{\mathbb D_{\mathrm{cris}}}

\newcommand{\BrigA}{\widetilde{\mathbb B}_{\mathrm{rig},A}^{\dagger}}
\newcommand{\Brigr}{\widetilde{\mathbb B}_{\mathrm{rig}}^{\dagger,r}}
\newcommand{\BrigrA}{\widetilde{\mathbb B}_{\mathrm{rig},A}^{\dagger, r}}
\newcommand{\RG}{\mathbf{R}\Gamma}
\newcommand{\Iw}{\mathrm{Iw}}
\newcommand{\CC}{\mathbb C}

\numberwithin{equation}{section}
\newtheorem{thm}{Theorem}[section]
\newtheorem{lemma}[thm]{Lemma}
\newenvironment{define}{\par\medskip\noindent\refstepcounter{thm}
\bgroup{\hspace*{-0.15 cm}\bf{Definition}
\thethm.}\bgroup}{\egroup \egroup\par\medskip}\newtheorem{prop}[thm]{Proposition}
\newtheorem{cor}[thm]{Corollary}
\newenvironment{rem}{\par\medskip\noindent\refstepcounter{thm}
\bgroup{\hspace*{-0.15 cm}\bf{Remark} \thethm.}\bgroup}{\egroup
\egroup\par\medskip} \DeclareMathOperator{\id}{id}\parskip 2pt

\newcounter{Athm}[section]

\renewcommand{\theAthm} {\arabic{Athm}}

\long\def\symbolfootnote[#1]#2{\begingroup%
\def\thefootnote{\fnsymbol{footnote}}\footnote[#1]{#2}\endgroup}

\begin{document}
\title[Exceptional zeros and Perrin-Riou's conjecture]{O\lowercase{n the exceptional zeros of $p$-non-ordinary $p$-adic} $L$-\lowercase{functions and a conjecture of} {P}\lowercase{errin}-{R}\lowercase{iou}}

\author{Denis Benois}
\author{K\^az\i m B\"uy\"ukboduk}

\address{Denis Benois \hfill\break\indent Institut de Math\'ematiques, Universit\'e de Bordeaux  \hfill\break\indent 351, Cours de la Lib\'eration 33405  \hfill\break\indent Talence, France}
\email{\normalsize denis.benois@math.u-bordeaux.fr}

\address{K\^az\i m B\"uy\"ukboduk\hfill\break\indent UCD School of Mathematics and Statistics \hfill\break\indent University College Dublin \hfill\break\indent  Ireland }
\email{\normalsize kazim.buyukboduk@ucd.ie}

\begin{abstract}
Our goal in this article is to prove a form of $p$-adic Birch and Swinnerton-Dyer formula for the second derivative of the $p$-adic $L$-function associated to a newform $f$ which is non-crystalline semistable at $p$ at its central critical point, by expressing this quantity in terms of a $p$-adic (cyclotomic) regulator defined on an extended trianguline Selmer group. We also prove a two-variable version of this result for height pairings we construct by considering infinitesimal deformations afforded by a Coleman family passing through $f$. This, among other things, leads us to a proof of an appropriate version of Perrin-Riou's conjecture in this set up.
\end{abstract}

\maketitle
\tableofcontents
\section{Introduction}
Fix forever an odd prime $p$. The primary goal in this article is to address a conjecture of Mazur--Tate--Teitelbaum for a $p$-semistable non-ordinary eigenform $f$ of arbitrary even weight, extending the results of \cite{kbbexceptional, venerucciarticle} in weight $2$. To that end, we will prove a formula relating the $p$-adic Abel--Jacobi image of a Heegner cycle on a suitably defined Kuga-Sato variety over a Shimura curve to the Beilinson--Kato element, generalizing the work of Venerucci in weight $2$. To achieve so, we will rely on a result of Seveso and the theory of $p$-adic heights we develop here by considering infinitesimal deformations afforded by a Coleman family $\mathbf{f}$ passing through $f$ (extending the approach in \cite{venerucciarticle}). If we specialize our results to forms of weight $2$, the theory we develop in this article allows us to recover Venerucci's results in \cite{venerucciarticle}, relying on the finiteness of the Tate--Shafarevich groups of elliptic curves of analytic rank at most one, the validity of the rank part of the Birch and Swinnerton--Dyer Conjecture and the non-vanishing of the $\al$-invariant.

Let $r_{\textup{an}}(f)$ denote the order of vanishing of the Hecke $L$-function $L(f,s)$ at its central critical point $s=k/2$. 
{Some of our results here simultaneously extend the work of Kato, Kurihara and Tsuji on the Mazur--Tate--Teitelbaum conjecture\footnote{Unpublished, but see  \cite[Th\'eor\`eme 4.17]{Cz04} and \cite[Corollary 4.3.3]{Ben14a}. Note that  \cite{Ben14a} treats also the near central point case.} (where similar $p$-adic leading term formulas has been established in the case $r_{\textup{an}}(f)=0$ but only for the cyclotomic $p$-adic $L$-function)
and \cite{kbbexceptional, venerucciarticle} (where only the $p$-ordinary case (weight $k=2$) has been treated).} Although the influence of these works (particularly the latter two) on our approach here will be evident to the reader, it turns out to be a rather non-trivial task to build the machine suitable for the very general set up we handle here. We believe that the framework we outline in this article is flexible enough to treat many other interesting examples and we hope that our work here would also serve as a signpost for future investigations on the exceptional zero phenomenon for non-ordinary $p$-adic $L$-functions (particularly when the corresponding complex analytic $L$-function vanishes). 

One of our main results expresses the second derivative of the  Amice--V\'elu, Manin, Vi\v{s}ik $p$-adic $L$-function $L_{p,\alpha}(f,\omega^{k/2},s)$ (where $\omega$ is the Teichm\"uller character) in terms of a regulator on an extended trianguline Selmer group (as studied in \cite{Ben15, Ben14b}). Along the way, under the hypothesis that $r_{\textup{an}}(f)=1$, we establish a relation between a certain Heegner cycle on an appropriately defined Kuga-Sato variety over a Shimura curve and the Beilinson--Kato class, confirming a higher-weight and $p$-semistable analogue of a prediction of Perrin-Riou in weight $2$.  This allows\footnote{The relation we obtain between the Heegner cycle and the Beilinson--Kato element has non-trivial content granted a suitable Gross--Zagier--Zhang formula and when the order of vanishing of $L(f,s)$ at $s=k/2$ is exactly $1$. See Theorem~\ref{thm:BKvsHeegner} and  Theorem~\ref{thm:PRconj} for details.} us to compute the order of vanishing of $L_{p,\alpha}(f,\omega^{k/2},s)$ at $s=k/2$ when $r_\textup{an}(f)=1$ and therefore obtain in this situation  a strong evidence towards a conjecture of Mazur--Tate--Teitelbaum.

Before we present our main results with more precision, let us introduce our set up.  Fix a positive integer $N$ coprime to $p$. Define $S$ to be the set consisting of the archimedean prime and the rational primes dividing $Np$.  Let $f=\underset{n=1}{\overset{\infty}\sum} a_nq^n$ be an elliptic newform of even  weight $k$ for
$\Gamma_0(Np)$ such that $a_p=p^{k/2-1}$. Denote by $W_f$ Deligne's $p$-adic Galois representation associated to $f$ and set
$V_f=W_f(k/2)$. We call $V_f$ the central critical twist of $W_f$. The two-dimensional representation $V_f$ has coefficients in a finite extension $E$ of $\mathbb Q_p$. Furthermore, it is unramified outside $S$ and  semi-stable at $p$. 

Let $\Dst (V_f)$ denote Fontaine's semistable Dieudonn\'e module associated to $V_f$ (viewed as a continuous $\textup{Gal}(\overline{\QQ}_p/{\QQ}_p)$-representation). Then $\Dst (V_f)$ is a two-dimensional $E$-vector space equipped with a Frobenius operator $\varphi$ and a monodromy $N$ given by 
\begin{align*}
&\Dst (V_f)= Ee_{\alpha}+Ee_{\beta}, \text{ where $\varphi (e_\alpha)=\alpha e_\alpha,$ $\varphi (e_\beta)=\beta e_\beta,$}\\
&N (e_\beta)=e_\alpha, \text{ and $N (e_\alpha)=0,$}  \\
& \beta=1,  \text{ and $\alpha=p^{-1}$\,.}
\end{align*}
Let $\mathbb{D}^\dagger_{\textup{rig}}(V_f)$ denote  Fontaine's (\'etale) $(\varphi,\Gamma)$-module associated to $V_f.$ We set $D:=Ee_\alpha$ which is the unique non-trivial $(\varphi,N)$-submodule of $\Dst (V_f).$ Let $\mathbb D_f$ denote the $(\varphi,\Gamma)$-submodule of  $\mathbb{D}^\dagger_{\textup{rig}}(V_f)$ associated  to $D$ by Berger,  and $\widetilde{\mathbb{D}}_f=\mathbb{D}^\dagger_{\textup{rig}}(V_f)/\mathbb D_f$. Both $\mathbb D_f$ and $\widetilde{\mathbb D}_f$ are $(\varphi,\Gamma)$-modules of rank $1$, and as such may be described rather explicitly (see Proposition~\ref{prop:ex sequence with Dalpha}). This is crucial for our calculations.

Attached to the pair $(V_f,D)$, one associates a \emph{Selmer complex} $\widetilde C_\textup{f}^{\bullet}(V_f)=S^{\bullet}
 (V_f,\mathbb{D}_f)$ (see Section~\ref{subsec:twovariablesheights}). Let $\mathbf R\Gamma (V_f,\mathbb D_f)$ denote the associated object in the derived category and define the \emph{extended trianguline Selmer group} $\widetilde H_\textup{f}^1(V_f)=\mathbf R^1\Gamma (V_f,\mathbb D_f)$ as the cohomology of the Selmer complex in degree $1$.  We define a variety of height pairings on these extended Selmer groups by considering various infinitesimal deformations of $V_f$ relying on the general framework established in \cite{Ben14b, Pot13},  generalizing previous constructions in \cite{venerucciarticle} in the $p$-ordinary situation (which in turn is based on the general theory developed in \cite{Ne06}). Among others, we have a \emph{symmetric} (cyclotomic) $p$-adic height pairing 
 $$\frak{h}_p\,: \widetilde H_\textup{f}^1(V_f) \times \widetilde H_\textup{f}^1(V_f) \lra E$$
 and a \emph{skew-symmetric} $p$-adic height pairing 
$$\frak{h}_\mathbf{f}^{\textup{c-wt}}\,: \widetilde H_\textup{f}^1(V_f) \times \widetilde H_\textup{f}^1(V_f) \lra E$$
which we call the \emph{central critical height pairing}. See Section~\ref{subsec:selmercomplexformodforms} for the definition of the former and Definition~\ref{def:centralcriticalheight} for the latter. These two height pairings appear naturally as the specializations of a \emph{two-variable} height pairing 
$$\mathcal{A}\left(\mathbb{H}_\mathbf{f}\right)     : \widetilde H_\textup{f}^1(V_f) \times \widetilde H_\textup{f}^1(V_f) \lra \frak{I}/\frak{I}^2$$
 where $\frak{I}\subset E[[\kappa-k,s]]$ is the ideal of functions those vanish at $(k,0)$. The pairing $\mathcal{A}\left(\mathbb{H}_\mathbf{f}\right)$  is introduced in Section~\ref{subsec:twovariablesheights}. It also satisfies an appropriate functional equation.
 
 The extended trianguline Selmer group compares with the Bloch--Kato Selmer group $H^1_\textup{f}(\QQ,V_f)$ via the exact sequence
 $$0\lra H^0(\widetilde{\mathbb D}_f) \stackrel{\partial_0}{\lra} \widetilde H^1_\textup{f}(V_f)
\lra H^1_\textup{f}(\QQ,V_f) \lra 0\,.$$
This exact sequence  admits a natural splitting $\spl: \widetilde H^1_f(V_f) \ra H^0(\widetilde{\mathbb D}_f)$ (Proposition \ref{prop:splitting}), and in our setting it turns out that the $E$-vector space $H^0(\widetilde{\mathbb D}_f)$ is one-dimensional. This is a reflection of the (simple) \emph{exceptional zero} that the $p$-adic $L$-function $L_{p,\alpha}(f,\omega^{k/2},s)$ possesses at the central critical value $s=k/2$.  

Until the end of this introduction, suppose $L(f,k/2)=0$. It follows from \cite{ka1} (see also \cite{Cz04},  \cite{Ben14a}) that $\textup{ord}_{s=k/2}\, L_{p,\alpha}(f,\omega^{k/2},s)\geq 2$. Furthermore, Kato's work \cite{ka1} equips us with an element $[z_f^{\textup{BK}}]\in H^1_\textup{f}(\QQ,V_f)$, which is the first layer of the Euler system of Beilinson--Kato elements. We note that a natural extension of a conjecture of Perrin-Riou~\cite[\S 3.3.2]{pr93grenoble} would assert that the class $[z_f^{\textup{BK}}]$ is non-zero if and only if $\textup{ord}_{s=k/2}L(f,s)=1$; see Theorem~D below for our result towards this prediction. We let $[\frak{z}_f^{\textup{BK}}]\in \widetilde{H}^1_\textup{f}(V_f)$ 
denote the canonical  lift of the class $[z_f^{\textup{BK}}]$ to $\widetilde{H}^1_\textup{f}(V_f)$. Let $d_{\widetilde{\delta}}$ denote a distinguished generator of $H^0(\widetilde{\mathbb D}_f)$ that we introduce at the very end of Section~\ref{subsec:selmercomplexformodforms} and which depends on the choice of
$e_{\alpha}.$ 
Finally, for a class $[x] \in H^1_\textup{f}(\QQ,V_f)$ we write
$$\Omega\left([x]\right):=\left(1-\frac{1}{p}\right)\Gamma(k/2)^{-1}\left\langle\Psi_2,\textup{res}_p([x])\right\rangle\cdot\left\langle\Psi_1,\partial_0(d_{\widetilde{\delta}})\right\rangle \in E\,.$$
Here $\{\Psi_1,\Psi_2\}$ is a distinguished basis of $H^1(\widetilde{\mathbb D}_f)$ which is given by Definition~\ref{def:psi1psi2} and the pairing $\left\langle\,,\,\right\rangle: H^1(\widetilde{\mathbb D}_f)\times H^1({\mathbb D}_f)\ra E$ is the natural pairing.

We are now ready to state a sample of our main results. The first is a $p$-adic 
Birch and Swinnerton-Dyer  formula for the second derivative of the $p$-adic $L$-function.
\begin{thma} 
\label{thma}
We have
$$\frac{\Omega \cdot [e_{\alpha},b_f^*]_{V_f}}{2}\cdot \frac{d^2}{ds^2}  L_{p,\alpha}(f,\omega^{k/2},s)\big{|}_{s=k/2}=
\det\left(\begin{array}{cc} \frak{h}_p\left(\partial_0(d_{\widetilde{\delta}}),\partial_0(d_{\widetilde{\delta}})\right) & \frak{h}_p\left(\partial_0(d_{\widetilde{\delta}}),[\frak{z}_f^{\textup{BK}}]\right)\\\\  \frak{h}_p\left([\frak{z}_f^{\textup{BK}}],\partial_0(d_{\widetilde{\delta}})\right) &  \frak{h}_p\left([\frak{z}_f^{\textup{BK}}],[\frak{z}_f^{\textup{BK}}]\right)\end{array}\right)\,,$$
where $\Omega=\Omega\left([z_f^{\textup{BK}}]\right)$, $b_f^*$ denote the  basis
of $\F^0\Dst (V_f)$ defined by Kato in \cite[Theorem 12.5]{ka1}, and
\[
[\,,\,]_{V_f}\,:\,\Dst (V_f)\times \Dst (V_f)\rightarrow E
\]
the canonical duality. 
\end{thma}

Note that the submodule $D\subset \Dst (V_f)$ defines a canonical splitting 
of the Hodge filtration on $\Dst (V)$\,:
\[
\Dst (V_f)=D\oplus \F^0\Dst (V_f).
\]
We let 
\[
h^{\textup{Nek}}\,:\,H^1_{\textup{f}}(\QQ,V_f)\otimes H^1_{\textup{f}}(\QQ,V_f)\lra E
\]
denote Nekov\'a\v r's  $p$-adic height pairing \cite{Ne92} associated to this splitting. Theorem~11 of \cite{Ben14b} gives a precise relationship between $h^{\textup{Nek}}$ and $\frak{h}_p$ which generalizes \cite[Theorem~11.4.6]{Ne06} to the non-ordinary case. Applying a standard argument (see, for example, the proof of \cite[Theorem~7.13]{Ne92}) we obtain the following reformulation of Theorem~A.

\begin{cortothmA} Assume that the Fontaine-Mazur $\al$-invariant $\mathcal{L}_{\textup{FM}}(f)$ does not vanish. Then, 
\begin{multline*}\left (1-\frac{1}{p}\right )\frac{[e_{\alpha},b_f^*]_{V_f}}{2\Gamma(k/2)} \langle \Psi_2, \res_p([z_f^{\textup{BK}}])\rangle \cdot \frac{d^2}{ds^2}  L_{p,\alpha}(f,\omega^{k/2},s)\big{|}_{s=k/2}=\\
\mathcal L_{\textup{FM}}(f)\cdot h^{\textup{Nek}}([z^{\textup{BK}}_f], [z^{\textup{BK}}_f]). 
\end{multline*}
\end{cortothmA}

Theorem~A is Theorem~\ref{tim:main} in the main text and it in fact follows from the following leading term formula for the two-variable $p$-adic $L$-function $L_p(\frak{X},\kappa,s)$ defined in Section~\ref{subsec:PRslogformodforms}. Here $\frak{X} \in H^1(G_{\QQ,S},\overline{V}_\mathbf{f})$ and $\overline{V}_\mathbf{f}:=V_\mathbf{f}\otimes_E \mathcal{H}^\iota$ is the cyclotomic deformation of the big Galois representation attached to a Coleman family $\mathbf{f}$ that specializes in weight $k$ to our eigenform $f$. Let $\textup{pr}: H^1(G_{\QQ,S},\overline{V}_\mathbf{f}) \ra H^1(G_{\QQ,S},{V}_{f})$ denote the obvious projection.

\begin{thmc}
Let $[x_{\textup{f}}]\in \widetilde H^1_f(V_f)$  be the canonical lift of $[x]\in  {H}^1_\textup{f}(\QQ,V_f)$. Suppose further that there is an element $\frak{X} \in H^1(G_{\QQ,S},\overline{V}_\mathbf{f})$  with the property that $\textup{pr}\left(\frak{X}\right)=[x]$. Then,\\
\begin{align*}\Omega([x])\cdot L_p(\frak{X},\kappa,s) \equiv \det\left(\begin{array}{cc} \mathcal{A}\left(\mathbb{H}_\mathbf{f}\right) \left(\partial_0(d_{\widetilde{\delta}}),\partial_0(d_{\widetilde{\delta}})\right) & \mathcal{A}\left(\mathbb{H}_\mathbf{f}\right) \left(\partial_0(d_{\widetilde{\delta}}),[x_{\textup f}]\right)\\\\  \mathcal{A}\left(\mathbb{H}_\mathbf{f}\right) \left([x_{\textup f}],\partial_0(d_{\widetilde{\delta}})\right) &  \mathcal{A}\left(\mathbb{H}_\mathbf{f}\right) \left([x_{\textup f}],[x_{\textup f}]\right)\end{array}\right) \mod \frak{I}^3\,.\end{align*}
\end{thmc}
This is Theorem~\ref{thm:leadingtermformula}(ii) below and we may deduce Theorem~A from its statement plugging in $\kappa=k$. It may also be used together with the main results of \cite{sevesoCJM} to deduce the following version of Perrin-Riou's prediction in \cite[\S 3.3.2]{pr93grenoble}, granted the interpolation of Beilinson-Kato elements in Coleman families, as proved in \cite[Theorem A]{BB_CriticalKato1}, see also \cite{ColmezWang, hansenarxiv, NakamuraArXiv, wangarxiv}. Before we state our result, let us introduce the necessary notation. Let $L$ be an algebraic number field containing all coefficients $a_n$ of $f$. Let $E \supset \QQ_{p^2}$ denote its completion at a fixed (arbitrary) prime above $p$. Let the $\mathcal{M}_{k-2}{/\QQ}$ denote the Iovita--Spiess Chow motive of weight $k$ modular forms (whose $p$-adic realisation affords representations associated to cusp forms which are new at $pN^-$). There exists a map
$$\log\Phi^{\textup{AJ}}:  \textup{CH}^{{k}/{2}}(\mathcal{M}_{k-2}\otimes H) \lra M_k(\Gamma,E)^*,$$
(which is essentially identical to the $p$-adic Abel--Jacobi map) where 
\begin{itemize}
\item $\Gamma$ is a certain `congruence subgroup' of a suitably chosen quaternion algebra, 
\item $H$ is a certain extension of $K$,  
\item $M_k(\Gamma,E)^*$ is the dual of the space of rigid analytic modular forms for $\Gamma$. 
\item $\textup{CH}^{{k}/{2}}(\mathcal{M}_{k-2}\otimes H)$ is  the Chow group of codimension $k/2$ cycles on $\mathcal{M}_{k-2}\otimes H$.
\end{itemize}
Let  $y^\epsilon\in \textup{CH}^{{k}/{2}}(\mathcal{M}_{k-2}\otimes H)$ denote a Heegner cycle which is given as in Theorem~\ref{thm:sevesomainmazurKitag}. 

For Theorem D, we will assume that $L(f,s)$ vanishes at $s=k/2$ to odd order. This allows us to choose the character $\epsilon$ (that plays a role in the definition of the Heegner cycle, for which we refer the reader to \cite{sevesoCJM}) to be the trivial character. This is our convention for the rest of this introduction.

\begin{thmd}
Suppose that $\log\Phi^{\textup{AJ}}(y^\epsilon)\big{|}_{f^{\textup{rig}}}$ is non-trivial. Then the restriction of the Beilinson-Kato class $\textup{res}_p\left([z_{f}^{\textup{BK}}]\right)$ at the prime $p$ does not vanish. 
\end{thmd}
This is Theorem~\ref{thm:PRconj} in the main body of this article. Theorem~D may be used together with Theorem~A  in order to deduce the following evidence towards the Mazur--Tate--Teitelbaum conjecture.
\begin{core}
\label{cor:orderofvanishingofpadicLintro} 
Assume that the (cyclotomic) $p$-adic height pairing $\frak{h}_p$ is non-degenerate.  Under the hypothesis of Theorem~D we have
$$\textup{ord}_{s=k/2} L_{p,\alpha}(f,\omega^{k/2},s) =2.$$ 
\end{core}

We expect that the hypothesis (both on Theorem D and Corollary E) requiring the non-vanishing of $\log\Phi^{\textup{AJ}}(y^\epsilon)\big{|}_{f^{\textup{rig}}}$ may soon be replaced by the hypothesis that $r_{\textup{an}}(f)=1$. Indeed, a suitable Gross--Zagier--Zhang formula (which is currently unavailable in the level generality we require) together with the injectivity of the Abel--Jacobi map and the non-degeneracy of Beilinson's height pairing (none of which is currently known) would show that the non-vanishing of the Abel--Jacobi image of Heegner cycle $y^\epsilon$ is equivalent to asking that $r_{\textup{an}}(f)=1$. This in particular would show in this situation that
$$\textup{ord}_{s=k/2} L_{p,\alpha}(f,\omega^{k/2},s)=1+r_{\textup{an}}(f)\,,$$
as predicted by Mazur--Tate--Teitelbaum. We further remark that the $p$-adic Abel--Jacobi map is always expected to be injective in this set up. 
\begin{rem}
\label{rem:nonexceptional}
Although we have not checked the details, it seems that our theory here applies equally well for $p$-semistable modular forms whose $p$-adic $L$-functions do \emph{not} necessarily possess an exceptional zero and whose Hecke $L$-function vanish at their central critical point to order $1$. This leads us to a $p$-adic Gross--Zagier formula which expresses the leading term of its $p$-adic $L$-function in terms of the $p$-adic height of the Beilinson--Kato element, allowing us to extend and generalize the results of \cite{pr93grenoble} to higher weight forms.
\end{rem}
\subsection*{Acknowledgements} The both authors wish to thank M. Seveso 
who kindly pointed to us that his work should be sufficient to complement our initial results to yield a proof of Perrin-Riou's conjecture in this context. They also thank 
Enis Kaya and the anonymous  referee for pointing out several inaccuracies.
The second named author (K.B.) thanks the first author for his interest in his prior work \cite{kbbexceptional} and repeated invitations to University of Bordeaux. It is  during these visits that much of this research was carried out. K.B. also thanks Daniel Disegni, Olivier Fouquet, Benjamin Howard, Rodolfo Venerucci, Xinyi Yuan and Wei Zhang for helpful conversations, and  Henri Darmon for inviting K.B. to a CRM workshop during which he had the chance to interact with M. Seveso. 

K.B. was partially supported by the T\"UB\.ITAK Grant 113F059, Science Academy Young Investigator Award (BAGEP) and EU Horizon 2020 MC-GF Grant CriticalGZ.
\section{Selmer complexes}
\subsection{$(\varphi,\Gamma)$-modules}
In this Section we review the theory of Selmer complexes 
and  $p$-adic heights following \cite{Ben14b} and apply 
the general machinery to the case of $p$-adic representations 
arising from modular forms. In particular, we construct an
infinitesimal deformation of the $p$-adic height pairing 
in the weight direction generalizing the work of Venerucci~\cite{venerucciarticle}.

We start by introducing some general notation. Let 
$\overline{\QQ}_p$ be an algebraic closure  of $\QQ_p,$  $G_p=\mathrm{Gal}(\overline{\QQ}_p/\QQ_p)$  and $\CC_p$  the $p$-adic completion of $\overline{\QQ}_p.$ We denote by $v_p\,:\,\CC_p \rightarrow \mathbb R \cup\{+\infty\}$ the $p$-adic valuation on $\CC_p$ normalized by $v_p(p)=1$ and set $\vert x\vert_p= \left (\frac{1}{p}\right )^{v_p(x)}.$ Fix a system 
$\varepsilon =(\zeta_{p^n})_{n\geqslant 1}$ of $p^n$th roots of the unity such that $\zeta_p\neq 1$ and $\zeta_{p^{n+1}}^p=\zeta_{p^n}$ for all $n.$ 
We set  $\Gamma=\mathrm{Gal}(\QQ_p(\zeta_{p^{\infty}})/\mathbb Q_p)$ and 
denote by  $\chi\,:\,\Gamma \rightarrow \ZZ_p^*$  the cyclotomic character. 
The group $\Gamma$ decomposes canonically into the direct sum 
$\Gamma=\Delta\times \Gamma_0,$ where $\Delta =\mathrm{Gal}(\QQ_p(\zeta_p)/\QQ_p)$
and $\Gamma_0=\mathrm{Gal}(\QQ_p(\zeta_{p^{\infty}})/\mathbb Q_p(\zeta_p)).$
We denote by $\left\langle \chi\right\rangle \,:\,\Gamma \rightarrow \ZZ_p^*$ the composition of $\chi \vert_{\Gamma_0}$
with the canonical projection $\Gamma\rightarrow \Gamma_0.$

 For each $0\leqslant r<1$ we denote by $\mathrm{ann} (r,1)$
the $p$-adic annulus 
\[
\mathrm{ann} (r,1)=\{ x\in \CC_p \mid  p^{-1/r}\leqslant \vert x\vert_p <1\}.
\] 
If  $E$ is a finite extension   of $\QQ_p$ we denote by $\RR_E^{(r)}$ the
ring of power series $f(\pi)=\underset{n=-\infty}{\overset{\infty}\sum} a_n\pi^n$
with coefficients in $E$ converging on $\mathrm{ann}(r,1).$ 
The Robba ring $\RR_E$ is defined to be the union $\RR_E=\underset{0\leqslant r<1}
\cup\RR_E^{(r)}.$ We equip $\mathcal R_E$ with a continuous action of the group $\Gamma= \mathrm{Gal} (\QQ_p(\zeta_{p^{\infty}})/\QQ_p)$  and a Frobenius $\varphi$ given by 
\begin{align}  
&\tau (f(\pi))=f((1+\pi)^{\chi (\tau)}-1),\qquad \tau \in \Gamma,\\
&\varphi (f(\pi))=f((1+\pi)^p-1).
\end{align}
More precisely, we have
\begin{align*}
&\tau (\RR_E^{(r)})=\RR_E^{(r)},\qquad \tau\in \Gamma,\\
&\varphi (\RR_E^{(r)})=\RR_E^{(pr)}.
\end{align*}
The rings $\RR_E^{(r)}$ are equipped with a canonical Fr\'echet topology
\cite{Ber02}.
For each    affinoid $E$-algebra $A$ define  
\[
\RR_A^{(r)}=\RR_E^{(r)}\widehat\otimes_E A,\qquad  \RR_A=\underset{0\leqslant r<1}
\cup\RR_A^{(r)}.
\] 
The actions of $\varphi$ and $\Gamma$ on $\RR_E$  extend by linearity to $\RR_A.$ 

\begin{define} i) A $(\varphi,\Gamma)$-module over $\RR_A^{(r)}$ is 
a finitely generated projective $\RR_A^{(r)}$-module $\mathbb D^{(r)}$
equipped with the following structures:

a) A $\varphi$-semilinear map
\[
\varphi \,:\, \mathbb D^{(r)} \rightarrow \mathbb D^{(r)}\otimes_{\RR_A^{(r)}}\RR_A^{(pr)}
\]
such that the induced linear map 
\[
\varphi^*\,:\,\mathbb D^{(r)}\otimes_{\RR_A^{(r)},\varphi}\RR_A^{(pr)} \rightarrow \mathbb D^{(r)}\otimes_{\RR_A^{(r)}}\RR_A^{(pr)}
\]
is an isomorphism of $\RR_A^{(pr)}$-modules. 

b) A semilinear continuous action of $\Gamma$ on $\mathbb D^{(r)}$ commuting 
with $\varphi.$ 
\\\\
ii) $\mathbb D$ is a $(\varphi,\Gamma)$-module over $\RR_A$ if 
$\mathbb D=\mathbb D^{(r)}\otimes_{\RR_A^{(r)}} \RR_A$ for some
$(\varphi,\Gamma)$-module $\mathbb D^{(r)}$ over $\RR_A^{(r)}.$ 
\end{define}

The theory of $(\varphi,\Gamma)$-modules was initiated by Fontaine in his fundamental paper \cite{Fo}. The reader may consult  \cite{Cz03} and \cite{Cz08} for an introduction and further references. 

For each continuous character $\delta \,:\,\QQ_p^* \rightarrow A^*$, we denote by $\RR_A(\delta)$ the $(\varphi,\Gamma)$-module $\RR_A\cdot e_{\delta}$ 
of rank $1$ generated by $e_{\delta}$ and such that 
\begin{align*}
&\tau (e_{\delta})=\delta (\chi (\tau))\cdot e_{\delta},\qquad \tau \in \Gamma,\\
&\varphi (e_{\delta})=\delta (p)\cdot e_{\delta}.
\end{align*}
If $A=E$ is a finite extension of $\QQ_p,$ each $(\varphi,\Gamma)$-module of rank $1$ over $\RR_E$ is isomorphic to $\RR_E(\delta)$ for some character $\delta$, see \cite{Cz08}.

A $p$-adic representation of $G_p=\mathrm{Gal}(\overline{\QQ}_p/\QQ_p)$ with coefficients in an affinoid algebra $A$ is a finitely generated projective   $A$-module equipped with a continuous linear action of $G_p.$  The theory of $(\varphi,\Gamma)$-modules associates to each $p$-adic representation  $V$ of $G_p$ a $(\varphi,\Gamma)$-module $\DdagrigA (V)$ over $\RR_A.$  The functor $V\mapsto \DdagrigA (V)$ is fully faithful and commutes with base change. If $A=E,$ we write $\Ddagrig (V)$
instead $\DdagrigA (V)$ to simplify notation. 
A theorem of Kedlaya proved in \cite{Ke} implies that the functor $\Ddagrig$ 
establishes an equivalence between the category of $p$-adic representations 
with coefficients in $E$ and 
the category of $(\varphi,\Gamma)$-modules over $\RR_E$ of slope $0.$

For each $(\varphi,\Gamma)$-module $\mathbb D$ we denote by $\mathbb D^*$
its dual module $\mathbb D^*=\mathrm{Hom}_{\RR_A}(\mathbb D,\RR_A).$
The twist $\mathbb D^*(\chi)$  of $\mathbb D^*$ by the cyclotomic character 
is the Tate dual of $\mathbb D.$ 
\subsection{Relation to  $p$-adic Hodge theory}
A filtered $\varphi$-module (resp. a filtered $(\varphi,N)$-module) is a finite-dimensional $E$-vector space $M$ equipped with an exhaustive decreasing filtration $(\F^iM)_{i\in\ZZ}$ and a bijective 
Frobenius $\varphi \,:\,M \rightarrow M$ (resp. a bijective Frobenius $\varphi$ and  a nilpotent monodromy $N\,:\,M \rightarrow M$ such that 
$N \varphi =p\varphi N$).
In \cite{Ber08}, Berger associated to each $(\varphi, \Gamma)$-module $\mathbb D$ 
a  filtered  $(\varphi,N)$-module  $\CDst (\mathbb D)$ and a filtered $\varphi$-module 
$\CDcris (\mathbb D)$ such that  $\CDcris (\mathbb D)=\CDst (\mathbb D)^{N=0}.$ 
Moreover 
\[
\dim_E\CDcris (\mathbb D)\leqslant \dim_E\CDst (\mathbb D)\leqslant \mathrm{rank}_{\mathcal R_E}(\mathbb D).
\]
We say that a $(\varphi,\Gamma)$-module  $\mathbb D$ is semistable (resp. crystalline) if  we have $\dim_E\CDst (\mathbb D)= \mathrm{rank}_{\mathcal R_E}(\mathbb D)$
(resp. if $ \dim_E\CDcris (\mathbb D)= \mathrm{rank}_{\mathcal R_E}(\mathbb D)$).
The functor $\mathbb D\mapsto \CDcris (\mathbb D)$ (resp. $\mathbb D\mapsto \CDst (\mathbb D)$) establishes an  equivalence between the category of crystalline 
(resp. semistable) $(\varphi,\Gamma)$-modules and the category of filtered 
$\varphi$-modules (resp. $(\varphi,N)$-modules) (see \cite{Ber08}). If $V$ is a 
$p$-adic representation with coefficients in $E$, then we have canonical and functorial isomorphisms 
\[
\Dst (V)=\CDst (\Ddagrig (V)),\qquad \Dcris  (V)=\CDcris (\Ddagrig (V)),
\]
where $\Dst$ and $\Dcris$ are the classical Fontaine functors. 

\subsection{Cohomology of $(\varphi,\Gamma)$-modules}
We review the cohomology of $(\varphi,\Gamma)$-modules. Let $\Gamma=\Delta\times \Gamma_0$  be the canonical decomposition of $\Gamma$ into the direct product of 
$\Delta=\mathrm{Gal}(\QQ_p(\zeta_p)/\QQ_p)$ and the pro-$p$-cyclic group 
$\Gamma_0=\mathrm{Gal}(\QQ_p(\zeta_{p^{\infty}})/\QQ_p(\zeta_p)).$ 
Fix a generator $\gamma\in \Gamma_0.$  Let $A$ be an affinoid algebra over 
$E .$
For any $(\varphi,\Gamma)$-module $\mathbb D$ over $\RR_A$ we can consider the Fontaine--Herr
complex    
\[
C^{\bullet}_{\varphi,\gamma}(\mathbb D)\,\,:\,\, 
\mathbb{D}^{\Delta}  \xrightarrow{d_0} 
\mathbb{D}^{\Delta} \oplus \mathbb{D}^{\Delta} 
\xrightarrow{d_1} \mathbb{D}^{\Delta},
\]
where  $d_0(x)=((\varphi-1)(x), (\gamma-1)x)$ and $d_1(y,z)=(\gamma-1)(y)-(\varphi-1)(z)$
(see \cite{H1}, \cite{Liu}, \cite{KPX}). Define 
\[
H^i(\mathbb D)=H^i(C^{\bullet}_{\varphi,\gamma}(\mathbb D)).
\] 
Then $H^i(\mathbb D)$ are finitely  generated  $A$-modules 
(\cite[Theorem 0.2]{Liu}, \cite[Theorem 4.4.2]{KPX}). In addition, we have 
a canonical isomorphism
\begin{equation}
\label{BrauerisomorphismforphiGammamodules}
H^2(\RR_A(\chi)) \simeq A
\end{equation}
given by $[x]\mapsto -\left (1-\frac{1}{p}\right )^{-1}\log^{-1} \chi (\gamma) \,\res (xdt),$ 
where $t=\log (1+\pi).$

If $\mathbb D=\DdagrigA (V),$ where $V$ is a $p$-adic representation of $G_p$  with coefficients in $A,$ 
there exist canonical (up to the choice of $\gamma$) and functorial isomorphisms
\begin{equation}
\label{formula:computationgaloiscohviaphigamma}
H^i(\DdagrigA (V))\simeq H^i(\QQ_p,V), \qquad i\in \ZZ.
\end{equation}
(see \cite[Theorem 0.1]{Liu}, \cite[Theorem 2.8]{Pot13}). Note that $H^i(\mathcal R_A(\chi))\simeq H^i(\QQ_p, A(\chi)).$

If $\mathbb D^*(\chi)$ is the dual $(\varphi,\Gamma)$-module, then 
the canonical pairing  $\mathbb D\otimes \mathbb D^*(\chi) \rightarrow \mathcal R_E(\chi)$ induces a well defined cup-product 
\[
\cup_{\varphi,\gamma}\,:\,C^{\bullet}_{\varphi,\gamma}(\mathbb D) \otimes 
C^{\bullet}_{\varphi,\gamma}(\mathbb D^*(\chi) ) \rightarrow C^{\bullet}_{\varphi,\gamma}(\mathcal R_E(\chi)).
\]
Together with the isomorphism (\ref{BrauerisomorphismforphiGammamodules})
this gives the duality for $(\varphi,\Gamma)$-modules 
\[
C^{\bullet}_{\varphi,\gamma}(\mathbb D)\simeq \mathbf{R}\mathrm{Hom}_A
(C^{\bullet}_{\varphi,\gamma}(\mathbb D^*(\chi)),A).
\]

For our construction of $p$-adic heights we need the derived version of the isomorphism
(\ref{formula:computationgaloiscohviaphigamma}) proved in \cite{Ben14b}.  Namely, let $\Brigr$ denote the ring of $p$-adic periods constructed by Berger in \cite{Ber02}. Set $\BrigrA=\Brigr \widehat\otimes_{\QQ_p} A,$ 
$\BrigA=\underset{r>0}\cup \BrigrA,$  $V^{\dagger}_{\mathrm{rig}}=V\otimes_A\BrigA$ and consider the complex $C^{\bullet}(G_p,V^{\dagger}_{\mathrm{rig}}).$
The exact sequence 
\[
0\rightarrow A\rightarrow \BrigA \xrightarrow{\varphi-1}\BrigA \rightarrow 0
\]
induces an exact sequence
\[
0\rightarrow C^{\bullet}(G_p,V)\rightarrow 
C^{\bullet}(G_p,V^{\dagger}_{\mathrm{rig}}) \xrightarrow{\varphi-1}
C^{\bullet}(G_p,V^{\dagger}_{\mathrm{rig}}) \rightarrow 0.
\]
Consider the total complex 
\[
K_p^{\bullet}(V)=\mathrm{Tot} \left (C^{\bullet}(G_p,V^{\dagger}_{\mathrm{rig}}) \xrightarrow{\varphi-1}
C^{\bullet}(G_p,V^{\dagger}_{\mathrm{rig}})
\right ).
\]
The canonical map $V\rightarrow V^{\dagger}_{\mathrm{rig}}$ induces a morphism
\[
\xi \,:\, C^{\bullet}(G_p,V)\rightarrow  K^{\bullet}_p(V)
\]
given by 

\[
\xi (x)=(0,x)\in C^{n-1}(G_p,V^{\dagger}_{\mathrm{rig}})\oplus C^n(G_p,V^{\dagger}_{\mathrm{rig}}), \qquad x\in C^n(G_p,V).
\]
On the other hand, consider the complex  
\[
C^{\bullet}_{\gamma}(\Ddagrig (V))\,\,:\,\,
\Ddagrig (V)^{\Delta}\xrightarrow{\gamma-1}\Ddagrig (V)^{\Delta},
\]
where the terms are placed in degrees $0$ and $1.$ We have a morphism 
of complexes 
\[
\alpha \,:\, C^{\bullet}_{\gamma}(\Ddagrig (V)) \lra C^{\bullet}(G_p,V^{\dagger}_{\mathrm{rig}})
\]
defined by 
\begin{equation*}
\begin{aligned}
&\alpha (x)=x,  &&{\hbox{ if } } x\in C^0_{\gamma}(\Ddagrig (V))\,,\\
&\alpha (x)(g)=\frac{g-1}{\gamma-1}(x){\,,} && {\hbox{ if } }  x\in C^1_{\gamma}(\Ddagrig (V)), \quad g\in G_p.
\end{aligned}
\end{equation*}
Since $C^{\bullet}_{\varphi,\gamma}(\Ddagrig (V))=\mathrm{Tot} \left(C^{\bullet}_{\gamma}(\Ddagrig (V))\xrightarrow{\varphi-1}C^{\bullet}_{\gamma}(\Ddagrig (V))\right ),$ 
this morphism induces a morphism (which we denote again by $\alpha$):
\begin{equation*}
\alpha \,:\, C^{\bullet}_{\varphi,\gamma}(\Ddagrig (V)) \rightarrow K_p^{\bullet}(V).
\end{equation*}

\begin{prop}
\label{proposition quasi-isomorphisms} The maps $\alpha$ and $\xi$ are quasi-isomorphisms and we have a diagram 
\begin{equation*}
\xymatrix{
C^{\bullet}(G_p,V) \ar[r]^{\xi}_{\simeq} &K_p^{\bullet}(V)\\
& C^{\bullet}_{\varphi,\gamma}(\Ddagrig (V))\,. \ar[u]^{\alpha}_{\simeq}
}
\end{equation*}

\end{prop}
\begin{proof} This is \cite[Proposition 2.4.2]{Ben14b}. See also
\cite[Proposition 9]{Ben15b}.
\end{proof}

\subsection{Iwasawa cohomology} 
\label{subsec:iwasawacohomology}
For each $n\geqslant 1$ we set $K_n=\QQ_p(\zeta_{p^{n+1}})^{\Delta}.$
Thus $K_n/\QQ_p$ is the cyclotomic  extension of degree $p^n.$
Let 
$\Lambda=O_E[[\Gamma_0]]$ denote the Iwasawa algebra over $O_E.$ Then 
the choice of a generator $\gamma \in \Gamma_0$ fixes an isomorphism 
$\Lambda \simeq O_E[[X]]$ such that $\gamma \mapsto 1+X.$ We denote by 
$\mathcal H$ the algebra of power series $f(X)=\underset{i=0}{\overset\infty\sum}
a_iX^i$ with coefficients in $E$ that converge on the $p$-adic open unit disk. 
If $A$ if an affinoid $E$-algebra, we set $\Lambda_A=\Lambda\widehat\otimes_{O_E} 
A$ and $\mathcal H_A=\mathcal H\widehat\otimes_EA.$
 The map $\iota (\tau)=\tau^{-1},$ $\tau \in \Gamma_0$  defines  involutions $\iota \,:\,\Lambda_A \rightarrow \Lambda_A$ and $\mathcal H_A \rightarrow \mathcal H_A.$ 

Let $V$ be a $p$-adic representation of $G_{p}$ with coefficients in 
an affinoid algebra $A.$
Fix a unit ball $A^o$ of $A.$ Then $V=T\otimes_{O_E}E$ for some finitely generated projective  $A^o$-submodule $T$ of $V$ stable under  $G_{p}.$  We write 
$V\otimes_{A}\Lambda_A^{\iota}$ (respectively
$V\otimes_{A}\mathcal H_A^{\iota}$) for $V\otimes_{A}\Lambda_A$ (respectively
 $V\widehat\otimes_{A}\mathcal H_A$) equipped with the diagonal  action of $G_p$ 
and the $\Lambda_A$-module (respectively $\mathcal H_A$-module) structure   given  by 
\[
\tau (v\otimes \lambda)=v\otimes \tau^{-1}\lambda,\qquad \tau \in \Gamma_0,
\quad v\in V,\quad \lambda \in\Lambda_A \,\,\textrm{(resp. $\lambda \in\mathcal H_A$).} 
\]
Consider the complexes $C^{\bullet}(G_{p},V\otimes_{A}\Lambda_A^{\iota})$
and $C^{\bullet}(G_{p},V\otimes_{A}\mathcal H_A^{\iota})$ of continuous cochains 
with coefficients in $V\otimes_{A}\Lambda_A^{\iota}$ and 
$V\otimes_{A}\mathcal H_A^{\iota}$ respectively.  Then 
\[
H^i\left (C^{\bullet}(G_{p},V\otimes_{A}\Lambda_A^{\iota})\right ) \simeq
H^i_{\Iw} (\QQ_p,V),
\]
where 
\[
H^i_{\Iw} (\QQ_p,V)=\varprojlim_{n} H^i(G_{K_n},T)\otimes_{O_E}E
\]
and 
\[
H^i(\QQ_p,V\otimes_{A}\mathcal H_A^{\iota}) \simeq
H^i_{\Iw} (\QQ_p,V)\otimes_{\Lambda_A}\mathcal H_A
\]
( see \cite{Pot13} Theorem~1.6).  To simplify notation, we set $\overline V=V\otimes_{A}\mathcal H_A^{\iota}.$ 

For any $(\varphi,\Gamma)$-module $\mathbb D$ over $\RR_A$ define 
$\overline{\mathbb D}=\mathbb D\otimes_{\RR_A}\DdagrigA (\mathcal H_A^{\iota})$
(see \cite{KPX} and \cite{Potcyc} for more detail).  The Iwasawa cohomology 
$H^{\bullet}_{\Iw}(\mathbb D)$ of $\mathbb D$ is defined to be the cohomology of the complex 
\[
C^{\bullet}_{\Iw}(\mathbb D)=C^{\bullet}_{\varphi,\gamma}(\overline{\mathbb D}).
\]
Note that this complex is quasi-isomorphic to the complex
\[
\left [\mathbb D^{\Delta}\xrightarrow{\psi-1}\mathbb D^{\Delta} \right ],
\]
where $\psi$ is the left inverse to $\varphi$ and the terms are placed in degrees $1$ and $2$ (\cite{KPX}, Theorem~4.4.8).  We have  natural  quasi-isomorphisms 
\begin{align*}
&C^{\bullet}_{\Iw}(\mathbb D)\otimes^{\mathbb L}_{\mathcal H_A}A\simeq 
C^{\bullet}_{\varphi, \gamma}(\mathbb D),\\
&C^{\bullet}(G_{p},\overline V) \simeq 
C^{\bullet}_{\Iw}(\DdagrigA (V)).
\end{align*}
In particular, 
\[
H_{\Iw}^*(\QQ_p,V)\otimes_{\Lambda_A}\mathcal H_A\simeq H^*_{\Iw}(\DdagrigA (V)).
\]

\subsection{The Bloch--Kato Selmer group} 
In this section we assume that $A=E.$
As usual, the first cohomology group  $H^1(\mathbb D)$ classifies extensions of 
the trivial $(\varphi,\Gamma)$-module $\mathcal R_E$ by $\mathbb D$: 
\[
0\rightarrow \mathbb D\rightarrow \mathbb D_x \rightarrow \mathcal R_E\rightarrow 0.
\]
We denote by $H^1_{\textup{f}}(\mathbb D)$ the subgroup of crystalline extensions, namely
\[
H^1_{\textup{f}}(\mathbb D)=\{ [x]\in H^1(\mathbb D) \mid \dim_E\CDcris (\mathbb D_x)=
\dim_E\CDcris (\mathbb D)+1\}.
\]
Note that this definition agrees with $H^1_{\textup{f}}$ of Bloch and Kato \cite[Proposition 1.4.2]{Ben11}, 
\cite{Nak}.  

Let $\delta\,\,:\,\,\QQ_p^*\rightarrow E^*$ be a continuous character.  
The $(\varphi,\Gamma)$-module  $\mathcal R_E(\delta)$ is semistable (and therefore crystalline) if and only if $\delta \vert_{\ZZ_p^*}(u)=u^m$ for some $m\in\ZZ.$
The following proposition summarizes main information about the cohomology of such modules.

\begin{prop}
\label{prop:propertiesofmodulesofrank1} 
Let $\delta\,:\,\QQ_p^*\rightarrow E^*$ be a continuous character such that 
$\delta \vert_{\ZZ_p^*}(u)=u^m$ for some $m\in\ZZ.$

i) If $m\leqslant 0$ and $\delta\neq x^m,$ then $\dim_E H^1(\mathcal R_E(\delta))=1$
and $H^1_{\textup{f}}(\mathcal R_E(\delta))=0.$ 

ii) If $m\leqslant 0$ and $\delta= x^m,$ then the map
\begin{align*}
&i_{\delta}\,\,:\,\,\CDcris (\mathcal R_E(\delta))\times \CDcris (\mathcal R_E(\delta))
\rightarrow H^1(\mathcal R_E(\delta)),\\
&i_{\delta}(x,y)=\cl (t^{-m}x,t^{-m}y),\qquad t=\log (1+\pi) 
\end{align*}
is an isomorphism. Let $i_{\delta, f}$ and $i_{\delta, c}$ denote the restriction 
of $i_{\delta}$ on the first and  second summand respectively. 
Then $\textup{im} (i_{\delta,f})=H^1_{\textup{f}}(\mathcal R_E(\delta))$ and we obtain the decomposition
\begin{equation}
\label{formula:decompositioncohomisoclinic}
H^1(\mathcal R_E(\delta))\overset{(\pr_f,\pr_c)}{=} H^1_{\textup{f}}(\mathcal R_E(\delta))\oplus H^1_c(\mathcal R_E(\delta)),
\end{equation}
where $H^1_c(\mathcal R_E(\delta))=\textup{im} (i_{\delta,c}).$

iii) If $m\geqslant 1$ and  $\delta (x)\neq \vert x\vert x^{m},$ then 
$H^1_{\textup{f}}(\mathcal R_E(\delta))=H^1(\mathcal R_E(\delta))$ is a
one-dimensional  $E$-vector space. 

iv) If $\delta (x)=\vert x\vert x^{m},$ then  $\chi \delta^{-1} (x)=x^{1-m}.$
Consider  local duality 
\[
\cup\,\,:\,\,H^1(\mathcal R_E (\delta))\times H^1(\mathcal R_E (\chi\delta^{-1}))
\rightarrow E
\]
and denote by $[\,\,,\,\,]\,:\,\CDcris (\mathcal R_E(\delta))\times 
\CDcris (\mathcal R_E(\chi\delta^{-1})) \rightarrow E$ the canonical pairing. 
Then the map
\begin{equation*}
i_{\delta}\,\,:\,\,\CDcris (\mathcal R_E(\delta))\times \CDcris (\mathcal R_E(\delta))
\rightarrow H^1(\mathcal R_E(\delta))
\end{equation*}
defined by 
\begin{equation*}
i_{\delta}(x,y)\cup i_{\chi\delta^{-1}} (\alpha,\beta)= [y,\alpha]-[x,\beta]
\end{equation*}
is an isomorphism.  Let $i_{\delta, f}$ and $i_{\delta, c}$ denote the restriction 
of $i_{\delta}$ on the first and  second summand respectively. 
Then $\textup{im} (i_{\delta,f})=H^1_{\textup{f}}(\mathcal R_E(\delta))$ and again we have 
the decomposition (\ref{formula:decompositioncohomisoclinic}) with
$H^1_c(\mathcal R_E(\delta))=\textup{im} (i_{\delta,c}).$ 
\end{prop}
\begin{proof} See Proposition~1.5.3, Proposition~1.5.4 and  Theorem~1.5.7 of \cite{Ben11}.
\end{proof}

\subsection{Selmer complexes} In this subsection we review the construction of Selmer 
complexes for non-ordinary Galois representations following \cite{Pot13} and \cite{Ben14b}. Fix a finite set of primes $S$ such that $p\in S.$ Let 
$G_{\QQ,S}$ denote the Galois group of the maximal algebraic extension of $\QQ$ 
unramified outside $S$ and $\infty$. For each prime $\ell$ we fix a decomposition group at $\ell$ which we identify with $G_{\ell}=\mathrm{Gal} (\overline \QQ_\ell/\QQ_\ell),$ and denote by $I_{\ell}$ the inertia subgroup of 
$G_{\ell}.$ If  $G$ is a topological group and $M$ a continuous $G$-module, we will write 
$C^{\bullet}(G,M)$ for the complex of continuous cochains of $G$ with coefficients in $M.$

 Let $V$ be a $p$-adic representation of $G_{\QQ,S}$
with coefficients in an affinoid algebra $A$. We denote by $\DdagrigA (V)$ the $(\varphi,\Gamma)$-module associated to the restriction of $V$ on the decomposition group at $p.$
Let $\mathbb D$ be a $(\varphi,\Gamma)$-submodule of $\DdagrigA (V)$ such that 
$\mathbb D$ is an $\RR_A$-direct summand of $\DdagrigA (V).$
Set $U_p^+(V, \mathbb D)=C^{\bullet}_{\varphi,\gamma}(\mathbb D).$  Composing the quasi-isomorphism $\alpha$ of Proposition~\ref{proposition quasi-isomorphisms}  with the canonical morphism 
\linebreak
$U_p^+(V, \mathbb D) \rightarrow C^{\bullet}_{\varphi,\gamma}(\DdagrigA (V)),$ we obtain a map 
\[
i_p^+\,:\,U_p^+(V, \mathbb D) \rightarrow K^{\bullet}_p(V),
\]
which we will consider as a local condition at $p.$  For each $\ell\in S\setminus \{p\}$ we set 
\[
U_{\ell}^+(V)= \left [V^{I_{\ell}} \xrightarrow{\Fr_{\ell}-1} V^{I_{\ell}}\right ],
\]
where $\Fr_{\ell}$ denotes the geometric Frobenius and  the terms are placed in degrees $0$ and $1.$
Define
\[
i_{\ell}^+\,:\,U_{\ell}^+(V) \rightarrow C^{\bullet}(G_{\ell},V)
\]
by 
\[
\begin{aligned}
&i_{\ell}^+(x)= x &&\text{in degree $0$,}\\
&(i_{\ell}^+(x))(\Fr_{\ell})=x &&\text{in degree $1$.}
\end{aligned}
\]
Set  
\[
K_S^{\bullet}(V)=K_p^{\bullet}(V) \bigoplus \left (\underset{\ell\in S\setminus \{p\}}\bigoplus C^{\bullet}(G_{\ell},V)\right )
\] 
and $U^+_S(V, \mathbb D)=U^+_{p}(V,\mathbb D) \oplus \left (\underset{\ell\in S\setminus \{p\}}\bigoplus U_{\ell}^+(V)
\right ).$ To uniformize notation, we will often write $K_{\ell}^{\bullet}(V)$ for $C^{\bullet}(G_{\ell},V)$ and $U_{\ell}^+(V,\mathbb D)$ for  $U_{\ell}^+(V)$ if $\ell\neq p.$ We have a diagram
\[
\xymatrix{
C^{\bullet}(G_{\QQ,S},V) \ar[r]^{\res_S} &K_S^{\bullet}(V)\\
& U_S^+(V, \mathbb D) \ar[u]^{i_S^+},
}
\]
where $i_S^+=(i_{\ell}^+)_{\ell\in S}$ and $\res_S$ denotes the localization map. 
\begin{define}
The Selmer complex associated to these data is defined as 
\[
S^{\bullet}(V, \mathbb D)=\mathrm{cone} \left [C^{\bullet}(G_{\QQ,S},V) \oplus 
U_S^+(V,\mathbb D) \xrightarrow{\res_S-i_S^+} K_S^{\bullet}(V) \right ][-1].
\]
 \end{define}
Each element $x_\textup{f}\in S^{i}(V, \mathbb D)$ may be written as a triple
\[
x_{\textup{f}}=(x,(x_{\ell}^+)_{\ell\in S}, (\lambda_{\ell})_{\ell\in S}),
\]
where $x\in C^{i}(G_{\QQ,S},V),$
$(x_{\ell}^+)_{\ell\in S}\in U_S^+(V,\mathbb D)^i$ and $(\lambda_{\ell})_{\ell\in S}\in K_S^{i-1}(V),$ 
and $x_{\textup{f}}$ is a cocycle if and only if 
\[
d(x)=0, \quad d \left ((x_{\ell}^+)_{\ell\in S}\right )=0, \quad
i_S((x_{\ell}^+)_{\ell\in S})=\res_S(x)+d(\lambda_{\ell})_{\ell\in S}.
\]
To simplify notation, we will often write $x_{\textup{f}}=(x,(x_{\ell}^+), (\lambda_{\ell}))$
in place of 
\linebreak
$x_{\textup{f}}=(x,(x_{\ell}^+)_{\ell\in S}, (\lambda_{\ell})_{\ell\in S}).$
\begin{define}
We denote by   $\RG (V,\mathbb D)$ the class of $S^{\bullet}(V, \mathbb D)$
in the  derived category of  $A$-modules  and define  
\[
H^i(V, \mathbb D):= \mathbf{R}^i\Gamma (V,\mathbb D).
\]
For each cocycle $x\in S^{i}(V, \mathbb D)$ we write $[x]$ for the class 
of $x$ in $H^i(V, \mathbb D).$
\end{define}

\begin{define}
\label{def:thesingularcone}
For each $\ell \in S$ we define a complex
$$\widetilde{U}_\ell(V,\mathbb{D}):=\textup{cone}\left(U_\ell^+(V,\mathbb{D})\stackrel{-i_\ell^+}{\lra} K_\ell^\bullet(V)\right)\,.$$
and set $\widetilde{U}_S(V,\mathbb{D}):={\displaystyle\bigoplus_{\ell\in S} \widetilde{U}_\ell(V,\mathbb{D})}$.
\end{define}
We have the following tautological exact triangles in the derived category
\be\label{eqn:selmersequence}
\widetilde{U}_S(V,\mathbb{D})[-1]\lra\RG (V,\mathbb D)\lra \RG(G_{\QQ,S},V)\stackrel{\widetilde{\textup{res}}_S}{\lra} \widetilde{U}_S(V,\mathbb{D})\,,
\ee
and for each $\ell \in S$,
\be\label{eqn:localsequence}
{U}_\ell^+(V,\mathbb{D})\lra K_\ell^\bullet(V)\lra \widetilde{U}_\ell(V,\mathbb{D})\lra {U}_\ell^+(V,\mathbb{D})[1]\,.
\ee

Set $\widetilde{\mathbb{D}}=\DdagrigA (V)/\mathbb D.$
Let $\RG(G_p,\mathbb{D})$ (resp. $\RG(G_p,\widetilde{\mathbb{D}})$)
denote the class of the complex $C_{\varphi,\gamma}^{\bullet}(\mathbb D)$ 
(resp. $C_{\varphi,\gamma}^{\bullet}(\widetilde{\mathbb D})$) 
in the corresponding derived category. We have a distinguished triangle
\be
\label{eqn:localsequencetilde1}
\RG(G_p,\mathbb{D})\lra \RG(G_p,\DdagrigA (V))\stackrel{\frak{s}}{\lra} \RG(G_p,\widetilde{\mathbb{D}})\lra\RG(G_p,\mathbb{D})[1]{\,.}
\ee
The quasi-isomorphism $\alpha$ of Proposition~\ref{proposition quasi-isomorphisms} together with the sequence (\ref{eqn:localsequence}) and (\ref{eqn:localsequencetilde1}) induces a functorial quasi-isomorphism 
\be\label{eqn:identifyingsungularquotients}\RG(G_p,\widetilde{\mathbb{D}})\stackrel{\rm qis}{\lra}  \widetilde{U}_p(V,\mathbb{D})\,.\ee
It is easy to see that $H^0(\widetilde{U}_{\ell}(V,\mathbb{D}))=0$
and the composition of (\ref{eqn:selmersequence}) and 
(\ref{eqn:identifyingsungularquotients}) induces a map
\be
\label{eqn:definitionofcoboundary}
\partial_0\,:\,H^0(\widetilde{\mathbb{D}})\rightarrow \mathbf R^1\Gamma (V,\mathbb D).
\ee
We give below an explicit description of this map. 
Let $d\in H^0(\widetilde{\mathbb D})$ and let $z\in \DdagrigA (V)^{\Delta}$
be any lift of $d.$ Then the class of $\partial_0(d)$ in 
$\mathbf R^1\Gamma (V,\mathbb D)$ can be represented by the cocycle 
\be
\label{eqn:formulaforcoboundary}
(0, (a^+_{\ell}),(\mu_{\ell})) \in C^1(G_{\QQ,S},V)\oplus U^+_S(V,\mathbb D)^1\oplus K^0(V),
\ee
such that  $a_{\ell}^+=\mu_{\ell}=0$ for all $\ell\neq p,$ and 
\[
\mu_p=\alpha (z),\qquad a_p^+=((\varphi-1)z, (\gamma-1)z),
\]
where $\alpha$ is the map from Proposition~\ref{proposition quasi-isomorphisms}.

\subsection{Duality for Selmer complexes}
We review the duality theory for Selmer complexes.
As usual, we denote by $\tau_{\geqslant m}$ the truncation map. 
Define
\[
Z^{\bullet}=\mathrm{cone} \left (\tau_{\geqslant 2} C^{\bullet}(G_{\QQ,S},
A(1))\xrightarrow{\res_S} \tau_{\geqslant 2} K^{\bullet}_S(A(1))\right ) [-1].
\]
The computation of the Brauer group in Class Field Theory  
yields an exact sequence 
\[
0\rightarrow H^2_S(A(1)) \rightarrow \underset{l\in S}\bigoplus H^2(\QQ_\ell, A(1))
\rightarrow A \rightarrow 0,
\]
which allows to construct a canonical, up to homotopy, quasi-isomorphism
\[
r_S\,:\,Z^{\bullet} \simeq A[-3] 
\]
(see \cite{Ne06}, Section 5.4.1). The formula
\begin{equation*}
(a_{i-1},a_i) \cup_{K,p} (b_{j-1}, b_j)= (a_i\cup b_{j-1}+(-1)^ja_{i-1}\cup \varphi (b_j),
a_i\cup b_j),
\end{equation*}
where $\cup$ denotes the cup-product of continuous Galois cochains, 
defines a morphism of complexes 
\[
\cup_{K,p}\,:\,K_p^{\bullet}(V)\otimes K_p^{\bullet}(V^*(1)) \rightarrow K^{\bullet}_p(A(1))
\]
(see \cite{Ben14b}, Proposition 1.1.5). To simplify notation,
we will write $\cup$ instead $\cup_{K,p}.$

Let $\mathbb D^{\perp}=\mathrm{Hom}_{\mathcal R}(\Ddagrig (V)/\mathbb D, \mathcal R_A(\chi))$
be the orthogonal complement to $\mathbb D$ under the canonical duality
$\DdagrigA (V)\times \DdagrigA (V^*(1)) \rightarrow \RR_A(\chi).$
Following  Nekov\'a\v r, we define a pairing 
\begin{equation*}
\cup_{V,\mathbb D}\,\,:\,\,S^{\bullet} (V,\mathbb D)\otimes_{A}
S^{\bullet} (V^*(1),\mathbb D^{\perp}) \rightarrow A[-3]
\end{equation*}
as the composition map 
\[
S^{\bullet} (V,\mathbb D)\otimes_{A}
S^{\bullet} (V^*(1),\mathbb D^{\perp}) \xrightarrow{} Z^{\bullet}
\xrightarrow{r_S} A[-3],
\]
where the first arrow is induced by the cup-product 
\begin{equation}
\label{formulaforcupproductofSelmercomplexes}
(x, (x_{\ell}^+),(\lambda_{\ell})) \cup (y, (y_{\ell}^+),(\mu_{\ell}))=
(x\cup y, (\lambda_{\ell}\cup i_S^+(y_{\ell}^+)+(-1)^{\deg (x)} \res_S (x)\cup \mu_{\ell}))
\end{equation}
(see \cite{Ne06}, Proposition~1.3.2).
Therefore we have a morphism in the derived category of $A$-modules 
\begin{equation}
\label{cupproductforSelmercomplexes}
\RG (V,\mathbb D)\otimes^{\mathbb L}_{A}
\RG (V^*(1),\mathbb D^{\perp}) \rightarrow A[-3].
\end{equation}
\begin{prop}
\label{prop:symmetryofcupproduct}
We have a commutative diagram
\begin{equation*}
\xymatrix{
\RG (V,\mathbb D) \otimes_A^{\mathbb L} \RG (V^*(1),\mathbb D^{\perp})
\ar[rr]^(.7){\cup_{V,\mathbb D}} \ar[d]^{s_{12}} &&A[-3] \ar[d]^{=}\\
\RG (V^*(1),\mathbb D^{\perp}) \otimes_A^{\mathbb L} \RG (V,\mathbb D)
\ar[rr]^(.7){\cup_{V^*(1),\mathbb D^{\perp}}}  &&A[-3],
}
\end{equation*}
where $s_{12}(x\otimes y)= (-1)^{\deg (x)\deg (y)}y\otimes x.$
\end{prop}
\begin{proof} This is \cite{Ben14b}, Theorem~3.1.8.
\end{proof}

If $A=E$ is a finite extension of $\QQ_p,$ this pairing induces a canonical duality
\begin{equation*}
\mathrm{Hom}_E(\RG (V,\mathbb D),E)\simeq \RG (V^*(1),\mathbb D^{\perp}),
\end{equation*}
but this is not true in general.

We also review here the Iwasawa theoretic analog of $\RG (V,\mathbb D).$ 
Recall that $\overline V=V\otimes_A\mathcal H_A^{\iota}$ and $\overline{\mathbb D}=
\mathbb{D}\otimes_{\RR_A}\DdagrigA (\mathcal H_A^{\iota}).$
With the previous notations define
\[
S^{\bullet}_{\Iw}(V,\mathbb D)=S^{\bullet}(\overline V,
\overline{\mathbb D}).
\]
We will write $\RG_{\Iw}(V,\mathbb D)$ for the class of 
$S^{\bullet}_{\Iw}(V,\mathbb D)$ in the derived category of $\mathcal H_A$-modules. 
Note that the following version of the control theorem holds true:
\begin{equation*}
\RG_{\Iw}(V,\mathbb D)\otimes_{\mathcal H}^{\mathbb L}E = \RG (V,\mathbb D).
\end{equation*}

\section{$p$-adic height pairings}  

\subsection{Construction of $p$-adic heights} 
\label{subsec:constructpadicheights}
We provide in this section an overview of the construction of $p$-adic heights for $p$-adic representations 
over affinoid algebras following \cite{Ben14b}. We shall make use of this general framework in order to obtain the analogues of the pairings considered in \cite{venerucciarticle} in the $p$-ordinary set up (that were in turn constructed relying on the general machine developed in \cite{Ne06}). We keep previous notation and conventions.  Let $A$ be an affinoid algebra over 
$E$ and let $V$ be a $p$-adic representation of $G_{\QQ,S}$ with coefficients in $A.$ 
We fix a $(\varphi,\Gamma)$-submodule $\mathbb D$ of $\DdagrigA (V)$ which is a
$\RR_A$-module direct summand of $\DdagrigA (V).$ Let $J_A$ denote the kernel of the augmentation map $\mathcal H_A\rightarrow A.$ Note that $J_A=(X)$ and $J_A/J_A^2\simeq A$ 
as $A$-modules. Since Selmer complexes commute with base change (see \cite{Pot13}, Theorem~1.12), the exact sequence 
\[
0\rightarrow J_A/J_A^2\rightarrow \mathcal H_A/J_A^2 \rightarrow A \rightarrow 0
\]
gives rise to a distinguished triangle 
\begin{equation}
\label{eqn: connecting cyclotomic beta}
\RG (V,\mathbb D)\otimes_AJ_A/J_A^2 \rightarrow \RG _{\Iw}(V,\mathbb D)\otimes^{\mathbb L}_{\mathcal H}\mathcal H/(X^2) \rightarrow  \RG (V,\mathbb D) \xrightarrow{\beta^{\mathrm{cyc}}_{V,\mathbb D}}
\RG (V,\mathbb D)[1]\otimes_AJ_A/J_A^2.
\end{equation}

\begin{define}
\label{definition:height pairing} The $p$-adic height pairing associated to the data 
$(V,\mathbb D)$ is defined as the morphism
\begin{multline*}
h_{V,\mathbb D}\,\,: \,\,
\RG (V,\mathbb D)\otimes_A^{\mathbb L}\RG (V^*(1),\mathbb D^{\perp})
\xrightarrow{\beta^{\mathrm{cyc}}_{V,\mathbb D}\otimes\mathrm{id}}\\
\left (\RG (V,\mathbb D)[1]\otimes J_A/J_A^2\right )\otimes_E^{\mathbb L}\RG (V^*(1),\mathbb D^{\perp})
\xrightarrow{\cup_{V,\mathbb D}} J_A/J_A^2[-2]{\,.}
\end{multline*}
\end{define}

The pairing $h_{V,\mathbb D}$ induces a pairing on cohomology groups 
\[
h_{V,\mathbb D}\,\,:\,\,
H^1(V,\mathbb D)\times H^1(V^*(1),\mathbb D^{\perp})
\lra J_A/J_A^2.
\]

\begin{prop}
\label{prop: propertiesofheights} The following diagram 
\begin{equation}
\nonumber
\xymatrix{
\RG(V,\mathbb D) \otimes_A^{\mathbb L} \RG(V^*(1),\mathbb D^{\perp}) 
\ar[rr]^(.7){h_{V,\mathbb D}^{}} \ar[d]^{s_{12}} &&J_A/J_A^2[-2]\ar[d]^{=}\\
\RG(V^*(1),\mathbb D^{\perp}) \otimes_A^{\mathbb L} \RG(V,\mathbb D)
\ar[rr]^(.7){h_{V^*(1),\mathbb D^{\perp}}^{}}
&&J_A/J_A^2[-2],
}
\end{equation}
where $s_{12}(a\otimes b)=(-1)^{\deg (a)\deg(b)} b\otimes a,$ commutes.

In particular, 
\[
h_{V,\mathbb D}^{i,j}(x,y)=(-1)^{ij}h_{V^*(1),\mathbb D^{\perp}}^{j,i}(y,x).
\]
\end{prop}
\begin{proof} This is Theorem~I of \cite{Ben14b}.
\end{proof}

Till the end of this subsection we assume that $A=E$ is a finite extension of 
$\QQ_p$ and that the restriction of $V$  to the decomposition group at $p$  is semistable. 

\begin{define}
\label{definition: splitting submodule} Assume that  $V$ is semistable at $p.$
We say that a $(\varphi,N)$-submodule $D$ of $\Dst (V)$ is a splitting submodule 
if 
\begin{equation}
\label{formula:splittingsubmodule}
\Dst (V)=D\oplus \F^0\Dst (V)
\end{equation}
as $E$-vector spaces.
\end{define} 

Let $D$ be  a splitting submodule  of $\Dst (V)$. By \cite{Ber08}, $D$ corresponds to
a unique $(\varphi,\Gamma)$-submodule $\mathbb D$  of $\Ddagrig (V)$ such that $\CDst (\mathbb D)=D.$
To simplify notation, we write $\RG (V,D)$   and $h_{V,D}$ for 
$\RG (V,\mathbb D)$ and $h_{V,\mathbb D}$ respectively.

We fix an isomorphism $J_E/J_E^2\simeq E$ setting $\gamma-1\pmod{J_E^2}\mapsto
\log \chi (\gamma).$
\begin{prop}
\label{prop: comparisionofheightswithouttrivialzeros}
Let $D$ be a splitting submodule of $\Dst (V)$. Assume that the following conditions hold true:

\textup{a)} $\Dcris (V)^{\varphi=1}=\Dcris (V^*(1))^{\varphi=1}=0;$

\textup{b)} $H^0(\Ddagrig (V)/\mathbb D)=H^0(\mathbb D^*(\chi))=0,$ where $\mathbb D^*=
\mathrm{Hom}_{\mathcal R}(\mathbb D,\mathcal R).$
\newline
Then we have $\mathbf R^1\Gamma (V,D)=H^1_{\textup{f}}(\QQ,V),$ $\mathbf R^1\Gamma (V^*(1),D^{\perp})=H^1_{\textup{f}}(\QQ,V^*(1)),$ and 
\[
h_{V,D}\,\,:\,\,
H^1_{\textup{f}} (\QQ,V)\times H^1_{\textup{f}} (\QQ,V^*(1))
\rightarrow E
\]
coincides with the $p$-adic height pairing constructed by Nekov\'a\v r in \cite{Ne92} \footnote{The $p$-adic height pairing constructed by Nekov\'a\v r depends on the choice of splitting of the Hodge filtration. One should take here the splitting defined by the decomposition \eqref{formula:splittingsubmodule}.}, where $H^1_{\textup{f}}(\QQ,-)$ stands for the Bloch--Kato Selmer group.
\end{prop}
\begin{proof} See \cite{Ben14b}, Theorem~III.
\end{proof}

Even when the splitting submodule  $D$ does not satisfy the condition b) of Proposition~\ref{prop: comparisionofheightswithouttrivialzeros}, we can still relate  $\mathbf R^1\Gamma (V,D)$ to the Bloch-Kato Selmer group if the restriction of $V$ on the decomposition group at $p$ satisfies some natural conditions. Conjecturally, such a situation appears if the associated $p$-adic $L$-function has an extra-zero. We refer to \cite[Section~7 ]{Ben14b},
for a systematic study of $p$-adic heights in this setting. In Section~\ref{sec:selmerandheightsformodforms} we review this theory in the particular case of elliptic modular forms.

\subsection{Cassels--Tate pairings} 
\label{subsec:casselstatepairings}
Nekov\'a\v r's construction of abstract Cassels-Tate pairings 
generalizes directly to our case. Let $A$ be an affinoid algebra.
We assume that $A$ is  an integral domain and denote by $\mathrm{Fr} (A)$ its field of fractions. Let $V$ be a $p$-adic representation of $G_{\QQ,S}$ with coefficients in $A$ and let  $\mathbb D$ be a $(\varphi,\Gamma)$-submodule of $\DdagrigA (V)$ such that  $\mathbb D$ is a $\RR_A$-module direct summand of $\Ddagrig (V).$  Consider the complex of flat modules 
\[
C^{\bullet}=\left [A\xrightarrow{-\mathrm{id}} \mathrm{Fr} (A)\right ] 
\]
placed in degrees $0$ and $1.$ Let $X^{\bullet}$ be a complex of 
$A$-modules.  The $\mathrm{Tor}$ spectral sequence for 
$X^{\bullet}\otimes_A C^{\bullet}$ degenerates into exact sequences
\begin{equation}
\label{spectralsequencefortor}
0\rightarrow H^{i-1}(X^{\bullet})\otimes_A (\mathrm{Fr}(A)/A)
\rightarrow H^i(X^{\bullet}\otimes_A C^{\bullet})
\rightarrow H^i(X^{\bullet})_{\mathrm{tor}}
\rightarrow 0. 
\end{equation}
Applying the functor $\otimes_AC^{\bullet}$ to the 
pairing (\ref{cupproductforSelmercomplexes}) and passing to cohomology groups, we get pairings
\[
H^i(\RG (V,\mathbb D)\otimes_A C^{\bullet}) \otimes_A 
H^j(\RG (V^*(1),\mathbb D^{\perp})\otimes_A C^{\bullet}) \rightarrow \mathrm{Fr}(A)/A,\qquad i+j=4.
\]
Since $\mathrm{Fr}(A)/A$ is $A$-divisible, it follows from (\ref{spectralsequencefortor})  that  this pairing factors through 
\begin{equation}
\label{casselstatepairing}
\mathrm{CT}_{V,\mathbb D}^{i,j}\,\,:\,\,
H^i(V,\mathbb D)_{\mathrm{tor}}\otimes_A
H^j(V^*(1),\mathbb D^{\perp})_{\mathrm{tor}} \rightarrow \mathrm{Fr}(A)/A.
\end{equation}
\begin{define}
\label{def:casselstatepairing}
We will call  $\mathrm{CT}_{V,\mathbb D}^{i,j}$ generalized Cassels-Tate 
pairings for $(V,\mathbb D).$
\end{define} 

\begin{prop} 
\label{prop:symmetryofcasselstate}
The pairings $\mathrm{CT}_{V,\mathbb D}^{i,j}  $ satisfy
\[
\mathrm{CT}_{V,\mathbb D}^{i,j}(x,y)=(-1)^{ij}\mathrm{CT}_{V,\mathbb D}^{j,i}(y,x),\qquad i+j=4.
\]
\end{prop}
\begin{proof} This follows from Proposition~\ref{prop:symmetryofcupproduct} and \cite[Formula (2.10.14.1)]{Ne06}. 
\end{proof}

\subsection{The pairing $h^{\mathrm{wt}}_{V_{\frak p},\mathbb D_{\frak p}}$}
\label{subsec:constructionofh^wt}
We maintain the assumptions of Section~\ref{subsec:casselstatepairings}. Assume in addition that 
$A$ is a principal ideal domain. 
Fix an $E$-point of $U=\mathrm{Spm}(A)$ which we will identify with a morphism $\psi\,:\,A\rightarrow E$ and set $\frak{p}=\ker (\psi).$ Let  $V_{\frak p}=V\otimes_{A,\psi}E$ and
$\mathbb D_{\frak p}=\mathbb D\otimes_{A,\psi}E.$ The exact sequence
\[
0\rightarrow \mathfrak{p}\rightarrow A\rightarrow E\rightarrow 0
\] 
induces a distinguished  triangle 
\[
\RG (V,\mathbb D)\otimes_A\frak p \rightarrow 
\RG (V,\mathbb D) \rightarrow \RG (V_{\frak p},\mathbb D_{\frak p})
\rightarrow \RG (V,\mathbb D)[1]\otimes_A\frak p.
\]
Since $\frak p$ is free over $A,$ we have 
$H^i(\RG (V,\mathbb D)\otimes_A\frak p)=
\mathbf{R}^i\Gamma (V,\mathbb D)\otimes_A\frak p,$ 
and this distinguished triangle induces an exact sequence 
\[
\ldots \lra H^1 (V_{\frak p},\mathbb D_{\frak p})\xrightarrow{\mu^{\mathrm{wt}}_{V,\mathbb D}} H^2(V,\mathbb D)\otimes_A \frak p
\lra H^2(V,\mathbb D)\lra \ldots {\,.}
\]
Since $\ker \left (H^2(V,\mathbb D)\otimes_A \frak p
\rightarrow H^2(V,\mathbb D)\right )=
H^2(V,\mathbb D)_{\frak p-\mathrm{tor}}\otimes_A\frak{p},$
we can compose $\mu^{\mathrm{wt}}_{V,\mathbb D}$ with the pairing (\ref{def:casselstatepairing}).

\begin{define} 
\label{def:abstractweightheight}
The weight-height pairing is defined to be the $E$-bilinear map 
\begin{multline}
h^{\mathrm{wt}}_{V_{\frak p},\mathbb D_{\frak p}}\,\,:\,\,
H^1 (V_{\frak p},\mathbb D_{\frak p})\otimes_E H^1 (V_{\frak p}^*(1),\mathbb D_{\frak p}^{\perp})
\xrightarrow{\mu^{\mathrm{wt}}_{V,\mathbb D}\otimes \mu^{\mathrm{wt}}_{V^*(1),\mathbb D^{\perp}}}
\\
\left (H^2(V,\mathbb D)_{\frak p-\mathrm{tor}}\otimes_A\frak{p}\right )\otimes_A 
 \left (H^2(V^*(1),\mathbb D^{\perp})_{\frak p-\mathrm{tor}}\otimes_A\frak{p}\right )\xrightarrow{\mathrm{CT}_{V,\mathbb D}^{2,2}}
\frak{p}^{-1}A/A\otimes_A\frak{p}^2\simeq \frak{p}/\frak{p}^2.  
\end{multline}
\end{define}
We remark that the pairing $h^{\mathrm{wt}}_{V_{\frak p},\mathbb D_{\frak p}}$ is 
symmetric by Proposition~\ref{prop:symmetryofcasselstate}, namely
\[
h^{\mathrm{wt}}_{V_{\frak p},\mathbb D_{\frak p}}(x,y)= 
h^{\mathrm{wt}}_{V_{\frak p}^*(1),\mathbb D_{\frak p}^{\perp}}(y,x).
\]
Tensoring the exact sequence
\[
0\rightarrow \frak p/\frak{p}^2\lra A/\frak{p}^2 \rightarrow 
E\rightarrow 0
\]
with $\RG (V,\mathbb D),$ we get  the coboundary map 
\[
\beta^{\mathrm{wt}}_{V_{\frak p},\mathbb D_{\frak p}}\,:\, H^1(V_{\frak p},\mathbb D_{\frak p}) \rightarrow H^2(V_{\frak p},\mathbb D_{\frak p})\otimes\frak p/\frak{p}^2.
\]

\begin{prop}
\label{prop:comparisiontwodefinitionsofweightheight}
We have 
\[
h^{\mathrm{wt}}_{V_{\frak p},\mathbb D_{\frak p}}(x,y)=\beta^{\mathrm{wt}}_{V_{\frak p},\mathbb D_{\frak p}}(x)\cup_{V_{\frak p},\mathbb D_{\frak p}}y.
\] 
\end{prop}
\begin{proof} The proof of this proposition follows repeating the proof of \cite[Proposition~0.17]{veneruccithesis} {verbatim} (where only the ordinary case is considered). We further remark that the weight-height pairing $h^{\mathrm{wt}}_{V_{\frak p},\mathbb D_{\frak p}}$ corresponds to $\widetilde{h}^\prime_{\mathcal{P},1,1}$ of op. cit., whereas the right-hand-side of our asserted identity corresponds to the explicit description of the height pairing $\widetilde{h}_{\mathcal{P},1,1}$ in loc. cit. (c.f., \cite[\S0.21]{veneruccithesis}).
\end{proof}

\section{Selmer complexes and $p$-adic heights  for modular forms} 
\label{sec:selmerandheightsformodforms}
\subsection{Selmer complexes for modular forms}
\label{subsec:selmercomplexformodforms}
In this section, we consider   $p$-adic representations arising from elliptic modular forms. 
Fix an integer $N\geqslant 1$ such that $p\nmid N$ and set 
$S=\{\mathrm{primes}\,\,\ell\mid N\}\cup\{p\}.$  
Let $f=\underset{n=1}{\overset{\infty}\sum} a_nq^n$ be an elliptic newform of even  weight $k$ for
$\Gamma_0(Np).$  
We denote by $W_f$ the $p$-adic representation associated to $f$ by Deligne and set
$V_f=W_f(k/2).$ Thus $V_f$ is a two-dimensional representation with coefficients in a finite extension $E$ of $\mathbb Q_p$ which is semistable at $p.$  The canonical (Poincar\'e duality) pairing
$W_f\times W_f \rightarrow E(1-k)$ induces an isomorphism
\begin{equation}
\label{formula:autodualityofV}
j\,:\,V_f\simeq V_f^*(1)\,.
\end{equation}
Since the pairing $V_f\times V_f\rightarrow E(1)$ is skew-symmetric, we have 
an anti-commutative  diagram
\begin{equation}
\begin{aligned}
\label{formula:anticommutativediagram}
\xymatrix{V_f\otimes V_f\ar[r]^{\id\otimes j} \ar[d]_{j\otimes\id} & V_f\otimes V_f^*(1)\ar[d]\\
V_f^*(1)\otimes V_f \ar[r] & E.
}
\end{aligned}
\end{equation}

Let $\Dst (V_f)$ denote Fontaine's semistable module associated to 
$V_f.$ Then $\Dst (V_f)$ is a two-dimensional $E$-vector space equipped with 
a decreasing two-step filtration, a Frobenius operator $\varphi,$ and a monodromy $N$ given by 
\begin{align*}
&\Dst (V_f)= Ee_{\alpha}+Ee_{\beta}, \text{ where $\varphi (e_\alpha)=\alpha e_\alpha,$ $\varphi (e_\beta)=\beta e_\beta,$}\\
&N (e_\beta)=e_\alpha, \text{ and $N (e_\alpha)=0,$}  \\
& \beta=p\alpha,  \text{ and $\alpha=p^{-k/2}a_p$,}\\
&\F^i \Dst (V_f)= \begin{cases}  \Dst (V_f), &\text{  if $i\leqslant -k/2$,} \\
E\,(e_\beta -  \mathcal L_{\textup{FM}}(f)e_\alpha), &\text{   if $-k/2+1\leqslant i\leqslant k/2-1$,}\\
0, &\text{if $i\geqslant k/2.$}
\end{cases}
\end{align*}
The element $\mathcal L_{\textup{FM}}(f)\in E$ that appear in the description  of the  filtration $(\F^i\Dst (V))_{i\in \mathbb Z}$ is called the Fontaine--Mazur $\mathcal L$-invariant. 

We remark  that  $D=E\, e_\alpha$  is the unique non-trivial  $(\varphi,N)$-submodule of $\Dst (V_f).$ Let $\mathbb D_f$ denote the associated  $(\varphi,\Gamma)$-submodule  of $\Ddagrig (V_f).$ We have a  tautological exact sequence
\begin{equation}
\label{taut}
0\rightarrow \mathbb D_f \xrightarrow{g} 
\mathbb{D}^\dagger_{\textup{rig}}(V_f)
\xrightarrow{h} \widetilde{\mathbb{D}}_f\rightarrow 0.
\end{equation}

Consider the Selmer complex associated to $(V_f,\mathbb D_f).$  In order to simplify notation, we will write (when $\mathbb D_f$ is understood) $S^{\bullet}
 (V_f)$ in place of $S^{\bullet}
 (V_f,\mathbb D_f)$ and set $\widetilde H_{\textup{f}}^i(V_f)=\mathbf R^i\Gamma (V_f,\mathbb D_f)$ to denote the cohomology of the Selmer complex $S^{\bullet}
 (V_f,\mathbb D_f)$ in degree $i$. The composition of the $p$-adic height pairing 
\be\label{eqn:padicpairingforfvaluesinaugmentation}
 h_{V_f,D}\,:\,\widetilde H^1_{\textup{f}}(V_f)\times \widetilde H^1_{\textup{f}}(V_f^*(1))\rightarrow J_E/J_E^2
\ee
 with the isomorphism $\widetilde H_\textup{f}^1(V_f)\simeq 
\widetilde H_\textup{f}^1(V_f^*(1))$ induced by  (\ref{formula:autodualityofV}) and the 
isomorphism
\begin{align*}
J_E/J_E^2&\,\,\,\,\simeq\,\, \,\,E,\\ 
\gamma-1 \pmod{J_E^2}&\longleftrightarrow \log \chi (\gamma)
\end{align*}
 gives an $E$-valued pairing
\begin{equation*}
\frak h_p\,\,:\,\,\widetilde H_\textup{f}^1(V_f)\times \widetilde H_\textup{f}^1(V_f) \lra E.
\end{equation*}
From Proposition~\ref{prop: propertiesofheights} and the anticommutativity of 
(\ref{formula:anticommutativediagram}) it follows that
 $\frak h_p$
is symmetric.

 We would like to compare $\widetilde H^1_{\textup{f}}(V_f)$ with  the classical Bloch-Kato Selmer group 
\linebreak 
$H^1_{\textup{f}}(\QQ,V_f).$  It follows from Proposition~\ref{prop: comparisionofheightswithouttrivialzeros} that  
\begin{equation*}
\widetilde H^1_{\textup{f}}(V_f)=H^1_{\textup{f}}(\mathbb Q,V_f), \qquad \text{if  $a_p \neq p^{k/2-1}.$} 
\end{equation*}
In the remainder of this subsection we assume that $a_p = p^{k/2-1}.$
Then $\mathbb D_f=\mathcal R_E(\delta)$ where $\delta (p)=\alpha p^{k/2}=a_p$ and 
$\delta (u)=u^{k/2},$ $u\in \ZZ_p^*.$      The quotient $\widetilde{\mathbb{D}}_f=\mathbb{D}^\dagger_{\textup{rig}}(V)/\mathbb{D}_f$ is a one-dimensional $(\varphi,\Gamma)$-module which is isomorphic  to $\mathcal R_E(\widetilde\delta)$ with
$\widetilde\delta (p)=\beta p^{1-k/2}$ and $\widetilde\delta (u)=u^{1-k/2},$ $u\in \ZZ_p^*.$ 

\begin{prop}
\label{prop:ex sequence with Dalpha}  Assume that $a_p=p^{k/2-1}.$ Then 
\\\\
\textup{i)} $\mathbb D_f \simeq \mathcal R_E(\vert x\vert x^{k/2})$ and  
$\widetilde{\mathbb D}_f\simeq \mathcal R_E(x^{1-k/2}).$ 
\\\\
\textup{ii)} The exact sequence (\ref{taut}) induces a long exact sequence
\begin{equation}
\label{cohomologysequenceinducedbyfiltration}
0\rightarrow H^0(\widetilde{\mathbb D}_f) \xrightarrow{\partial^{\mathrm{loc}}_0} H^1(\mathbb D_f) \xrightarrow{\frak{r}_1}
H^1(G_p,V_f) \xrightarrow{\frak{s}_1} H^1(\widetilde{\mathbb D}_f) \xrightarrow{\partial^{\mathrm{loc}}_1}
H^2(\mathbb D_f) \rightarrow 0,
\end{equation}
where we also have $\dim_E H^0(\widetilde{\mathbb D}_f)=\dim_E H^2(\mathbb D_f)=1.$
In particular, the element 
 $d_{\widetilde{\delta}}=t^{k/2-1}e_{\widetilde{\delta}},$ where 
 $e_{\widetilde\delta}$ is the generator of the $(\varphi,\Gamma)$-module
 $\widetilde{\mathbb D}_f=\mathcal R_E (\widetilde{\delta})$,
spans the $E$-vector space $H^0(\widetilde{\mathbb D}_f).$
Moreover,
$$H^1(\mathbb D_f)=\textup{im}(\partial^{\mathrm{loc}}_0)\oplus H^1_{\textup{f}}(\mathbb D_f)\,,\,\,\, \textup{im} (\frak{r}_1)= H^1_{\textup{f}}(G_p,V_f)\,,\,\, \hbox{ and }\,\, H^1(\widetilde{\mathbb D}_f)=\textup{im}(\frak{s}_1)\oplus H^1_{\textup{f}}(\widetilde{\mathbb D}_f).$$
\end{prop}
\begin{proof} The first part is obvious and the second is a particular case 
of \cite[Lemma~2.1.8]{Ben11}.
\end{proof}

We continue to assume that $a_p=p^{k/2-1}.$ Proposition~\ref{prop:propertiesofmodulesofrank1}
gives homomorphisms
\begin{align*}
&\partial^{\textup{loc}}_c\,:\,H^0(\widetilde{\mathbb D}_f)\xrightarrow{\partial^{\mathrm{loc}}_0} 
H^1(\mathbb D_f) \xrightarrow{\pr_{c}}H^1_c(\mathbb D_f),\\
&\partial^{\textup{loc}}_\textup{f}\,:\,H^0(\widetilde{\mathbb D}_f)\xrightarrow{\partial^{\mathrm{loc}}_0} 
H^1(\mathbb D_f) \xrightarrow{\pr_{\textup{f}}}H^1_\textup{f}(\mathbb D_f).
\end{align*}
Denote by $\rho_c\,:\,H^0(\widetilde{\mathbb D}_f) \rightarrow D$ 
and $\rho_\textup{f}\,:\,H^0(\widetilde{\mathbb D}_f) \rightarrow D$ the compositions of 
these maps with the canonical isomorphisms $H^1_c(\widetilde{\mathbb D}_f)\simeq D$
and $H^1_\textup{f}(\widetilde{\mathbb D}_f)\simeq D$, respectively. 
This discussion may be summarized in the following diagram: 
\[
\xymatrix{
D \ar[r]^(.45){i_{\delta,c}}  &  H^1_c(\mathbb D_f)  \\
H^0(\widetilde{\mathbb D}_f) \ar[r]^{\partial^{\mathrm{loc}}_0} \ar[u]^{\rho_c}
\ar[ur]^{\partial^{\textup{loc}}_\textup{c}} \ar[dr]^{\partial^{\textup{loc}}_\textup{f}}
\ar[d]_{\rho_\textup{f}} &H^1(\mathbb D_f) \ar[u]_{\pr_c} \ar[d]^{\pr_\textup{f}}\\
D \ar[r]^(.45){i_{\delta,\textup{f}}}& H^1_{\textup{f}}(\mathbb D_f)\,.
}
\]
Note that it follows from Proposition~\ref{prop:ex sequence with Dalpha} that the maps $\partial^{\textup{loc}}_\textup{c}$ and $\rho_c$ are isomorphisms, and we have a well defined map
\[
\left (\partial^{\textup{loc}}_c\right )^{-1}\circ \pr_c\,:\,H^1(\mathbb D_f) \xrightarrow{} 
 H^1_c(\mathbb D_f)
\xrightarrow{} H^0(\widetilde{\mathbb D}_f).
\]

Recall that we have defined the map
\[
\partial_0\,:\, H^0(\widetilde{\mathbb D}_f) \rightarrow  \widetilde H^1_{\textup{f}}(V_f)
\]
c.f. \eqref{eqn:definitionofcoboundary} and \eqref{eqn:formulaforcoboundary} above.

\begin{prop}\label{prop:splitting}
 Assume that  $a_p =p^{k/2-1}.$ \\
\textup{i)} The exact sequences~(\ref{eqn:selmersequence}) and \eqref{cohomologysequenceinducedbyfiltration} induce an exact sequence 
\begin{equation}
\label{sequence:relation between extended and BK selmer}
0\lra H^0(\widetilde{\mathbb D}_f) \xrightarrow{\partial_0}  \widetilde H^1_{\textup{f}}(V_f)
\lra H^1_{\textup{f}}(\mathbb Q,V_f) \lra 0.
\end{equation}
Moreover, the map $\spl \,:\,\widetilde H^1_{\textup{f}}(V_f) \rightarrow H^0(\widetilde{\mathbb D}_f)$ given by
\[ 
\spl \left([x,(x_{\ell})_{\ell\in S}, (\lambda_{\ell})_{\ell\in S}]\right)=(\partial^{\textup{loc}}_c)^{-1}\circ \pr_c(x_p)
\]
defines a canonical  splitting of \eqref{sequence:relation between extended and BK selmer}.
\\
\textup{ii)} The composition map 
$\rho_c^{-1}\circ \rho_\textup{f}\,:\,D\rightarrow D$ coincides with the multiplication by $\mathcal L_{\textup{FM}}(f).$
\end{prop} 

\begin{proof} The first statement is proved in \cite{Ben14b}, Proposition~7.1.5. 
The second statement is  proved in \cite{Ben11}, p. 1619, formula (32).  
\end{proof}

Consider the non-degenerate skew-symmetric pairing $\Ddagrig (V_f) \otimes \Ddagrig (V_f) \xrightarrow{\mathscr P} \mathcal R_E(\chi)$ induced by the pairing $V_f\otimes V_f \rightarrow E(1)$. Since $\mathbb D_f\subset \Ddagrig (V_f)$ is its own orthogonal complement with respect to the pairing $\mathscr{P}$, 
we have an induced non-degenerate pairing $\widetilde{\mathbb D}_f \otimes  {\mathbb D}_f \rightarrow \mathcal R_E(\chi)$, which in turn gives rise to a pairing
\[
\left\langle \,,\,\right\rangle \,:\,H^1(\widetilde{\mathbb D}_f)\otimes 
H^1(\mathbb D_f) \lra E.
\]
We fix some notation that is relevant to the constructions above and which will be used in Theorem~\ref{thm:thepropertiesofthetwovarheight} below. We continue to assume that $a_p=p^{k/2-1}.$ 
Recall what that the action of $\Gamma$ on the element  $t=\log (1+\pi)$ is given by
$\gamma (t)=\chi (\gamma) t,$ $\gamma\in \Gamma$ and that $\varphi (t)=pt.$
\begin{define}
\label{def:psi1psi2}
Let $\Psi_1:=[(-t^{k/2-1}e_{\widetilde{\delta}},0)]$ and $\Psi_2:=\log \chi(\gamma) [(0,t^{k/2-1}e_{\widetilde{\delta}})]$ be two elements of $H^1(\widetilde{\mathbb D}_f).$
\end{define}
We note that $\left\{\Psi_1, \Psi_2\right\}$ is a basis of $H^1(\widetilde{\mathbb D}_f)$ by Proposition~\ref{prop:propertiesofmodulesofrank1}. Furthermore,  $\Psi_1$ spans the crystalline subspace $H^1_\textup{f}(\widetilde{\mathbb D}_f)$ and 
$\Psi_2$ spans the  subspace $H^1_c(\widetilde{\mathbb D}_f).$ Let $\{\Psi_1^*,\Psi_2^*\}$
denote the skew-dual basis of $H^1(\mathbb D_f)$, in the sense that we have
\[
\left\langle \Psi_1,\Psi_1^*\right\rangle=1,\qquad \left\langle \Psi_2, \Psi_2^*\right\rangle=-1,\qquad \left\langle\Psi_1,\Psi_2^*\right\rangle=\left\langle
\Psi_2, \Psi_1^*\right\rangle=0.
\]
These elements are compared with those defined in \cite[Section 1.2.5]{Ben14a} in the following way: 
$$\Psi_1=\mathbf{x}_{k/2-1}, \,\Psi_2=\mathbf{y}_{k/2-1},\,\Psi_1^*=\boldsymbol{\beta}_{k/2},\,\, \hbox{ and }\,\,\,\Psi_2^*=
\boldsymbol{\alpha}_{k/2}\,.$$

\begin{cor}
\label{ref:crucialproperty1and2}
A class $[x,(x_{\ell}^+)_{\ell\in S}, (\lambda_{\ell})_{\ell\in S}] \in \widetilde{H}^1_\textup{f}(V_f)$ is the lift of $[x]\in H^1_\textup{f}(\QQ,V_f)$ under the splitting of Proposition~\ref{prop:splitting}(i) if and only if $\left\langle \Psi_1\,,\,[x_p^+]\right\rangle=0$.
\end{cor}
\begin{proof}
Note that a class $[x_\textup{f}]=[x,(x_{\ell}^+)_{\ell\in S}, (\lambda_{\ell})_{\ell\in S}]$ is the lift of a class $[x]\in H^1_\textup{f}(\QQ,V_f)$ under the splitting of Proposition~\ref{prop:splitting}(i) iff 
\begin{align*}
[x_f] \in \ker\left(\spl\right)&=\left\{[x_\textup{f}]=[x,(x_{\ell}^+)_{\ell\in S}, (\lambda_{\ell})_{\ell\in S}] \in \widetilde{H}^1_\textup{f}(V_f)\,: (\partial^{\textup{loc}}_c)^{-1}\circ \pr_c\left([x_p^+]\right)=0\right\}\\
&=\left\{[x_\textup{f}] \in \widetilde{H}^1_\textup{f}(V_f)\,:\pr_c\left([x_p^+]\right)=0\right\}\\
&=\left\{[x_\textup{f}] \in \widetilde{H}^1_\textup{f}(V_f)\,: [x_p^+] \in \textup{span}\{\Psi_2^*\}\right\}\\
&=\left\{[x_\textup{f}]\in \widetilde{H}^1_\textup{f}(V_f)\,: \left\langle\Psi_1, [x_p^+]\right\rangle=0 \right\}
\end{align*}
as we have claimed.
\end{proof}

\subsection{$p$-adic families of modular forms} 
\label{subsec:padicfamilies}
Let $U=\overline B(k,p^{-r})$ ($r\geqslant1$) denote the closed disk about $k$ of radius $p^{-r}$
in the weight space $\mathcal W.$  We consider $U$ as 
an affinoid space. The ring $\mathcal O(U)$ of analytic functions on $U$ 
is isomorphic to the Tate algebra  $A=E\left \{\left \{\displaystyle\frac{w}{p^r}\right \}\right \}.$  Define
\[
\kappa (w)=k+\frac{\log(1+w)}{\log(1+p)} \in A.
\]
Then $w=(1+p)^{\kappa (w)-k}-1.$  For each $f\in A$ we define  
$\mathcal A^{\mathrm{wt}} (f)\in E[[\kappa-k]]$ by setting 
$\mathcal A^{\mathrm{wt}} (f):=f((1+p)^{\kappa-k}-1)$. We also set $\varpi_\kappa:=\frac{\log(1+w)}{\log(1+p)} \in A$ so that $\mathcal{A}^{\textup{wt}}(\varpi_\kappa)=\kappa-k$.

Recall that  $\mathcal H$ denotes the ring of formal power series 
$f(X)\in E[[X]]$ which converge on the open unit disk and that $\Gamma_0$ denotes 
the maximal pro-$p$-cyclic subgroup of $\Gamma=\mathrm{Gal}(\QQ_p(\zeta_{p^\infty})/\QQ_p).$
Fix a generator 
$\gamma\in \Gamma_0$ and set 
\[
s=\frac{\log(1+X)}{\log\chi (\gamma)}.
\]
Then $X=\chi (\gamma)^s-1$ and for each $f(X)\in \mathcal H$ we set
$\mathcal A^{\mathrm{cyc}}(f)=f(\chi (\gamma)^{-s}-1).$  

\begin{rem}
 \label{rem:cyclopairingheight}
 The pairing given as the compositum of the arrows
 $$\widetilde H^1_{\textup{f}}(V_f)\times \widetilde H^1_{\textup{f}}(V_f)\stackrel{j^{-1}}{\lra}\widetilde H^1_{\textup{f}}(V_f)\times \widetilde H^1_{\textup{f}}(V_f^*(1))\stackrel{h_{V,D}}{\lra} J_E/J_E^2 \stackrel{\mathcal{A}^{\textup{cyc}}}{\lra} E$$
 is the pairing $-\frak{h}_p$\,.
 \end{rem}
 
The transformations $\mathcal A^{\mathrm{wt}}$ and $\mathcal A^{\mathrm{cyc}}$
induce a map $\mathcal A\,:\,\mathcal H_A \rightarrow E[[\kappa-k,s]]$ which we call 
the two-variable Amice transform. In more explicit form,  we set 
\[
\mathcal A(f):=\underset{n\in \ZZ}\sum \mathcal A^{\mathrm{wt}}(a_n)\cdot  \mathcal A^{\mathrm{cyc}}(X)^n.
\] 
for $f(X)=\underset{n\in \ZZ}\sum a_nX^n$ with $a_n\in A$. Let $\boldsymbol{\chi}\,:\,\Gamma_0
\rightarrow \mathcal H_A^\times$ be the character given by
\[
\boldsymbol{\chi}(\tau)=\chi (\tau)^{\varpi_\kappa},
\qquad \tau \in \Gamma_0.
\] 
For each $\kappa \in k+p^{r-1}\ZZ_p$, we shall denote by $\psi_{\kappa}$
the morphism 
\begin{align*}
\psi_\kappa:\, A&\lra E\\
w &\longmapsto (1+p)^{\kappa-k}-1.
\end{align*}
Set $I=\{ \kappa \in \ZZ_{\geq 2} \mid \kappa\equiv k\pmod{(p-1)p^{r-1}}\} $ and let
\[
\mathbf{f}=\underset{n=1}{\overset{\infty}\sum} \mathbf{a}_nq^n \in A[[q]]
\] 
be a $p$-adic family of cuspidal eigenforms passing through $f$, in the sense of \cite{coleman}. 
This means that for every  point  $\kappa\in I$ the series  $\mathbf{f}_{\kappa}=\underset{n=1}{\overset{\infty}\sum} \psi_{\kappa}(\mathbf{a}_n)q^n$ is the $q$-expansion of a weight $\kappa$ eigenform on $\Gamma_0(Np)$ and $\mathbf{f}_k=f.$ On shrinking $\overline B(k,p^{-r})$ if necessary 
and using \cite[Corollary B5.7.1]{coleman}, we may assume that $\mathbf{f}$ is a family of constant slope $k/2-1$.

Let $W_{\mathbf f}$ denote the big Galois representation associated to the family $\mathbf{f}$ with coefficients in $A=\mathcal O(U).$ Set $V_{\mathbf f}=W_{\mathbf f}(k/2).$
We have a skew-symmetric pairing 
\begin{equation}
\label{pairingforVf}
V_{\mathbf{f}}\times V_{\mathbf{f}} \rightarrow A(\chi{\boldsymbol{\chi}}^{-1}).
\end{equation}  
In particular, the representation $V_{\mathbf f}(\boldsymbol{\chi}^{1/2})$ is 
self-dual.
\begin{rem}
\label{rem:reconstructtheselfdualHidafamily}
Let $\displaystyle{\Theta:=\frac{\gamma-\boldsymbol{\chi}^{{1}/{2}}\left(\gamma^{-1}\right)}{\log\chi(\gamma)} \in \mathcal{H}_A}$. We then have the following natural isomorphism of Galois modules:
$$\overline{V}_{\mathbf{f}}\big{/}\Theta\cdot \overline{V}_{\mathbf{f}}\cong V_{\mathbf{f}}(\boldsymbol{\chi}^{1/2})\,.$$
We remark that we have $\gamma^{-1}$ in the definition of $\Theta$ (as opposed to $\gamma$ itself) due to our definition of the Galois action on $\mathcal{H}_A^\iota$.
\end{rem}

Set $V_{\kappa}=V_{\mathbf{f}}\otimes_{A,\psi_{\kappa}}E.$ When $\kappa$ is a positive integer 
we have $V_{\kappa}=W_{\mathbf{f}_{\kappa}}(k/2),$ where 
$W_{\mathbf{f}_{\kappa}}$ is the $p$-adic Galois representation
associated to $\mathbf{f}_{\kappa}$ by Deligne. According to Kisin~\cite{overconvergentmodformsandFMconj}, there exists an
analytic function $\alpha (w)\in A$ with values in $E$ such that 

$\bullet$ $\psi_{k}(\alpha)=\psi_k(\mathbf{a}_p) p^{-k/2};$

$\bullet$ $\Dcris (V_{\kappa})^{\varphi=\psi_{\kappa}(\alpha)}$ is a one-dimensional $E$-vector space for each $\kappa \in I.$ 

To simplify notation, we will often write $\alpha (\kappa)$ instead $\psi_{\kappa}(\alpha).$ 
The second condition implies that $V_{\kappa}$ is trianguline
and therefore semistable at  all $\kappa\in I.$

 Let $\alpha (\kappa)$ and 
$\beta (\kappa)$ denote the eigenvalues of $\varphi$ acting on 
$\Dst (V_{\kappa}).$  Since the Hodge weights of $V_{\kappa}$ are 
$(-k/2,\kappa-k/2-1),$ from the weak admissibility of $\Dst (V_{\kappa})$
it follows that 
\[
v_p(\alpha (\kappa))+v_p(\beta (\kappa))=\kappa-k-1, \qquad \kappa \in I.
\]
Since $v_p(\alpha (\kappa))=-1,$ we have 
$v_p(\beta (\kappa ))=\kappa-k.$ This implies 
that $V_{\kappa}$ is crystalline for all $\kappa\in I\setminus\{k\}.$ 
By \cite[Theorem~0.3.4]{Liu}, this data defines a triangulation of 
the $(\varphi,\Gamma)$-module $\DdagrigA (V_{\mathbf{f}}).$ More precisely,
$\DdagrigA (V_{\mathbf{f}})$ sits in an exact sequence
\[
0\lra \mathbb D_{\mathbf{f}} \lra \DdagrigA (V_{\mathbf{f}})
\lra \widetilde{\mathbb D}_{\mathbf{f}}\lra 0\,,
\]
where $\mathbb D_{\mathbf{f}}=\RR_A\cdot e_{\boldsymbol{\delta}}$ and
$\widetilde{\mathbb D}_{\mathbf{f}}=\RR_A\cdot e_{\widetilde{\boldsymbol{\delta}}},$ 
are $(\varphi,\Gamma)$-modules of rank $1$ defined by characters
$\boldsymbol{\delta}\,:\,\QQ_p^*\rightarrow A^*$ and 
$\widetilde{\boldsymbol{\delta}}\,:\,\QQ_p^*\rightarrow A^*$ such that   
\begin{align*}
&\boldsymbol{\delta} (u)=u^{k/2},\qquad \boldsymbol{\delta}(p)=p^{k/2}\alpha (w){\,,}\\
&\widetilde{\boldsymbol{\delta}} (u)=u^{k/2+1-\kappa (w)},\qquad 
\widetilde{\boldsymbol{\delta}}(p)=p^{-k/2}\alpha^{-1} (w).
\end{align*}
Note that $\psi_k(\boldsymbol{\delta})=\delta,$ 
$\psi_k(\widetilde{\boldsymbol{\delta}})=\widetilde\delta$ and
$\mathbb D_f=\mathbb D_{\mathbf{f}}\otimes_{A,\psi_k}E.$

\begin{thm}[Stevens, Coleman--Iovita]\label{thm:stevensderformula} Assume that 
\[
a_p=\mathbf{a}_p(k)=p^{k/2-1}.
\]
Then for the Fontaine-Mazur $\al$-invariant we  have 
 $$\mathcal{L}_{\textup{FM}}(f)=-2 p\cdot \alpha^\prime(k),$$
where $\alpha (\kappa)=p^{-k/2}\mathbf{a}_p(\kappa).$
\end{thm}
\begin{proof}Since  $\al_{\textup{C}}(f)=\al_{\textup{FM}}(f)$ by \cite{CI10},
the theorem follows from \cite{St10}, where  such a formula 
was  proved  for Coleman's $\al$-invariant $\al_{\textup{C}}(f).$ 
The first direct proof of this formula for Fontaine-Mazur's 
$\al$-invariant was discovered  by Colmez  \cite{Col10}. Another proof based on the theory of  $(\varphi,\Gamma)$-modules  may be found in \cite{Ben10}.
\end{proof}

 Let $\frak{p}=\ker(\psi_k)$ be the prime of $A$ corresponding to the eigenform $f$.
We call \emph{central critical weight-height pairing} $h_{\mathbf{f}}^{\textup{c-wt}}$  the pairing $h^{\mathrm{wt}}$  given via the general theory in Section~\ref{subsec:constructionofh^wt} for the family $V_{\mathbf{f}}(\boldsymbol{\chi}^{1/2})$ that is equipped with 
the triangulation $\mathbb D_{\mathbf{f}}(\boldsymbol{\chi}^{1/2}):$ 
\[
h_{\mathbf{f}}^{\textup{c-wt}}=h^{\textup{wt}}_{V_{\psi_k(\mathbf{f})},\mathbb D_{\psi_k(\mathbf{f})}}\,:\,
\widetilde H^1_{\textup{f}}(V_f)\times 
\widetilde H^1_{\textup{f}}(V_f^*(1)) \rightarrow \frak p/\frak p^2. 
\]
 Denote by 
\be\label{eqn:centralcriticalheightvaluesinpmodp2}
\mathbb H_{\mathbf{f}}^{\textup{c-wt}}\,:\,\widetilde H^1_{\textup{f}}(V_f)\times 
\widetilde H^1_{\textup{f}}(V_f) \rightarrow \frak p/\frak p^2
\ee
the composition of the pairing $h_{\mathbf{f}}^{\textup{c-wt}}$ with the isomorphism $H^1_{\textup{f}}(V_f)\simeq H^1_{\textup{f}}(V_f^*(1))$ that is induced by the skew-symmetric pairing $V_f\times V_f\rightarrow E(1)$.
Since $\left(V_{\mathbf{f}}(\boldsymbol{\chi}^{1/2}),\mathbb D_{\mathbf{f}} (\boldsymbol{\chi}^{1/2})\right)$ is self-dual and the pairing $h_{\mathbf{f}}^{\textup{c-wt}}$ is symmetric, we conclude that the pairing $\mathbb H_{\mathbf{f}}^{\textup{c-wt}}$ is skew-symmetric:
\begin{equation}
\label{weightheightisskewsymformodforms}
\mathbb H_{\mathbf{f}}^{\textup{c-wt}}\left([x_{\textup{f}}],[y_{\textup{f}}]\right)=-\mathbb H_{\mathbf{f}}^{\textup{c-wt}}\left(([y_{\textup{f}}],[x_{\textup{f}}]\right).
\end{equation}

\begin{define}\label{def:centralcriticalheight} We denote by 
\[
\frak h_{\mathbf{f}}^{\textup{c-wt}}\,:\,\widetilde H^1_{\textup{f}}(V_f)\times \widetilde H^1_{\textup{f}}(V_f) \rightarrow E
\]
the pairing defined via
\[
\mathcal A \left (\mathbb H_{\mathbf{f}}^{\textup{c-wt}}([x_{\textup{f}}],[y_{\textup{f}}])\right )=\frak h_{\mathbf{f}}^{\textup{c-wt}}([x_{\textup{f}}],[y_{\textup{f}}]) (\kappa-k).
\]
\end{define}


\subsection{Two-variable $p$-adic heights}
\label{subsec:twovariablesheights}
In this section, we construct  infinitesimal deformations of Nekov\'a\v r's $p$-adic  heights along the weight direction. Our construction is a direct generalization of Venerucci's  two-variable height pairing to the non-ordinary case. 
To do this, we replace the theory of \cite{Ne06} by its non-ordinary version developed  in \cite{Pot13} and \cite{Ben14b}.
We keep notation and conventions of Section~\ref{subsec:padicfamilies}. 

Let $\overline V_{\mathbf{f}}=V_{\mathbf{f}}\widehat\otimes_A\mathcal H_A^{\iota}$ and $\overline{\mathbb D}_{\mathbf{f}}=\mathbb D_{\mathbf{f}}\otimes_{\RR_A}\DdagrigA(\mathcal H_A^{\iota})$ (see Section~\ref{subsec:iwasawacohomology}).
Let $\frak{P}$ denote the kernel of the augmentation map 
\begin{align*}\mathcal H_A&\lra E\\
f(X)&\mapsto\psi_k (f(0))   
\end{align*}
Consider the tautological  exact sequence 
\[
0\rightarrow \frak{P}/\frak{P}^{2}\rightarrow \mathcal H_A/\frak{P}^2 \xrightarrow{} 
E\rightarrow 0.
\]
Tensoring this exact sequence with $\RG (\overline V_{\mathbf{f}},\overline{\mathbb D}_{\mathbf{f}})$ we get a distingushed triangle 
\[
\RG (\overline V_{\mathbf{f}},\overline{\mathbb D}_{\mathbf{f}})\otimes_{\mathcal H_A}^{\mathbb L}\frak{P}/\frak{P}^2
\rightarrow 
\RG (\overline V_{\mathbf{f}},\overline{\mathbb D}_{\mathbf{f}})\otimes_{\mathcal H_A}^{\mathbb L}\mathcal H_A/\frak{P}^2
\rightarrow 
\RG (\overline V_{\mathbf{f}},\overline{\mathbb D}_{\mathbf{f}})\otimes_{\mathcal H_A}^{\mathbb L}E.
\]
By the base change theorem for Selmer complexes (c.f., \cite{Pot13}, Section 1), we have a canonical isomorphism 
\begin{equation}
\label{eqn:base change Selmer}
\RG (\overline V_{\mathbf{f}},\overline{\mathbb D}_{\mathbf{f}})\otimes_{\mathcal H_A}^{\mathbb L}E \stackrel{\sim}{\lra}
\RG (V_f,\mathbb D_f).
\end{equation}
 Using the natural identification
\begin{equation}
\label{eqn: iso for Selmer}
\RG (\overline V_{\mathbf{f}},\overline{\mathbb D}_{\mathbf{f}})\otimes^{\mathbb L}_{\mathcal H_A}\mathfrak{P}/\mathfrak{P}^2\simeq 
\RG (V_f,\mathbb D_f)\otimes_E\mathfrak{P}/\mathfrak{P}^2,
\end{equation}
this distinguished triangle translates to
\[
\RG (V_f,\mathbb D_f)\otimes_E\mathfrak{P}/\mathfrak{P}^2
 \rightarrow 
\RG (\overline V_{\mathbf{f}} ,\overline{\mathbb D}_{\mathbf{f}})\otimes_{\mathcal H_A}^{\mathbb L}\mathcal H_A/\frak{P}^2 \rightarrow 
\RG (V_f,\mathbb D_f)
\xrightarrow{\beta} 
\RG (V_f,\mathbb D_f)[1]\otimes_E\mathfrak{P}/\mathfrak{P}^2.
\]
On the level of cohomology, we have a map
\[
\widetilde H^1_{\textup{f}}(V_f)\rightarrow \widetilde H^2_{\textup{f}}(V_f)\otimes_E\mathfrak{P}/\mathfrak{P}^2
\]
which we denote again by $\beta.$

\begin{define}
\label{def:twovarheight}
The \emph{two-variable height pairing} $\mathbb{H}_{\mathbf{f}}$ is defined as the compositum of the following arrows:
$$
\widetilde{H}^1_{\textup{f}}(V_f)\otimes_E \widetilde{H}^1_{\textup{f}}(V_f) \stackrel{\beta\otimes j}{\lra} \left(\widetilde{H}^2_{\textup{f}}(V_f)\otimes_E \frak{P}/\frak{P}^2\right)\otimes_E \widetilde{H}^1_{\textup{f}}(V_f^*(1))\stackrel{\cup_{V_f,\mathbb D_f}}{\lra}  \frak{P}/\frak{P}^2\,$$
where $\cup_{V_f,\mathbb D_f}$ is the  cup-product (\ref{cupproductforSelmercomplexes}).
\end{define}

Let $\frak I$ denote the ideal of the ring $E[[\kappa-k,s]]$ generated by
$\kappa-k$ and $s.$  The map $\mathcal A$ identifies  $\frak P/\frak{P}^2$ with 
$\frak I/\frak I^2.$ 

The main formal properties of the pairing $\mathbb H_{\mathbf{f}}$ are listed in the following theorem.   

{\begin{thm}
\label{thm:thepropertiesofthetwovarheight} For any $[x_{\textup{f}}], [y_{\textup{f}}] \in \widetilde{H}^1_{\textup{f}}(V_f)$, we have the following identities in $\frak I/\frak I^2$.\\\\
\textup{i)} $\frac{\partial}{\partial s} \mathcal A\left(\mathbb{H}_{\mathbf{f}}\left([x_{\textup{f}}], [y_{\textup{f}}]\right)\right)\big{|}_{s=0,\kappa=k}=-\frak{h}_p\left([x_{\textup{f}}], [y_{\textup{f}}]\right)$\, and

 $\frac{\partial}{\partial \kappa} \mathcal A\left(\mathbb{H}_{\mathbf{f}}\left([x_{\textup{f}}], [y_{\textup{f}}]\right)\right) (\kappa,(\kappa-k)/2)\big{|}_{\kappa=k}=\frak{h}_{\mathbf {f}}^{\textup{c-wt}}\left([x_{\textup{f}}], [y_{\textup{f}}]\right)\,.$
\\\\
\textup{ii)}  $\mathcal A\left(\mathbb{H}_{\mathbf{f}}\left([y_{\textup{f}}],[x_{\textup{f}}]\right)\right)(\kappa,s)=-\mathcal A\left(\mathbb{H}_{\mathbf{f}}\left([x_{\textup{f}}], [y_{\textup{f}}]\right)\right)(\kappa,\kappa-k-s)$\,.\\\\
\textup{iii)} $(${Rubin-style formulae, Part I}$)$ Assume that $a_p=p^{k/2-1}.$
Let  $d_{\widetilde{\delta}}=t^{k/2-1}e_{\widetilde{\delta}}$ denote the generator
of $H^0(\widetilde{\mathbb D}_f)$ defined in Proposition~\ref{prop:ex sequence with Dalpha}. Let $[x_{\textup{f}}]
\in  \widetilde{H}^1_{\textup{f}}(V_f)$ denote the lift of a class $[x] \in H^1_{\textup{f}}(\QQ, V_f)$ with respect to the canonical splitting $\textup{spl}$. Suppose that the class $[x_{\textup{f}}]$ is represented by  $\left(x,(x_{\ell}^+),(\lambda_{\ell})\right)$. We then have the following identity in $\frak{I}/\frak{I}^2$\,:

\be\label{eqn:Rubin2}\mathcal A\left(\mathbb{H}_{\mathbf{f}}\left([x_{\textup{f}}], \partial_0 (d_{\widetilde{\delta}})\right)\right)=\left\langle\Psi_2,[x_p^+]\right\rangle\cdot s\,.\ee
Furthermore,
\be\label{eqn:Rubin1}\mathcal A\left(\mathbb{H}_{\mathbf{f}}\left(\partial_0(d_{\widetilde\delta}), \partial_0(d_{\widetilde\delta})\right)\right)= \left\langle\Psi_2,\partial^{\textup{loc}}_0(d_{\widetilde\delta}) \right\rangle\cdot \left(s-\frac{\kappa-k}{2}\right)\,. \ee

\end{thm}
\begin{proof} i) These two identities follow directly from definitions, see in particular \cite[Section 0.22 in Appendix C]{veneruccithesis} for a general formalism in the $p$-ordinary setting. We provide some details below. The evaluation at $k$ map  $\psi_k$  gives rise to a commutative diagram
\[
\xymatrix{
0 \ar[r] & \mathfrak P/\mathfrak P^2 \ar[d]\ar[r] & \mathcal H_A/
\mathfrak P^2 \ar[d] \ar[r] & E \ar @{=}[d] \ar[r]  &0\\
0 \ar[r] &J_E/J_E^2 \ar[r] &\mathcal H_E/J_E^2 \ar[r] & E \ar[r]
& 0.
}
\]
Tensoring the upper row with $\RG (\overline{V}_{\mathbf f},\overline{\mathbb D}_{\mathbf f})$, the bottom row with $\RG(\overline{V}_{f},\overline{\mathbb D}_{f})$ and relying on the isomorphisms \eqref{eqn:base change Selmer} and \eqref{eqn: iso for Selmer}, we obtain the following commutative diagram, where the vertical map in the middle is induced by the  specialization  map $V_{\mathbf f}\rightarrow V_f$:
\[
\xymatrix{
 \RG ({V}_{f},{\mathbb D}_{f})\otimes_{E}\mathfrak P/\mathfrak P^2 \ar[d]\ar[r] &  \RG (\overline{V}_{\mathbf f},\overline{\mathbb D}_{\mathbf f})\otimes_{\mathcal H_A}^{\mathbb L}\mathcal H_A/
\mathfrak P^2 \ar[d] \ar[r] &  \RG ({V}_{f},{\mathbb D}_{f})  \ar @{=}[d] 
\\
 \RG ({V}_{f}, {\mathbb D}_{f})\otimes_{E} J_E/J_E^2 \ar[r] & \RG (\overline{V}_{f},\overline{\mathbb D}_{f})\otimes_{\mathcal H_E}^{\mathbb L}\mathcal H_E/J_E^2 \ar[r] & 
 \RG ({V}_{f}, {\mathbb D}_{f})
}
\]
The rows of this diagram are exact triangles, and therefore we have 
a commutative diagram
\[
\xymatrix{
\widetilde H^1_{\textup{f}}(V_f)\ar[r]^-{\beta} 
\ar[dr]_{\beta^{\mathrm{cyc}}}
& \widetilde H^2_{\textup{f}}(V_f)\otimes_E\mathfrak{P}/\mathfrak{P}^2
\ar[d]^{\psi_k}\\
& \widetilde H^2_{\textup{f}}(V_f)\otimes_E J_E/J_E^2,
}
\]
where $\beta^{\mathrm{cyc}}$ is the connecting morphism given as in \eqref{eqn: connecting cyclotomic beta}. By the functoriality of cup products, it follows that the diagram 
\[
\xymatrix{
\widetilde H^1_{\textup{f}}(V_f) \times \widetilde H^1_{\textup{f}}(V_f)
\ar[r]^-{\mathbb H_{\mathbf f}} 
\ar[dr]_{h_{V_f,D}}
&\mathfrak{P}/\mathfrak{P}^2 \ar[d]^{\psi_k}\\
& J_E/J_E^2
}
\]
also commutes. This proves the first formula.
We remark that  the sign in the first identity is due to our normalization of the cyclotomic Amice transform; c.f. Remark~\ref{rem:cyclopairingheight}. 

We now prove the second formula. The assignment $\gamma \mapsto \boldsymbol{\chi}^{1/2}(\gamma^{-1})$ induces  a canonical projection $\mathfrak P/\mathfrak P^2 \rightarrow \mathfrak p/\mathfrak p^2.$ Analogous arguments to those above show that the diagram 
\[
\xymatrix{
 \RG ({V}_{f},{\mathbb D}_{f})\otimes_{E}\mathfrak P/\mathfrak P^2 \ar[d]\ar[r] &  \RG (\overline{V}_{\mathbf f},\overline{\mathbb D}_{\mathbf f})\otimes_{\mathcal H_A}^{\mathbb L}\mathcal H_A/
\mathfrak P^2 \ar[d] \ar[r] &  \RG ({V}_{f},{\mathbb D}_{f})  \ar @{=}[d]
\\
 \RG ({V}_{f}, {\mathbb D}_{f})\otimes_{E} \mathfrak p/{\mathfrak p}^2 \ar[r] & \RG ({V}_{\mathbf f}(\boldsymbol{\chi}^{1/2}),{\mathbb D}_{\mathbf f}(\boldsymbol{\chi}^{1/2}))\otimes_{A}^{\mathbb L} A/{\mathfrak p}^2 \ar[r] & 
 \RG ({V}_{f}, {\mathbb D}_{f}),
}
\]
 where the vertical map in the middle is induced by the projection $\overline{V}_{\mathbf f} \rightarrow V_{\mathbf f}(\boldsymbol{\chi}^{1/2})$ (c.f. Remark~\ref{rem:reconstructtheselfdualHidafamily}),  commutes. Using Proposition~\ref{prop:comparisiontwodefinitionsofweightheight} and mimicking the arguments used in the proof of the first formula, we
obtain a commutative diagram
\[
\xymatrix{
\widetilde H^1_{\textup{f}}(V_f) \times \widetilde H^1_{\textup{f}}(V_f)
\ar[r]^-{\mathbb H_{\mathbf f}} 
\ar[dr]_{{\mathbb H}^{\textup{c-wt}}_{\mathbf f}}
&\mathfrak{P}/\mathfrak{P}^2 \ar[d]\\
& \mathfrak{p}/\mathfrak{p}^2.
}
\]
It follows directly from definitions that the vertical map, after applying the two-variable Amice transform, is the map given by $s\mapsto (\kappa -k)/2$. This proves the second formula. 

ii) 
Let us write $(\kappa,s)=(\kappa, (\kappa-k)/2)+ (0,s-(\kappa-k)/2)$ and
use the formulae in Part (i), so that we have
\[
\mathcal A\left(\mathbb{H}_{\mathbf{f}}\left([y_{\textup{f}}],[x_{\textup{f}}]\right)\right)(\kappa,s)=
\mathfrak{h}_{\mathbf{f}}^{\textup{c-wt}}\left([y_{\textup{f}}],[x_{\textup{f}}]\right) \cdot (\kappa-k)- 
\frak h_p([y_{\textup{f}}],[x_{\textup{f}}])\cdot \left (s-\frac{\kappa-k}{2}\right )\mod \frak{I}^2.
\]
On the other hand, if we write $(\kappa,\kappa-k-s)=(\kappa, (\kappa-k)/2)+
(0, (\kappa-k)/2-s)$, we deduce using the same formulae that
\begin{align*}
    \mathcal A\left(\mathbb{H}_{\mathbf{f}}\left([x_{\textup{f}}],[y_{\textup{f}}]\right)\right)(\kappa,\kappa-k-s)&=
\mathfrak{h}_{\mathbf{f}}^{\textup{c-wt}}\left([x_{\textup{f}}],[y_{\textup{f}}]\right) \cdot (\kappa-k)\\
 &\,\qquad\qquad\qquad + \frak h_p([x_{\textup{f}}],[y_{\textup{f}}])\cdot \left (s-\frac{\kappa-k}{2}\right )\mod \frak{I}^2.
\end{align*}
Since  $\frak{h}_p$ is symmetric and $\mathfrak{h}_{\mathbf{f}}^{\textup{c-wt}}$ is anti-symmetric, these formulae together show that 
\[
\mathcal A\left(\mathbb{H}_{\mathbf{f}}\left([y_{\textup{f}}],[x_{\textup{f}}]\right)\right)(\kappa,s)=\mathcal A\left(\mathbb{H}_{\mathbf{f}}\left([x_{\textup{f}}],[y_{\textup{f}}]\right)\right)(\kappa,\kappa-k-s),
\]
as we claimed.

iii) Let $ d_{\widetilde\delta}:=t^{k/2-1}e_{\widetilde\delta}$ 
and  $d_{\widetilde{\boldsymbol\delta}}=t^{k/2-1}e_{\widetilde{\boldsymbol\delta}}
\in {\widetilde{\mathbb D}_{\mathbf f}}^{\Delta}.$ 
Choose any lift $\mathbf z\in \DdagrigA (V_{\mathbf f})^{\Delta}$ of $d_{\widetilde{\boldsymbol\delta}}$ under the 
canonical projection of $\DdagrigA (V_{\mathbf f})$ onto 
$\widetilde{\mathbb D}_{\mathbf f}.$ Then  $z=\psi_k(\mathbf z)\in \Ddagrig (V_f)^{\Delta}$ is a lift of  $d_{\widetilde{{\delta}}}$ under the 
projection of $\Ddagrig (V_f)^{\Delta}$ onto $\widetilde{\mathbb D}_f^{\Delta}$
and by (\ref{eqn:formulaforcoboundary}) the class $\partial_0(d_{\widetilde\delta})\in \widetilde H^1_{\textup{f}}(V_f)$ may be represented by the cocycle
\[
(0, (a^+_{\ell}),(\mu_{\ell})) \in C^1(G_{\QQ,S},V_{f})\oplus U^+_S(V_{f},\mathbb D_f)^1\oplus K^0(V_{f}),
\]
where   $a_{\ell}^+=\mu_{\ell}=0$ for all $\ell\neq p,$ and 
\[
\mu_p=\alpha (z),\qquad a_p^+=((\varphi-1)z, (\gamma-1)z).
\]
Let $\mathbf z\in C^0_{\varphi,\gamma}(V_{\mathbf{f}})=\DdagrigA (V_{\mathbf{f}})$ be any lift of $z$ under the projection induced by the augmentation map $\psi_k: A\rightarrow E$, whose kernel is the prime $\frak{p}$ associated to the form $f$. Then, 
$$\mathbf z\otimes 1\in \overline{\DdagrigA (V_{\mathbf{f}})}:=\DdagrigA (V_{\mathbf{f}})\otimes_{\RR_A}
\DdagrigA (\mathcal H_A^{\iota})$$ 
is a lift of $d_{\widetilde{\boldsymbol\delta}}\otimes 1$ under the projection 
$$\overline{\DdagrigA (V_{\mathbf{f}})} \lra \widetilde{\mathbb D}_{\mathbf{f}}\otimes_{\RR_A}\DdagrigA (\mathcal H_A^{\iota}).$$ 
Setting $\mathbf a_{\ell}^+=\boldsymbol{\mu}_{\ell}=0$ for all $\ell\neq p$ and 
\[
\boldsymbol \mu_p=(\mathbf z\otimes 1),\qquad \mathbf a_p^+=((\varphi-1)(\mathbf z
\otimes 1), (\gamma-1)(\mathbf z\otimes 1))
\]
we see that $(0, (\mathbf a^+_{\ell}),(\boldsymbol \mu_{\ell}))$ is a cochain in
$S^1(\overline V_{\mathbf f},\overline{{\mathbb D}}_\mathbf{f})$ which  lifts
$\partial_0(d_{\widetilde \delta}).$ We therefore infer that $\beta (\partial_0(d_{\widetilde \delta}))$
is the class of the differential
\[
d(0, (\mathbf a^+_{\ell}),(\boldsymbol \mu_{\ell}))=(0,\star,(\mathbf v_{\ell})),
\]
where $\mathbf v_{\ell}=0$ for all $\ell\neq p$ and $\mathbf{v}_p=-((\varphi-1)(\mathbf{z}\otimes 1),(\gamma-1)(\mathbf{z}\otimes 1)).$ 
(Note that the differential of the Selmer complex differs by the sign $-1$
from the differential of the corresponding cone. This explains the sign in the formula above.)
Let $[x_{\textup{f}}]=[x,(x_{\ell}^+),(\lambda_{\ell})]\in H^1_{\textup{f}}( V_f).$ Then by the definition of the cup-product  given by (\ref{formulaforcupproductofSelmercomplexes}),
\be\label{eqn:proofth42}
\mathbb H_{\mathbf{f}}(\partial_0(d_{\widetilde \delta}),[x_{\textup{f}}])=
\textup{inv}_p(\overline{\mathbf v}_p\cup i_p^+(x_p^+))=\textup{inv}_p(\widetilde{\mathbf v}_p\cup x_p^+),
\ee 
where $\overline{\mathbf v}_p\in C^1_{\varphi,\gamma}(\Ddagrig (V_f))\otimes_E \mathfrak P/\mathfrak P^2\subset C^1_{\varphi,\gamma}(\overline{\DdagrigA(V_{\mathbf f})})\otimes_{\mathcal H_A} {\mathcal H}_A/\mathfrak P^2$ denotes the reduction of $\mathbf v_p$ modulo 
$\mathfrak P^2,$ $\widetilde{\mathbf v}_p$ denotes the image of $\overline{\mathbf v}_p$ in
$C^1_{\varphi,\gamma}(\widetilde{\mathbb D}_{f})\otimes_E \frak P/\frak P^2\subset
\linebreak
C^1_{\varphi,\gamma}(\widetilde{\mathbb D}_{\mathbf{f}}\widehat\otimes_{\RR_A}\DdagrigA (\mathcal H_A^{\iota})/\mathfrak P^2)$ under the natural projection
and $\textup{inv}_p\,:\,H^2(\mathcal R_E (\chi)) \simeq E$ is the isomorphism of the local class field theory. 

The element $\widetilde{\mathbf v}_p$ is explicitly given by the formula
\[
\widetilde{\mathbf v}_p=-((\varphi-1)(t^{k/2-1} e_{\widetilde{\boldsymbol\delta}} \otimes 1),(\gamma-1)(t^{k/2-1} e_{\widetilde{\boldsymbol\delta}}  \otimes 1)) \pmod{\frak P^2}.
\]
Since $\varphi (e_{\widetilde{\boldsymbol\delta}})=p^{-k/2}\alpha^{-1}(\kappa) e_{\widetilde{\boldsymbol\delta}}$\, and 
\[
\frac{1}{\alpha (\kappa)}\equiv\frac{1}{\alpha (k)}-\frac{\alpha'(k)}{\alpha^2 (k)}\cdot(\kappa-k)
\pmod{\frak I^2},
\]
 we have 
 \[
\mathcal A \left ((\varphi-1)(t^{k/2-1}e_{\widetilde{\boldsymbol\delta}} \otimes 1)
\right )
\equiv-p\alpha'(k)\cdot(\kappa-k)\cdot d_{\widetilde\delta}
\pmod{\frak I^2}.
\]
Taking into account the statement of Theorem~\ref{thm:stevensderformula}, we deduce that 
\[
\mathcal A \left ((\varphi-1)(t^{k/2-1} e_{\widetilde{\boldsymbol\delta}} \otimes 1)
\right ) \equiv \frac{\al_{\textup{FM}}(f)}{2}\cdot(\kappa-k)\cdot d_{\widetilde{\delta}}\pmod{\frak I^2}.
\]
 Recall that we have set $\displaystyle\kappa (w)=k+\frac{\log (1+w)}{\log (1+p)}.$
Since $\gamma (e_{\widetilde{\boldsymbol\delta}})=\chi (\gamma)^{k/2+1-\kappa (w)} e_{\widetilde{\boldsymbol\delta}},$ we have 
\begin{align}
\begin{aligned}
\label{aligned_eqn_15_02_2022_16_17}
\mathcal A \left ((\gamma-1)(t^{k/2-1} e_{\widetilde{\boldsymbol\delta}} \otimes 1)
\right)&= \mathcal A \left (\chi (\gamma)^{k-\kappa (w)}\cdot t^{k/2-1} e_{\widetilde{\boldsymbol\delta}} \otimes \gamma^{-1}-
t^{k/2-1} e_{\widetilde{\boldsymbol\delta}}  \otimes 1\right)\\
&=\mathcal A \left (\chi (\gamma)^{k-\kappa (w)}\cdot 
t^{k/2-1} e_{\widetilde{\boldsymbol\delta}}  \otimes (\gamma^{-1}-1)\right)\\
&\qquad\qquad\qquad+\mathcal A \left ((\chi (\gamma)^{k-\kappa (w)}-1) \cdot t^{k/2-1} e_{\widetilde{\boldsymbol\delta}}  \otimes 1\right)\\
&\equiv \left(\log\chi(\gamma)\cdot(k-\kappa)+1\right)\cdot \log\chi(\gamma)\cdot s\cdot t^{k/2-1}  e_{\widetilde{\boldsymbol\delta}} \otimes 1
\\
&\qquad\qquad\qquad\qquad\quad+\log \chi (\gamma)(k-\kappa)\cdot t^{k/2-1}  e_{\widetilde{\boldsymbol\delta}}  \otimes 1 \\
&\equiv s\log \chi (\gamma) \cdot( t^{k/2-1}e_{\widetilde{\boldsymbol\delta}} \otimes 1)\\
&\qquad\qquad\qquad\quad\,\,\,+\log \chi (\gamma)\cdot(k-\kappa)\cdot ( t^{k/2-1}e_{\widetilde{\boldsymbol\delta}} \otimes 1) \\
&\equiv (s+k-\kappa) \log \chi (\gamma) \cdot( t^{k/2-1}e_{\widetilde{\boldsymbol\delta}}  \otimes 1)\\
&\equiv(s+k-\kappa) \log \chi (\gamma) \cdot d_{\widetilde{\delta}} \qquad \qquad \qquad \,\,\pmod{\frak I^2}\,,
\end{aligned}
\end{align}
 where congruence on the third line of \eqref{aligned_eqn_15_02_2022_16_17} follows from the definition of the
 Amice transform, namely  
 \begin{align*}
     \mathcal{A}^{\rm wt}(\chi (\gamma)^{k-\kappa (w)})= \chi (\gamma)^{k-\kappa} =1+\log\chi(\gamma)\cdot (k-\kappa) + O((\kappa-k)^2)
 \end{align*} 
  \begin{align*}
      \mathcal{A}^{\rm cyc}(\gamma^{-1}-1)=\chi(\gamma)^s-1=\log\chi(\gamma)\cdot s+O(s^2)\,.
  \end{align*}  

To summarize, we have just verified that
\[
\mathcal A ( {\widetilde{\mathbf v}}_p)=  
\left (\frac{\al_{\textup{FM}}(f)}{2}\cdot(\kappa-k),
\log \chi (\gamma)\cdot(s+k-\kappa) \right )  d_{\widetilde{\delta}}
\]
and for the class of  $\widetilde{\mathbf v}_p$ in $H^1(\widetilde{\mathbb D}_f)\otimes_E\frak P/\frak P^2$ we have 
\[
\mathcal A \left ( \left [\widetilde{\mathbf v}_p\right ]\right )=
-\frac{\al_{\textup{FM}}(f)}{2}(\kappa-k)\Psi_1+(s+k-\kappa) \Psi_2\pmod{\frak I^2}.
\]

We are now ready to prove our formulas. First assume that $[x_{\textup{f}}]=[x,(x_{\ell}^+),(\lambda_{\ell})]$
is the canonical lift of some $[x]\in H^1_{\textup{f}}(\QQ, V_f).$ Then $[x_p^+]\in H^1_\textup{f}(\mathbb D_f)$ 
and using the formula (\ref{eqn:proofth42})   we have
\be\label{eqn_2022_02_27_16_56}
\mathcal A\left(\mathbb{H}_{\mathbf{f}}\left(\partial_0 (d_{\widetilde\delta}),[x_{\textup{f}}]\right)\right )=
\left <\Psi_2,[x_p^+]\right> \cdot(s+k-\kappa) .
\ee
Therefore,
\[
\mathcal A\left(\mathbb{H}_{\mathbf{f}}\left([x_{\textup{f}}], \partial_0 (d_{\widetilde\delta})\right)\right )=\left <\Psi_2,[x_p^+]\right >\cdot s.
\]
We now prove the second formula. Write $\partial^{\textup{loc}}_0(d_{\widetilde\delta})=a\Psi_1^*+b\Psi_2^*.$
Then $\al_{\textup{FM}}(f)=b/a$ by \cite[Proposition~2.3.7]{Ben11} and 
\begin{multline*}
\mathcal A\left(\mathbb{H}_{\mathbf{f}}\left(\partial_0 (d_{\widetilde\delta}), \partial_0 (d_{\widetilde\delta})\right)\right )=a\cdot\frac{\al_{\textup{FM}}(f)}{2}\cdot(\kappa-k)+b\cdot(s+k-\kappa)
\\
=b\cdot\left (s-\frac{1}{2} (\kappa-k)\right )=\left <\Psi_2,\partial^{\textup{loc}}_0 (d_{\widetilde\delta})\right >\cdot\left (s-\frac{1}{2} (\kappa-k)\right ).
\end{multline*}
\end{proof}
We remark that the identities in Theorem~\ref{thm:thepropertiesofthetwovarheight}(i) are equivalent to saying that 
\be\label{eqn:twovarexpanded1}\mathcal A\left(\mathbb{H}_{\mathbf{f}}\left([x_{\textup{f}}], [y_{\textup{f}}]\right)\right)(\kappa,s)=-\frak{h}_p\left([x_{\textup{f}}], [y_{\textup{f}}]\right)\cdot\left(s-\frac{\kappa-k}{2}\right)+\frak{h}_{\mathbf {f}}^{\textup{c-wt}}\left([x_{\textup{f}}], [y_{\textup{f}}]\right)\cdot\left(\kappa-k\right)\ee
\subsection{Rubin-style Formulae, Part II}
\label{subsec:RS2}
Throughout this section, we shall be in one of the following two settings (that we will simultaneously treat):
\begin{itemize}
       \item[\mylabel{item_A}{(\textbf{A})}] 
 We have $\frak{B}=\mathcal{H}_E$ and $\frak{J}=J_E,$ where $J_E=X\cdot \mathcal{H}_E$ is the kernel of the augmentation map and $X=\gamma-1.$ We set $V=\overline V_f=V_f\otimes \mathcal{H}_E^\iota$ and $\mathbb D=\overline{\mathbb D}_f=\mathbb D_f \otimes_{\mathcal{R}} \mathbb{D}^\dagger_{\textup{rig}} (\mathcal{H}_E^\iota).$ The $(\varphi,\Gamma)$-module  $\mathbb D$  is a submodule of the $(\varphi,\Gamma)$-module $\DdagrigfrakB (V)$ associated $V$ and we set 
$\widetilde{\mathbb D}=\DdagrigfrakB (V)/\mathbb D$ (and it also equals $\widetilde{\mathbb D}_f\otimes  \mathbb{D}^\dagger_{\textup{rig}} (\mathcal{H}_E^\iota)$).
Observe that we have $V/\frak{J}V \cong V_f$ (and a similar isomorphism for all relevant $(\varphi,\Gamma)$-modules). Furthermore, we have the symmetric $p$-adic height pairing
$$\frak{H}\,:\, \widetilde{H}^1_{\textup{f}}(V_f) \times \widetilde{H}^1_{\textup{f}}(V_f) \lra \frak{J}/\frak{J}^2$$
induced from the pairing (\ref{eqn:padicpairingforfvaluesinaugmentation}) via the isomorphism $V_f\cong V_f^*(1)$\,.
\\
\item[\mylabel{item_B}{(\textbf{B})}] We have $\frak{B}=A$ where $A$ is an affinoid domain as in Section~\ref{subsec:padicfamilies}. In this set up, we will consider the Galois representation $V:=V_{\mathbf{f}}(\boldsymbol{\chi}^{1/2})$, which is the central critical twist of the big Galois representation associated to the family $\mathbf{f}$. This representation is equipped with a triangulation, which is given by the $(\varphi,\Gamma)$-submodule $\mathbb{D}:=\mathbb D_{\mathbf{f}}(\boldsymbol{\chi}^{1/2}).$ As above, we will also let $\widetilde{\mathbb D}:=\mathbb{D}_{\textup{rig},A}^\dagger(V)/\mathbb D$. The prime $\frak{p}$ is that corresponds to the form $f$ we have fixed at the start and we set $\frak{J}=
\frak{p}$. Observe that again  $V/\frak{J}V\cong V_f$ (and a similar isomorphism for all relevant $(\varphi,\Gamma)$-modules). Furthermore, we have an \emph{anti-symmetric} $p$-adic height pairing
$$\frak{H}:= \mathbb{H}_\mathbf{f}^{\textup{c-wt}}: \widetilde{H}^1_{\textup{f}}(V_f) \times \widetilde{H}^1_{\textup{f}}(V_f) \lra \frak{p}/\frak{p}^2$$
which is given by (\ref{eqn:centralcriticalheightvaluesinpmodp2}).
\end{itemize}
\begin{define} In the situation of \ref{item_A} or \ref{item_B}
set $V_\varepsilon=V\otimes\frak{B}/\frak{J}^2,$   $\mathbb{D}_{\varepsilon}:=\mathbb{D}\otimes\frak{B}/\frak{J}^2$ and $\widetilde{\mathbb{D}}_{\varepsilon}=\widetilde{\mathbb{D}}\otimes\frak{B}/\frak{J}^2$.
\end{define}

Note that these objects also make their appearance in Sections 2.6.1 and 3.2.1 of \cite{Ben14b}, where $V_\varepsilon$ is denoted by $\widetilde{V}$, etc.

In the situation of \ref{item_A} or \ref{item_B}, note that we have the following exact sequence
\be\label{ean:infinitesimaldeformVf}
0\lra V_f\otimes\frak{J}/\frak{J}^2 \lra V_\varepsilon\lra V_f\lra 0\,\ee
as well as the following tautological exact triangle in the derived category:
\be
\label{eqn:localsequencetilde}
\RG(G_p,\mathbb{D}_{?})\lra \RG(G_p,\mathbb{D}^\dagger_{\textup{rig}}(V_{?}))\stackrel{\frak{s}}{\lra} \RG(G_p,\widetilde{\mathbb{D}}_{?})\lra\RG(G_p,\mathbb{D}_{?})[1]\,
\ee
where $?={f},\varepsilon$ and for a $(\varphi,\Gamma)$-module $D$, we denote by $\RG(G_p,D)$ the image of $C_{\varphi,\gamma}(D)$ in the derived category.
The quasi-isomorphism $\alpha$ of Proposition~\ref{proposition quasi-isomorphisms} together with the sequence (\ref{eqn:localsequence}) and (\ref{eqn:localsequencetilde}) induces a functorial quasi-isomorphism 
\be\label{eqn:identifyingsungularquotients_2}\RG(G_p,\widetilde{\mathbb{D}}_{?})\stackrel{\rm qis}{\lra}  \widetilde{U}_p(V,\mathbb{D}_{?})\,\ee
which is compatible with all duality statements in the obvious sense.
By the local-global compatibility of the Langlands correspondence (see \cite{Car86})
it follows that 
\be
\label{eqn:Hfiszeroatl}
H^0(G_{\ell},V_f)=H^1_{\textup{f}}(G_{\ell},V_f)=0,\qquad \ell\neq p.
\ee
In particular,
\[
\RG(G_{\ell}, V_{?})\stackrel{\rm qis}{\lra}  \widetilde{U}_{\ell}(V_{?},\mathbb{D}_{?}), \qquad \ell\neq p.
\]
To simplify notation, we set $\widetilde{U}_S(V_{?})=\underset{\ell\in S}\oplus
 \widetilde{U}_{\ell}(V_{?},\mathbb{D}_{?}).$

 The exact sequence (\ref{ean:infinitesimaldeformVf}) induces a commutative diagram of complexes
\be\label{eqn:deformedtautseq}\xymatrix{
& 0\ar[d]&0\ar[d]& 0\ar[d]&\\
0\ar[r]&\widetilde{U}_S(V_f)[-1] \otimes \frak{J}/\frak{J}^2 \ar[d]\ar[r]& {S}^\bullet(V_f)\otimes \frak{J}/\frak{J}^2 \ar[d]\ar[r]&{C}^\bullet(G_{\QQ,S},V_f)\otimes \frak{J}/\frak{J}^2  \ar[r]\ar[d]&  0\\
0\ar[r]&\widetilde{U}_S(V_\varepsilon)[-1]\ar[d] \ar[r]&  {S}^\bullet(V_\varepsilon) \ar[d]\ar[r]&{C}^\bullet(G_{\QQ,S},V_\varepsilon) \ar[d]\ar[r]& 0\\
0\ar[r]&\widetilde{U}_S(V_f)[-1] \ar[d]\ar[r]& {S}^\bullet(V_f)\ar[d] \ar[r]&{C}^\bullet(G_{\QQ,S},V_f) \ar[d]\ar[r]& 0\\
& 0&0& 0& 
}\ee
where we recall that ${S}^\bullet(V_{?})$ is the shorthand for the Selmer complex ${S}^\bullet(V_{?},\mathbb{D}_{?})$ (for $?=f,\varepsilon$).

In the level of cohomology, the diagram (\ref{eqn:deformedtautseq}) induces  the following commutative diagram with exact rows:
$$\scalebox{.95}{\xymatrix@R=2pc @C=1.3pc {&&&&H^0(\widetilde{U}_S(V_f))\ar[d]^(.45){\beta^0}\\
&&&&H^1(\widetilde{U}_S(V_f))\otimes\frak{J}/\frak{J}^2\ar[d]^(.45){i}\\
&&&H^1(G_{\QQ,S},V_\varepsilon)\ar[d]^{\textup{pr}_0}\ar[r]^(.5){\widetilde{\textup{res}}_S}&H^1(\widetilde{U}_S(V_\varepsilon))\ar[d]^{\textup{pr}_0}\\
&H^0(\widetilde{U}_S(V_f))\ar[d]_{\beta^0}\ar[r]&\widetilde{H}^1_{\textup{f}}(V_f)\ar[d]^{\beta^1}\ar[r]&H^1(G_{\QQ,S},V_f)\ar[d]^{\beta^1}\ar[r]_(.5){\widetilde{\textup{res}}_S}&H^1(\widetilde{U}_S(V_f))\ar[d]\\
&H^1(\widetilde{U}_S(V_f))\otimes\frak{J}/\frak{J}^2\ar[r]^(.55){\partial_1}&\widetilde{H}^2_{\textup{f}}(V_f)\otimes\frak{J}/\frak{J}^2\ar[r]&H^2(G_{\QQ,S},V_f)\otimes\frak{J}/\frak{J}^2\ar[r]&H^2(\widetilde{U}_S(V_f))\otimes\frak{J}/\frak{J}^2
}}$$
where $\textup{pr}_0$ is the map induced from the exact sequence 
\eqref{ean:infinitesimaldeformVf}.

\begin{lemma}
\label{lem:bocksteinnormalizedderivative}
Suppose that we are given a class $[x_{\textup{f}}]=[(x,(x_\ell^+)_{\ell \in S},(\lambda_\ell)_{\ell \in S})] \in \widetilde{H}^1_{\textup{f}}(V_f)$ such that $\textup{pr}_0\left([\mathbb{X}]\right)=[x] \in H^1_{\textup{f}}(\QQ,V_f)$ for some $[\mathbb{X}]\in H^1(G_{\QQ,S},V_\varepsilon)$.
Then there exists a class 
$$[\frak{D}\mathbb{X}]=\left([\frak{D}\mathbb{X}_\ell]\right)_{\ell \in S} \in H^1(\widetilde{U}_S(V_f))\otimes\frak{J}/\frak{J}^2$$
 such that 
\begin{itemize}
 \item[\mylabel{item_i}{(\textbf{i})}] $i([\frak{D}\mathbb{X}])=\widetilde{\textup{res}}_S\left([\mathbb{X}]\right)\,.$\\
 \item[\mylabel{item_ii}{(\textbf{ii})}] $\beta^1([x_{\textup{f}}])=-\partial_1([\frak{D}\mathbb{X}])$\,.
\end{itemize}
Moreover, if $\mathcal{L}_{\rm FM}(f)\neq 0$ then the cohomology class $[\frak{D}\mathbb{X}]$ is uniquely determined.
\end{lemma}
\begin{proof}
Since we have 
$$\textup{pr}_0\circ\widetilde{\textup{res}}_S\left([\mathbb{X}]\right)=\widetilde{\textup{res}}_S([x])=0$$ 
it follows from Lemma~1.2.19 of \cite{Ne06} (or simply chasing the diagram above) that there exists a class $d \in H^1(\widetilde{U}_S(V_f))\otimes\frak{J}/\frak{J}^2$ that verifies 
\begin{itemize}
\item $i(d)=\widetilde{\textup{res}}_S\left([\mathbb{X}]\right)$\,,\\
\item $\beta^1([x_{\textup{f}}])+\partial_1(d)-\partial_1\circ\beta^0(t)=0$ for some $t\in H^0(\widetilde{U}_S(V_f))$\,.
\end{itemize}
Set $[\frak{D}\mathbb{X}]=d-\beta^0(t)$.
To conclude our proof, we explain that the properties \ref{item_i} and \ref{item_ii} determine the cohomology class $[\frak{D}\mathbb{X}]$ uniquely if  $\mathcal{L}_{\rm FM}(f)\neq 0$. Note that the cocycle $\frak{D}\mathbb{X}$ is a priori defined up to an element  $\beta^0(u) \in \widetilde{U}_S(V_f)^1\otimes\frak{I}/\frak{I}^2$, where $u\in \widetilde{U}_S(V_f)^0$ with $[\partial_1\circ\beta^0(u)]=[\beta^1\circ\partial_0(u)]=0$.

We claim that if $\mathcal{L}_{\rm FM}(f)\neq 0$, then the condition $[\beta^1\circ\partial_0(u)]=0$ implies that $[u]\in H^0(\widetilde{U}_S(V_f))$ vanishes (which in turn shows that $[\frak{D}\mathbb{X}]$ is uniquely determined). Indeed, note that we have $H^0(\widetilde{U}_S(V_f))\simeq H^0(\widetilde{\mathbb{D}}_f)$ (since we have $H^1(\QQ_\ell,V_f)=0$ by the local-global compatibility of Langlands correspondence for all $\ell\neq p$) is a $1$-dimensional $E$-vector space spanned by $d_{\widetilde \delta}$. Let $c\in E$ denote the unique scalar so that $[u]=c\cdot d_{\widetilde \delta}$\,; we contend to prove that $c=0$. It follows from the definition of the height pairing in terms of Bockstein maps that
$$0=\langle [\beta^1\circ\partial_0(u)],[y_{\rm f}]\rangle=\frak{H}(\partial_0[u],[y_{\rm f}])=c\cdot \frak{H}(\partial_0 (d_{\widetilde \delta}),[y_{\rm f}])$$
for all $[y_{\rm f}]\in \widetilde{H}^1_{\textup{f}}(V_f)$. In particular, if we pick $[y_{\rm f}]=\partial_0(d_{\widetilde \delta})$ and use \eqref{eqn:Rubin1}, we infer that
$$c\cdot \langle \Psi_2, \partial^{\textup{loc}}_0(d_{\widetilde\delta})\rangle=0\,.$$
Since we have assumed that $\mathcal{L}_{\rm FM}(f)\neq 0$, it follows from \cite[Proposition~2.3.7]{Ben11} (see also the final paragraph of the proof of Theorem~\ref{thm:thepropertiesofthetwovarheight}) that $\langle \Psi_2, \partial^{\textup{loc}}_0(d_{\widetilde\delta})\rangle\neq 0$. This shows that $c=0$, as required.
\end{proof}

The following statement is what we call the \emph{analytic Rubin-style formula} for the height pairing $\frak{H}$ and the elliptic modular form $f$ (in either of the situations \ref{item_A} or \ref{item_B}).
\begin{thm}[Analytic Rubin-style formula]
\label{thm:rubinsformula1}
Let $[x_{\textup{f}}]=[(x, (x_\ell^+),\lambda_\ell)]$ and 
\linebreak
$[y_{\textup{f}}]=[(y, (y_\ell^+),\mu_\ell)]  \in \widetilde{H}^1_{\textup{f}}(V_f)$ be two elements such that that $[x]\in  {H}^1_{\textup{f}}(\QQ,V_f)$. Suppose further that there is an element $[\mathbb{X}] \in H^1(G_{\QQ,S},V_\varepsilon)$  with the property that $\textup{pr}_0\left([\mathbb{X}]\right)=[x]$. Then,
\be\label{eqn_thm_4_13_17_02_2022}
\frak{H}([x_{\textup{f}}],[y_{\textup{f}}])=-\textup{inv}_p\left(\frak{D}\mathbb{X}_p\cup i_p^+(y_p^+)\right)\,
\ee
where $\textup{inv}_p$ denotes the local invariant map at $p$ and $\frak{D}\mathbb{X}_p$ is any cocycle representing $[\frak{D}\mathbb{X}_p]$.
\end{thm}
The proof of Theorem~\ref{thm:rubinsformula1} will be given after Remark~\ref{rem:Appnoambiguity} and Corollary~\ref{thm:rubinsformula}.
\begin{rem}
\label{rem:Appnoambiguity}
Although the cocycle $\frak{D}\mathbb{X}$ is defined only up to an element $\beta^0(u) \in \widetilde{U}_S(V_f)^1\otimes\frak{I}/\frak{I}^2$ as in the proof of Lemma~\ref{lem:bocksteinnormalizedderivative}, Equation~\eqref{eqn_thm_4_13_17_02_2022} shows that the right hand side of this formula does not depend on the choice of the cocycle $\frak{D}\mathbb{X}$ (as the left hand side doesn't). In more precise terms, we have
\be\label{eqn_4_21_17_02_2022_6_21}
\textup{inv}_p\left(\beta^0(u)\cup i_p^+(y_{p}^+)\right)=0
\ee
for all $u \in \widetilde{U}_S(V_f)^0$ such that $[\beta^1\circ\partial_0(u)]=0$ (i.e. for all $u$ such that $\frak{D}\mathbb{X}$ and $\frak{D}\mathbb{X}+\beta^0(u)$ both verify the conclusions of Lemma~\ref{lem:bocksteinnormalizedderivative}) and $y_p^+$ as above.
\end{rem}
\begin{define}
\label{def:bocksteinder} Set $\omega=X$ in the situation of \ref{item_A} and $\omega=\varpi_{\kappa}$ in the situation of \ref{item_B}. We let  $\frak{d}\left[\mathbb{X}\right] \in H^1(\widetilde{\mathbb{D}}_f)$ be
the unique element such that $\frak{d}\left[\mathbb{X}\right]\otimes \omega\pmod{\frak{J}^2} \in H^1(\widetilde{\mathbb{D}}_f) \otimes \frak{I}/\frak{I}^2$ corresponds to $[\frak{D}\mathbb{X}_p]$ under the isomorphism 
\be\label{eqn_isom_4_23_2022_02_17_13_02}
H^1(\widetilde{U}_p(V_f,\mathbb D_f))\otimes \frak{J}/\frak{J}^2\simeq H^1(\widetilde {\mathbb D}_f)
\otimes \frak{J}/\frak{J}^2
\ee
induced from \eqref{eqn:identifyingsungularquotients_2}. The class $\frak{d}\left[\mathbb{X}\right]$ will be called the \emph{Bockstein-normalized derivative of} of $\mathbb{X}$. For any $u\in \widetilde{U}_S(V_f)^0$ with $\partial_1\circ\beta^0(u)=0$, we define $\beta^0_?[u]\in H^1(\widetilde{\mathbb{D}}_f)$ (where $?={\rm cyc}$ in the situation of \ref{item_A} and $?={\rm wt}$ in the situation of \ref{item_B}) as the unique element such that $\beta^0_?[u] \otimes \omega\pmod{\frak{J}^2} \in H^1(\widetilde{\mathbb{D}}_f)\otimes \frak{J}/\frak{J}^2$ corresponds to $[\beta^0(u)]\in H^1(\widetilde{U}_p(V_f,\mathbb D_f))\otimes \frak{J}/\frak{J}^2$ under the isomorphism \eqref{eqn_isom_4_23_2022_02_17_13_02}.
\end{define}
\begin{cor}
\label{thm:rubinsformula}
In the situation of Theorem~\ref{thm:rubinsformula1} we have
$$\frak{H}([x_{\textup{f}}],[y_{\textup{f}}])=-\left\langle \frak{d}\left[\mathbb{X}\right], [y_p^+]\right\rangle\cdot \omega\,\, $$
where $\left\langle\,,\,\right\rangle : H^1(\widetilde{\mathbb{D}}_f)\times H^1(\mathbb{D}_f)\ra E$
is the canonical pairing and $\omega=X$ when we are in the situation of \textup{\ref{item_A}} and $\omega=\varpi_{\kappa}$ in the situation of \textup{\ref{item_B}}.
\end{cor}

\begin{proof}[Proof of Theorem~\ref{thm:rubinsformula1}]
The proof of this theorem is purely formal and follows the proof of 
\cite[Proposition 11.3.15]{Ne06} with obvious modifications.
Let
\[
Z^{\bullet}=\mathrm{cone} \left (\tau_{\geqslant 2} C^{\bullet}(G_{\QQ,S},
A(1))\xrightarrow{\res_S} \tau_{\geqslant 2} K^{\bullet}_S(A(1))\right ) [-1].
\]
The  global class field theory gives rise to the following diagram with exact rows:
$$\xymatrix@C=.5cm{H^2(G_{\QQ,S},\QQ_p(1))\otimes\frak{J}/\frak{J}^2\ar[r]^(.48){\textup{res}_S}\ar@2{-}[d]&\oplus_{\ell \in S}H^2(G_\ell,\QQ_p(1))\otimes\frak{J}/\frak{J}^2 \ar[r]^(.52){}\ar@2{-}[d]& H^3(Z^{\bullet})\otimes\frak{J}/\frak{J}^2\ar[r]\ar@{.>}[d]^{\textup{inv}_S}& 0\\
H^2(G_{\QQ,S},\QQ_p(1))\otimes\frak{J}/\frak{J}^2\ar[r]^(.48){\textup{res}_S}&\oplus_{\ell \in S}H^2(G_\ell,\QQ_p(1))\otimes\frak{J}/\frak{J}^2 \ar[r]^(.7){\sum_{\ell\in S}\textup{inv}_\ell}& \frak{J}/\frak{J}^2 \ar[r]& 0{\,.}}$$

Suppose $[z_{\textup{f}}]=\left[(z,z_S^+,\omega_S)\right] \in  \widetilde{H}^2_{\textup{f}}(V_f)\otimes\frak{J}/\frak{J}^2$ and $[y_{\textup{f}}]=\left[(y, y_S^+,\mu_S)\right]\in \widetilde{H}^1_{\textup{f}}(V_f)$, where we have set  $z_S^+=(z_\ell^+)$, $\omega_S=(\omega_\ell)$, $y_S^+=(y^+_{\ell})$, and
$\mu_S=(\mu_{\ell})$ to simplify notation. The formula 
$$z_\textup{f}\,\widetilde\cup\, y_\textup{f} :=(z\cup y, \omega_S\cup \textup{res}_S(y) + i_S^+(z_S^+)\cup \mu_S) \in Z^3\otimes\frak{J}/\frak{J}^2$$
defines a cup product 
$$\widetilde\cup: \left(\widetilde{C}^2_{\textup{f}}(V_f)\otimes\frak{J}/\frak{J}^2\right)\otimes\widetilde{C}_{\textup{f}}^{1}(V_f)\lra Z^3\otimes\frak{J}/\frak{J}^2\,$$
which is homotopic to the cup-product (\ref{formulaforcupproductofSelmercomplexes}) by
\cite[Proposition~1.3.2]{Ne06}. Therefore the duality  
$$\langle \,,\,\rangle_{\textup{PT}} \,:\,\left(\widetilde{H}^2_{\textup{f}}(V_f)\otimes\frak{I}/\frak{I}^2\right)\otimes\widetilde{H}^1_{\textup{f}}(V_f)\lra H^3(Z^{\bullet})\otimes \frak{J}/\frak{J}^2
\xrightarrow{\textup{inv}_S}\frak{J}/\frak{J}^2 $$
induced by (\ref{cupproductforSelmercomplexes}) on the level of cohomology can be computed 
by
\[
\langle [z_{\textup f}],[y_{\textup f}]\rangle_{\textup{PT}} = \textup{inv}_S(z_\textup{f}\,\widetilde\cup\, y_\textup{f}).
\]
Since the cohomological dimension of $G_S$ is $2$, $z\cup y$ is a coboundary 
and there exists a cochain $W$ such that $dW=z\cup y.$
Therefore one may compute $\textup{inv}_S([z_\textup{f}\,\widetilde\cup\, y_\textup{f}])$ to be 
\[\textup{inv}_S([z_\textup{f}\,\widetilde\cup\, y_\textup{f}])=\sum_{\ell\in S}\textup{inv}_{\ell}\left(\omega_\ell\cup \textup{res}_\ell(y) + i_\ell^+(z_\ell^+)\cup \mu_\ell + \textup{res}_{\ell}(W)\right).\]
Let now  $[z_{\textup{f}}]=\beta^1([x_{\textup{f}}]),$ where $[x_{\textup{f}}]\in \widetilde H^1(V_f).$
Since  $\textup{pr}_0(\mathbb{X})=[x],$ it follows from the exact sequence
\[
H^1(G_{\QQ,S},V_{\varepsilon})\lra H^1(G_{\QQ,S},V_f)
\xrightarrow{\beta^1} H^2(G_{\QQ,S},V_f)\otimes \frak{J}/\frak{J}^2
\] 
that $z$ is a coboundary. Write $z=dA$ for some $A\in C^1 (G_{\QQ,S},V_f)$
and take $W=A\cup y.$ Then $i_{\ell}^+(z_{\ell}^+)=d(\omega_{\ell}+\res_{\ell}(A))$
for each $\ell\in S$ and it is easy to check that
\[
\beta^1\left([x_{\textup{f}}]\right)=\partial_1 \circ\widetilde{\res}_S\left(\left[\left( \omega_{\ell}+\res_{\ell}(A)\right)_{\ell}\right]\right).
\]
Thus,
\begin{align*}
\frak H([x_{\textup{f}}],[y_{\textup{f}}])&=\sum_{\ell\in S}\textup{inv}_{\ell}\left(\omega_\ell\cup \textup{res}_\ell(y) + i_\ell^+(z_\ell^+)\cup \mu_\ell + \textup{res}_{\ell}(A)\cup \res_{\ell}(y)\right)\\
&=\sum_{\ell\in S}\textup{inv}_{\ell}\left((\omega_\ell+\res_{\ell}(A))\cup \textup{res}_\ell(y) + d(\omega_{\ell}+\res_{\ell}(A))\cup \mu_\ell)\right)\\
&=\sum_{\ell\in S}\textup{inv}_{\ell}\left((\omega_\ell+\res_{\ell}(A))\cup (\textup{res}_\ell(y) + d \mu_\ell)\right)\\
&=\sum_{\ell\in S}\textup{inv}_{\ell}\left((\omega_\ell+\res_{\ell}(A))\cup i_{\ell}^+(y_{\ell}^+)\right)\\
&=\sum_{\ell\in S}\textup{inv}_{\ell}\left(\widetilde{\res}_{\ell}(\omega_\ell+\res_{\ell}(A))\cup i_{\ell}^+(y_{\ell}^+)\right ).
\end{align*}
We remark that $\widetilde{\res}_{\ell}(\omega_\ell+\res_{\ell}(A))$ and  $i_{\ell}^+(y_{\ell}^+)$ are cocycles for every $\ell$. Furthermore, 
\be
\label{eqn_Langlands_compatibility_vanishing}
\textup{inv}_{\ell}\left(\widetilde{\res}_{\ell}(\omega_\ell+\res_{\ell}(A))\cup i_{\ell}^+(y_{\ell}^+)\right )=0, \qquad \ell \neq p,
\ee
because  $H^1_{\textup{f}}(\QQ_\ell,V_f)=0$ for $\ell \neq p$ by the local-global compatibility of the Langlands correspondence (which in turn implies, by local duality and the local Euler characteristic formula, that $H^1(\QQ_\ell,V_f)=0$ for all such $\ell$). It follows from the definition of $\frak{D}\mathbb X_p$ that
\[
\widetilde{\res}_{p}(\omega_p+\res_{p}(A))=-\frak{D}\mathbb X_p+\widetilde{\res}_p(B)
\]
for some $B\in C^1(G_{\QQ,S},V_f).$ An easy computation shows that  
\[
\textup{inv}_{p} (\widetilde{\res}_p(B)\cup i_{p}^+(y_{p}^+))=
\sum_{\ell\in S}\textup{inv}_{\ell} (\widetilde{\res}_{\ell}(B)\cup i_{\ell}^+(y_{\ell}^+))=0
\]
where the first identity follows from the same reasoning as \eqref{eqn_Langlands_compatibility_vanishing} and the vanishing statement follows on copying the final displayed equation in the proof of \cite[Proposition 11.3.15]{Ne06}:
\begin{align*}
    \sum_{\ell\in S}\textup{inv}_{\ell} \left(\widetilde{\res}_{\ell}(B)\cup i_{\ell}^+(y_{\ell}^+)\right)&=
    \sum_{\ell\in S}\textup{inv}_{\ell} \left({\res}_{\ell}(B)\cup i_{\ell}^+(y_{\ell}^+)\right)\\
    &=\sum_{\ell\in S}\textup{inv}_{\ell} \left({\res}_{\ell}(B)\cup (\res_{\ell}(y)+d\mu_\ell)\right)\\
    &=\sum_{\ell\in S}\textup{inv}_{\ell} \left({\res}_{\ell}(B\cup y)- d({\res}_{\ell}(B)\cup\mu_\ell)\right)=0\,.
\end{align*} 
We therefore conclude that
\[
\frak H([x_{\textup{f}}],[y_{\textup{f}}])=-\textup{inv}_{p}\left(\frak{D}\mathbb X_p\cup i_p^+(y_p^+) \right )
\]
and  Theorem~\ref{thm:rubinsformula1} is proved.
\end{proof}
\subsubsection{}
\label{subsubsec_441_2022_02_17}
Our main objective in \S\ref{subsubsec_441_2022_02_17} is to prove Theorem~\ref{thm:rubinsformula2var}, where we establish a Rubin-style formula for the central critical weight-height pairing $\frak{h}_{\mathbf{f}}^{\textup{c-wt}}$.
\begin{define}
\label{def:definepartialderivatives} Let $\frak{X}\in H^1 (G_{\QQ,S},\overline{V}_{\mathbf{f}}).$ 
\\\\
i)  We define $\frak{X}^{\textup{cyc}} \in H^1(G_{\QQ,S},\overline{V}_{f}),$  
$\frak{X}^{\textup{wt}} \in H^1(G_{\QQ,S},{V}_{\mathbf{f}})$ and 
$\frak{X}^{\textup{c-wt}}\in H^1(G_{\QQ,S},{V}_\mathbf{f}(\boldsymbol{\chi}^{1/2}))$ to be the images of $\frak{X}$ under obvious projection maps.
\\\\
ii) If, in addition, $\mathrm{pr}_0(\frak{X})\in H^1_{\mathrm{f}}(\QQ,V_f),$  we denote by $\frak{d}_{\textup{cyc}}\,\frak{X} \in H^1(\widetilde{\mathbb D}_f),$ 
$\frak{d}_{\textup{wt}}\frak{X} \in H^1(\widetilde{\mathbb D}_f)$ and   $\frak{d}_{\textup{c-wt}}\frak{X} \in H^1(\widetilde{\mathbb D}_f)$ 
the Bockstein normalized derivatives of the classes $[\mathbb{X}]=\frak{X}^{\textup{cyc}} \mod {X^2},$  $[\mathbb{X}]=\frak{X}^{\textup{wt}} \mod {\varpi_{\kappa}^2}$ and   $[\mathbb{X}]=\frak{X}^{\textup{c-wt}} \mod {\varpi_{\kappa}^2},$ respectively.  




\end{define}

\begin{prop}
\label{prop:comparebocksteinpartialder}
Suppose we are given a class $\frak{X} \in H^1(G_{\QQ,S},\overline{V}_\mathbf{f})$ whose image under the natural map 
$H^1(G_{\QQ,S},\overline{V}_\mathbf{f})\ra H^1(G_{\QQ,S},{V}_{f})$ lands in $H^1_\textup{f}(\QQ,V_f)$. Then, 
$$\frak{d}_{\textup{c-wt}}\frak{X}=\frac{\frak{d}_{\textup{cyc}}\frak{X}}{2}\,+\,\frak{d}_{\textup{wt}}\frak{X}{\,+\,\beta^0_{\rm cyc}[u] + \beta^0_{\rm wt}[v]}$$
for some $u,v\in \widetilde{U}_p(V_f)^0$ with $[\beta^1\circ\partial_0(u)]=0=[\beta^1\circ\partial_0(v)]$, where $\beta^0_{\rm cyc}[u]$ and $\beta^0_{\rm cyc}[v]$ are given as in Definition~\ref{def:bocksteinder}.
\end{prop}
It follows from the discussion of Remark~\ref{rem:Appnoambiguity} that we have $\beta^0_{\rm cyc}[u]=0=\beta^0_{\rm wt}[v]$ if $\mathcal{L}_{\rm FM}(f)\neq 0$. However, our main result in \S\ref{subsubsec_441_2022_02_17} (Theorem~\ref{thm:rubinsformula2var} below) does not require this as an input owing to \eqref{eqn_4_21_17_02_2022_6_21}.
\begin{proof}[Proof of Proposition~\ref{prop:comparebocksteinpartialder}]
Let $\Pi_\mathbf{f}: H^1_\Iw(\mathbb{D}^\dagger_{\textup{rig}}(V_{\mathbf{f}}))\ra H^1_\Iw(\widetilde{\mathbb{D}}_\mathbf{f})$ and $\textup{pr}_{\gamma,\kappa}: H^1_\Iw(\widetilde{\mathbb{D}}_\mathbf{f}) \ra H^1(\widetilde{\mathbb{D}}_{f})$ denote the obvious maps. 

Since we have $H^2_{\rm Iw}(\widetilde{\mathbb{D}}_\mathbf{f})=0=H^2(\widetilde{\mathbb{D}}_\mathbf{f})$, we have an isomorphism
\be
\label{eqn_15_02_2022_16_38_1}
H^1_{\rm Iw}(\widetilde{\mathbb{D}}_\mathbf{f})/(\gamma-1) \xrightarrow{\sim} H^1(\widetilde{\mathbb{D}}_\mathbf{f})
\ee
that factors the natural morphism $H^1_{\rm Iw}(\widetilde{\mathbb{D}}_\mathbf{f})\ra H^1(\widetilde{\mathbb{D}}_\mathbf{f})$, and an isomorphism 
\be
\label{eqn_15_02_2022_16_38_2}
H^1(\widetilde{\mathbb{D}}_\mathbf{f})/(\varpi_\kappa) \xrightarrow{\sim} H^1(\widetilde{\mathbb{D}}_{f})
\ee
that factors the natural map $H^1(\widetilde{\mathbb{D}}_\mathbf{f}) \to H^1(\widetilde{\mathbb{D}}_{f})$. Combining \eqref{eqn_15_02_2022_16_38_1} and \eqref{eqn_15_02_2022_16_38_2}, we deduce that the natural map $H^1_{\rm Iw}(\widetilde{\mathbb{D}}_\mathbf{f}) \xrightarrow{\textup{pr}_{\gamma,\kappa}} H^1(\widetilde{\mathbb{D}}_{f})$ factors through an isomorphism 
$$H^1_{\rm Iw}(\widetilde{\mathbb{D}}_\mathbf{f})/(\gamma-1,\varpi_\kappa)\xrightarrow{\sim} H^1(\widetilde{\mathbb{D}}_{f})\,.$$
In other words, 
\be
\label{eqn_15_02_2022_16_38_3}
\ker(\textup{pr}_{\gamma,\kappa})=(\gamma-1,\varpi_\kappa)H^1_{\rm Iw}(\widetilde{\mathbb{D}}_\mathbf{f})\,.
\ee
Under our running hypothesis, we have $\textup{pr}_{\gamma,\kappa}\circ \Pi_{\bf f}(\frak{X})=0$, hence $\Pi_{\bf f}(\frak{X})\in \ker(\textup{pr}_{\gamma,\kappa})$. This observation combined with \eqref{eqn_15_02_2022_16_38_3} shows that  we may write
\be\label{eqn:totalderivativeofX}\Pi_\mathbf{f}\left(\frak{X}\right)=\frac{\gamma-1}{\log \chi(\gamma)}\cdot\frak{X}_\gamma+\varpi_\kappa\cdot\frak{X}_\kappa\ee
for some $\frak{X}_\gamma, \frak{X}_\kappa \in H^1_\Iw(\widetilde{\mathbb{D}}_\mathbf{f})$. Since
\begin{align*}
H^0(\widetilde{\mathbb D}_f)&=\ker\left(H^1_{\Iw}(\widetilde{\mathbb D}_f)\stackrel{[\gamma-1]}{\lra}  H^1_{\Iw}(\widetilde{\mathbb D}_f)\right)\\
&=\ker \left(H^1(\widetilde{\mathbb D}_{\bf f})\stackrel{[\varpi_{\kappa}]}{\lra} H^1 (\widetilde{\mathbb D}_{\bf f})\right)
\end{align*}
we can ensure also that 
$$\textup{pr}_{\gamma,\kappa}(\frak{X}_{\gamma})\otimes 1=\frak{d}_{\textup{cyc}}\frak{X} \hbox{ \,\,\,\,and \,\,\,\,} \textup{pr}_{\gamma,\kappa}(\frak{X}_{\kappa})\otimes 1=\frak{d}_{\textup{wt}}\frak{X}\,. $$
Let us rewrite (\ref{eqn:totalderivativeofX}) in the  following form:
\begin{align}\label{eqn:shiftfortheccder}\Pi_\mathbf{f}\left(\frak{X}\right)&=\left(\frac{\gamma-1}{\log \chi(\gamma)}-\frac{\varpi_\kappa}{2}\right)\cdot\frak{X}_\gamma+\varpi_\kappa\cdot\left(\frac{\frak{X}_\gamma}{2}+\frak{X}_\kappa\right).
\end{align}
Recall that we have set  $\Theta=\displaystyle\frac{\gamma-{\boldsymbol{\chi}}^{1/2}(\gamma^{-1})}{\log \chi (\gamma)}$ and let us denote by $I_\Theta\subset \mathcal{H}_A$ the ideal generated by $\Theta$. Then the class $\frak{X}^{\textup{c-wt}} \in H^1(\mathbb{D}^\dagger_{\textup{rig}}(V_{\mathbf{f}}(\boldsymbol{\chi}^{1/2})))$ is the image of $\frak{X}$ under the compositum 
$$H^1_\Iw(\mathbb{D}^\dagger_{\textup{rig}}(V_{\mathbf{f}}))\stackrel{\sim}{\lra}H^1(\mathbb{D}^\dagger_{\textup{rig}}(\overline{V}_{\mathbf{f}})){\lra} H^1(\mathbb{D}^\dagger_{\textup{rig}}(V_{\mathbf{f}}(\boldsymbol{\chi}^{1/2})))\,,$$
where the second map is induced from the projection $$\textup{pr}_\Theta: \overline{V}_{\mathbf{f}}\lra \overline{V}_{\mathbf{f}}/\Theta\cdot
\overline{V}_{\mathbf{f}}\cong {V}_{\mathbf{f}}(\boldsymbol{\chi}^{1/2}).$$ 
that we have discussed in detail as part of Remark~\ref{rem:reconstructtheselfdualHidafamily}. The multiplication by $\gamma-\boldsymbol{\chi}^{1/2}(\gamma)$ yields an isomorphism
$$\left[\gamma-\boldsymbol{\chi}^{1/2}(\gamma)\right]\,:\,\overline{V}_{\mathbf{f}} \stackrel{\sim}{\lra} \Theta\cdot \overline{V}_{\mathbf{f}}$$
(where the inversion of $\gamma$ is due to the fact that it acts on $\mathcal{H}_A^\iota$) that in turn induces a natural inclusion
\begin{align*}
(\gamma-\boldsymbol{\chi}^{1/2}(\gamma)) \cdot 
H^1_{\Iw}(\QQ_p,{V}_{\mathbf{f}}) \subset H^1(G_p, I_{\Theta}\overline{V}_{\bf f})\,.
\end{align*}
It therefore follows that
\be
\label{eqn:kernelofprojection}
(\gamma-\boldsymbol{\chi}^{1/2}(\gamma)) \cdot 
H^1_{\Iw}(\QQ_p,{V}_{\mathbf{f}}) \subset
\ker \left (
H^1_{\Iw}(\QQ_p, V_{\mathbf f})\xrightarrow{\textup{pr}_\Theta}
H^1(G_p, V_{\mathbf f} (\boldsymbol{\chi}^{1/2}))
\right ).
\ee

We have a commutative diagram
\[
\xymatrix{H^1_\Iw(\QQ_p,V_{\mathbf{f}})
\ar[r]^{\Pi_{\mathbf f}} \ar[d]^{\textup{pr}_{\Theta}} &H^1_\Iw(\widetilde{\mathbb D}_{\mathbf f}) \ar[d]^{\textup{pr}_{\Theta}}\\
H^1(G_p, V_{\mathbf f}(\boldsymbol{\chi}^{1/2})) \ar[r]^{\Pi} 
&H^1(\widetilde{\mathbb D}_{\mathbf f}(\boldsymbol{\chi}^{1/2})){\,.}
}
\]
Using \eqref{eqn:shiftfortheccder}, \eqref{eqn:kernelofprojection} and the Taylor expansion 
$$ \frac{\gamma-\boldsymbol{\chi}^{1/2}(\gamma)}{\log\chi(\gamma)}=
\frac{\gamma-1}{\log\chi(\gamma)}-\frac{\varpi_{\kappa}}{2}+C\varpi_{\kappa}^2$$
of $(\gamma-\boldsymbol{\chi}^{1/2}(\gamma))/\log\chi(\gamma)$ we infer that
$$\Pi(\frak{X}^{\textup{c-wt}})=\Pi\circ \textup{pr}_{\Theta} (\frak{X})=\varpi_\kappa\cdot\textup{pr}_\Theta\left( \frac{\frak{X}_\gamma}{2}+\frak{X}_\kappa\right)-C\cdot\varpi_\kappa^2\cdot\textup{pr}_\Theta\left(\frak{X}_\gamma\right).$$
By the commutativity of the diagram
$$\xymatrix{H^1_\Iw(\widetilde{\mathbb{D}}_\mathbf{f}) \ar[r]^(.45){\textup{pr}_\Theta}\ar[rd]_{\textup{pr}_{\kappa,\gamma}}&H^1(\widetilde{\mathbb{D}}_\mathbf{f}(\boldsymbol{\chi}^{1/2}))\ar[d]^{\textup{pr}_{1/2}} \\
& H^1(\widetilde{\mathbb{D}}_{f})\,,
}$$
where  the map $\textup{pr}_{{1}/{2}}$ is induced by the natural reduction map modulo $\varpi_\kappa$, we see that the element 
\[
\frac{\frak{d}_{\textup{cyc}}\frak{X}}{2}+\frak{d}_{\textup{wt}}\frak{X}=
\textup{pr}_{\kappa,\gamma} \left (
\frac{\frak{X}_\gamma}{2}+\frak{X}_\kappa \right )
\]
verifies the condition \ref{item_i} of Lemma~\ref{lem:bocksteinnormalizedderivative}\,: 
\[
i \left (\frac{\frak{X}_\gamma}{2}+\frak{X}_\kappa \right )=
\widetilde{\res}_p(\frak{X}^{\textup{c-wt}}).
\]
We next check that it satisfies also the condition \ref{item_ii} of Lemma~\ref{lem:bocksteinnormalizedderivative}. We denote by
$\beta^{\textup{cyc}},$ $\beta^{\textup{wt}}$ and $\beta^{\textup{c-wt}}$
the coboundary maps $\widetilde H^1_{\textup{f}}(V_f)\rightarrow  
\widetilde H^2_{\textup{f}}(V_f)\otimes \frak{I}/\frak{I}^2$ (where $I=J_E$ or $\frak{p}$, depending on whether we are in the situation of \ref{item_A} or \ref{item_B}) associated to the respective families
$\overline{V}_f$, $V_{\mathbf f}$ and $V_{\mathbf{f}}(\boldsymbol{\chi}^{1/2})$ of Galois representations. Starting off with the commutative diagram
\[
\xymatrix{
0\ar[r] &\frak{p}V_{\mathbf{f}}/\frak{p}^2V_{\mathbf{f}}
\ar[r] &V_{\mathbf{f}}/\frak{p}^2V_{\mathbf{f}}\ar[r] &V_f\ar[r] &0 \\
0\ar[r] &\frak{P}\overline{V}_{\mathbf{f}}/\frak{P}^2\overline{V}_{\mathbf{f}} \ar[u] \ar[d]
\ar[r] &\overline{V}_{\mathbf{f}}/\frak{P}^2\overline{V}_{\mathbf{f}} \ar[u] \ar[d] \ar[r] &V_f
\ar[u]_{=} \ar[d]^{=}\ar[r] &0\\
0\ar[r] &J_E\overline{V}_{f}/J_E^2\overline{V}_{f}
\ar[r] &\overline{V}_{f}/J_E^2\overline{V}_{f}\ar[r] &V_f\ar[r] &0
}
\]
we obtain the following commutative diagram:
\[
\xymatrix{
&\widetilde{H}^2_f(V_f)\otimes_E\frak{p}/\frak{p}^2,\\
\widetilde{H}^1_f(V_f) \ar[ur]^{\beta^{\textup{wt}}} \ar[r]^(.4){\beta} \ar[dr]^{\beta^{\textup{cyc}}}
&\widetilde{H}^2_f(V_f)\otimes_E\frak{P}/\frak{P}^2\ar[u]
\ar[d],\\
&\widetilde{H}^2_f(V_f)\otimes_E J_E/J_E^2
}
\]
We therefore have,
\[
\beta([x_{\textup{f}}])=\beta^{\textup{cyc}}\left(\left[x_{\textup{f}}\right]\right)\cdot(\gamma-1)+
\beta^{\textup{wt}}\left(\left[x_{\textup{f}}\right]\right)\cdot\varpi_\kappa=-\left(\partial_1(\frak{d}_{\textup{cyc}}(\frak{X}))\cdot(\gamma-1)+
\partial_1(\frak{d}_{\textup{wt}}(\frak{X}))\cdot\varpi_{\kappa}\right)
\]
and
\[
\bigl. \beta^{\textup{c-wt}}\left(\left[x_{\textup{f}}\right]\right)=\beta([x_{\textup{f}}])\Bigr \vert_{\gamma-1=\frac{\varpi_{\kappa}}{2}}= -\partial_1\left (\frac{\frak{d}_{\textup{cyc}}\frak{X}}{2}+\frak{d}_{\textup{wt}}\frak{X}\right )\cdot \varpi_ {\kappa}\,.
\]
The proof of the proposition is now complete since the Bockstein normalized derivatives are uniquely determined up to an element $\beta^0(x)\in \widetilde{U}_p(V_f)^1$ where $x\in  \widetilde{U}_p(V_f)^0$ with $[\beta^1\circ\partial_0(x)]=0$ (c.f. the proof of the final portion of Lemma~\ref{lem:bocksteinnormalizedderivative} and Remark~\ref{rem:Appnoambiguity}).
\end{proof}

\begin{thm}[Central Critical Rubin-style formula for the $p$-adic height]
\label{thm:rubinsformula2var}
Let $[x_{\textup{f}}]=[(x, x_\ell,\lambda_\ell)]$ and $[y_{\textup{f}}]=[(y, y_\ell,\mu_\ell)]  \in \widetilde{H}^1_{\textup{f}}(V_f)$ be two elements such that $[x]\in  {H}^1_\textup{f}(\QQ,V_f)$. Suppose further that there is an element $\frak{X} \in H^1(G_{\QQ,S},\overline{V}_\mathbf{f})$  with the property that $\textup{pr}_0\left(\frak{X}\right)=[x]$. Then,
$$\frak{h}_{\mathbf{f}}^{\textup{c-wt}}\left([x_{\textup{f}}],[y_{\textup{f}}]\right)=-\frac{\left\langle \frak{d}_{\textup{cyc}}\frak{X}, [y_p^+]\right\rangle}{2}-\left\langle \frak{d}_{\textup{wt}}\frak{X}, [y_p^+]\right\rangle \,.$$
\end{thm}
\begin{proof}
This follows from Corollary~\ref{thm:rubinsformula} and Proposition~\ref{prop:comparebocksteinpartialder} together with \eqref{eqn_4_21_17_02_2022_6_21}.
\end{proof}

\section{The two-variable Perrin-Riou logarithm}
 
\subsection{Perrin-Riou's logarithm: General Theory} In this subsection we review the construction of the large exponential map for crystalline $(\varphi,\Gamma)$-modules of rank $1$ with coefficients in  an affinoid algebra $A.$ 
We refer the reader to \cite{Nak, Nak13} for general constructions and further results.

Let $\boldsymbol{\delta} \,:\,\mathbb Q_p^*\rightarrow A^*$ be a continuous character.  Recall that we denote by $\RR_A(\boldsymbol\delta )$ the 
$(\varphi,\Gamma)$-module $\RR_A\cdot e_{\boldsymbol\delta}$  of rank $1$ defined  by  
\[
\varphi (e_{\boldsymbol\delta})=\boldsymbol{\delta} (p)\cdot e_{\boldsymbol\delta},\qquad  
\tau (e_{\boldsymbol\delta})=\boldsymbol\delta (\chi (\tau))\cdot e_{\boldsymbol\delta},\quad \tau \in \Gamma.
\]
Set   $\alpha (x) =\boldsymbol\delta (p)\in  A.$
We assume that $\boldsymbol\delta $ is crystalline and of constant slope, namely that 

a) $\boldsymbol\delta \vert_{\mathbb Z_p^*} (u)=u^m$ for some integer $ m\geqslant 1.$ 

b) The function $x \mapsto v_p(\alpha (x))$ is constant on $U=\mathrm{Spm}(A).$

The crystalline module $\CDcris (\RR_A(\boldsymbol\delta ))$
associated to $\RR_A(\boldsymbol\delta )$ is the free $A$-module of rank $1$ generated by 
$d_{\boldsymbol\delta}=t^{-m}e_{\boldsymbol\delta}.$ Note that 
\[
\varphi (d_{\boldsymbol\delta})=p^{-m}\boldsymbol\delta (p)d_{\boldsymbol\delta}.
\]
The Iwasawa cohomology $H^1_{\Iw}(\RR_A(\boldsymbol\delta ))$ is canonically isomorphic to
$\RR_A(\boldsymbol\delta )^{\Delta,\psi=1}.$ 
Set $\mathcal E_A=\RR_A\cap A[[\pi]].$ The set $\mathcal E_A^{\Delta,\psi=0}$ is the free
$\mathcal H_A$-submodule of $\mathcal E_A$ generated by 
\[
1+\pi_0=1+\frac{1}{\vert\Delta\vert}\underset{\sigma\in \Delta}\sum \sigma (\pi).
\]
We equip $\mathcal E_A$ with the operator $\partial =(1+\pi)\,\displaystyle\frac{d}{d\pi}.$

Let $z\in \CDcris (\RR_A(\boldsymbol\delta ))\otimes_A \mathcal E_A^{\Delta,\psi=0}.$
It may be shown that the equation
\[
(\varphi-1) F=z-\frac{\partial^mz(0)}{m!}t^m,\qquad t=\log (1+\pi)
\]
has a solution in $\CDcris (\RR_A(\boldsymbol\delta ))\otimes_A \mathcal E_A$ and we define
\begin{equation}
\label{formulalargeexponential}
\textup{Exp}_{\RR (\boldsymbol\delta)}
(z)=
(-1)^m\frac{\log \chi (\gamma)}{p}\,t^m\partial^m (F).
\end{equation}
Exactly as in the classical case $A=E$ (see \cite{Ber03}), it is not hard to check that  
$\textup{Exp}_{\RR (\boldsymbol\delta)}
(z) \in \RR (\boldsymbol\delta )^{\Delta,\psi=1}.$ Therefore we have a well defined map
\begin{equation*}
\textup{Exp}_{\RR (\boldsymbol\delta)}
\,:\,\CDcris (\RR_A(\boldsymbol\delta ))\otimes_A\mathcal E_A^{\Delta,\psi=0} \rightarrow H^1_{\Iw}(\RR_A(\boldsymbol\delta )).
\end{equation*}

Now let $V$ be a $p$-adic representation with coefficients in $A$ and let $\mathbb D=\RR_A(\boldsymbol\delta)$ be a crystalline $(\varphi,\Gamma)$-submodule of $\DdagrigA (V)$ of constant slope.

We denote by \footnote{The
definition of the  functor $\DdagrigA$ depends on the  choice of 
a fixed compartible system $\varepsilon =(\zeta_{p^n})_{n\geqslant 0}$ 
of primitive $p^n$th roots of the unity, because this choice 
identifies  $1+\pi$ with the  element  $[\varepsilon]\in \mathbf B_{\textup{dR}}$ of Fontaine's ring of de Rham periods.
In order to stress this dependence, we shall write  the exponential with the superscript $\varepsilon.$ 
} 

\[
\textup{Exp}_{V,\mathbb D}^{\varepsilon}\,:\,\CDcris (\mathbb D)\otimes_A \mathcal E_A^{\Delta,\psi=0} \rightarrow H^1_{\Iw}(\QQ_p,V)
\]
the composition of $\textup{Exp}_{\RR (\boldsymbol\delta)}
$ with the natural  map
$H^1_{\Iw}(\RR_A(\boldsymbol\delta))\rightarrow H^1_{\Iw}(\QQ_p,V).$

Recall  (see \ref{subsec:iwasawacohomology}) that we set $K_n=\QQ_p(\zeta_{p^{n+1}})^{\Delta}.$    The cohomological pairings
\[
\langle \,,\,\rangle_n \,:\, H^1(K_n,V^*(1))\times H^1(K_n,V)\rightarrow A
\]
give rise to a $\Lambda_A$-linear pairing
\[
\left\langle \,,\,\right\rangle_{\Iw} \,:\, H^1_{\Iw}(\QQ_p,V^*(1))\times H^1_{\Iw}(\QQ_p,V)^{\iota}
\rightarrow \Lambda_A.
\]
By linearity, this pairing extends to
\[
\left\langle \,,\,\right\rangle_{\Iw} \,:\, H^1_{\Iw}(\QQ_p,V^*(1))\otimes_{\Lambda_A}\mathcal H_A\times H^1_{\Iw}(\QQ_p,V)^{\iota}\otimes_{\Lambda_A}\mathcal H_A
\rightarrow \mathcal H_A.
\]
For any $\eta \in \CDcris (\mathbb D)$ the element $\widetilde\eta=\eta\otimes (1+\pi_0)$
lies in $\CDcris (\mathbb D)\otimes \mathcal E_A^{\Delta,\psi=0}$ and we define
a map
\begin{equation*}
\frak{Log}_{V^*(1),\eta}^{\varepsilon}\,:\,H^1_{\Iw}(\QQ_p,V^*(1))\rightarrow \mathcal H_A
\end{equation*}
by
\[
\frak{Log}^{\varepsilon}_{V^*(1),\eta}(z)=\left\langle z, \textup{Exp}_{V,\mathbb D}^{\varepsilon^{-1}}(\widetilde\eta)^{\iota}\right\rangle_{\Iw} .
\]
The following theorem summarizes the main properties of these maps.

\begin{thm}
\label{them:propertiestwovarPRlog}
\textup{i)} The maps $\textup{Exp}_{V,\mathbb D}^{\varepsilon}$ and $\frak{Log}_{V^*(1),\eta}^{\varepsilon}$
commute with the base change.

\textup{ii)} Assume that $x\in \mathrm{Spm}(A)$ is an $E$-valued point such that 
$V_x$ is semistable at $x.$ Let $E_x=A/\frak m_x.$ Then $\textup{Exp}_{V_x,\mathbb D_x}^{\varepsilon}$
coincides with the restriction of Perrin-Riou's large exponential map
\[
\textup{Exp}_{V_x,m}^{\varepsilon}\,:\,\Dcris (V_x)\otimes_E \mathcal E_{E_x}^{\Delta,\psi=0} \rightarrow H^1_{\Iw}(\QQ_p,V_x)
\]
on $\CDcris (\mathbb D_x)\otimes_{E_x} \mathcal E_{E_x}^{\Delta,\psi=0}.$
Therefore, $\frak{Log}^{\varepsilon}_{V_x^*(1),\eta_x}$ coincides with 
Perrin-Riou's large logarithm map as defined in \cite{Ben14a}.

\textup{iii)} Let $\frak a\,:\,\mathcal H_A\rightarrow A$ denote the augmentation map induced from $\gamma\mapsto 1.$  Then
\be
\label{eqn:improvedlogarithm}
\frak{a}\circ  \frak{Log}^{\varepsilon}_{V^*(1),\eta}(z)=
(m-1)! \, \frac{1-p^{-1}\alpha (x)^{-1}}{1-\alpha (x)}\, 
\mathrm{Log}_{V^*(1),\eta}(\mathrm{pr}_0(z)), 
\ee  
where $\mathrm{pr}_0\,:\,H^1_{\Iw} (\QQ_p,V)\rightarrow H^1(\QQ_p,V)$
is the canonical projection,  $\mathrm{Log}_{V^*(1),\eta}(\mathrm{pr}_0(z))=\left\langle\mathrm{pr}_0(z), \exp_{\mathbb D}(\eta)\right\rangle\in A$ 
and $\exp_{\mathbb D}$ is the Bloch-Kato exponential map $($for its general definition, see \cite{Ben11,Nak13}$)$.  
\end{thm}
\begin{proof} i) This follows directly  from constructions. 

ii) This portion follows from  i) and Berger's interpretation  in \cite{Ber03} of the large exponential map; more precisely, by Berger's result that Perrin-Riou's large exponential may be defined via (\ref{formulalargeexponential}).

iii) By the construction of the large exponential map we have 
\[
\mathrm{pr}_0 (\mathrm{Exp}_{V,\mathbb D}^{\epsilon^{-1}}(\widetilde\eta)^{\iota})=(m-1)! \exp_{V} \left (\frac{1-p^{-1}\varphi^{-1}}{1-\varphi}\, \eta\right )
\]
(see \cite{PR94} or \cite[Corollaire 4.10]{BB08}, for the proof in the case $A=E.$ The same computation proves this formula in the general case.) Therefore
\[
\frak{a}\circ  \frak{Log}^{\varepsilon}_{V^*(1),\eta}(z)=
(m-1)! \, \frac{1-p^{-1}\alpha (x)^{-1}}{1-\alpha (x)} \langle
\mathrm{pr}_0(z),\exp_{V}(\eta)\rangle.
\]
\end{proof}

\begin{rem}
The formula (\ref{eqn:improvedlogarithm}) implies that 
$\frak{a}\circ  \frak{Log}^{\varepsilon}_{V^*(1),\eta}(z)$ factors into a product of two analytic functions. This factorization may be seen as the counterpart (in the context of the theory of $(\varphi,\Gamma)$-modules) of the identification of the improved $p$-adic $L$-function. For the multiplicative group, this formula was proved by Venerucci in \cite[Proposition 3.8]{venerucciarticle}. 
\end{rem}

\subsection{Perrin-Riou's logarithm for modular forms}
\label{subsec:PRslogformodforms}
We apply the theory outlined in the previous section to the situation studied in Section~\ref{subsec:padicfamilies}. Let $(V_{\mathbf{f}}, \mathbb D_{\mathbf{f}})$ be  the triangulation 
associated to a $p$-adic family $\mathbf{f}$ of cuspidal eigenforms passing through $f$ and an eigenvalue $\alpha$ of $f.$ We fix $\varepsilon$ and write $\mathrm{Exp}_{f}$ for the sake of brevity to denote the exponential map 

\[
\mathrm{Exp}_{V_f,\mathbb D}^{\varepsilon}\,:\, \Dcris(V_f)^{\varphi=\alpha}\otimes_E
\mathcal E^{\Delta, \psi=0}\lra H^1_{\Iw}(\QQ_p, V_f).
\]
Recall that we denote by  $d_{\delta}$
the element $e_{{\delta}}\otimes t^{-k/2}.$ 
Set $\eta=d_{\delta}.$ 
Using the skew-symmetric pairing $V_f\times V_f\rightarrow E,$ 
we may consider the map $\frak{Log}^{\epsilon}_{V_f,\eta }.$ In order to simplify the notation, we shall denote this map by $\frak{Log}_f$.

Assume that $a_p=p^{k/2-1}.$ Then   $\alpha=\alpha (k)=p^{-1}$ and 
from Theorem~\ref{them:propertiestwovarPRlog} it follows that 
$\mathrm{Exp}_{f}(d_{\delta}\otimes (1+\pi_0))\in (\gamma-1)\cdot\mathcal H_E.$ 
Let $\mathrm{pr}_0$ denote the natural projection $H^1_{\Iw}(\Ddagrig (V_f))\rightarrow
H^1(G_p,V_f).$

\begin{prop}
\label{prop:derofcycloPRmap}
\textup{i)} There exists a unique $F\in H^1_{\Iw}(\mathbb D_f)$ such that 
$$(\gamma-1) F=\mathrm{Exp}_{f}(d_{\delta}\otimes (1+\pi_0)).$$
\\
\textup{ii)} For every $z \in H^1_{\Iw}(\Ddagrig (V_f))$ we have 
$$ \left(1-\frac{1}{p}\right)\left\langle z,  F \right\rangle_{\Iw}\equiv\Gamma(k/2)\,(\log \chi (\gamma))^{-1}\langle \Pi_f\circ\textup{pr}_{\gamma}(z), \Psi_1^*\rangle \pmod {J_E^2}\,,$$
where $\Pi_f\,:\, H^1_{\Iw}(\Ddagrig (V_f))\rightarrow H^1_{\Iw}(\widetilde{\mathbb D}_f)$ and $\mathrm{pr}_{\gamma}\,:\,H^1_{\Iw}(\widetilde{\mathbb D}_f)\rightarrow H^1(\widetilde{\mathbb D}_f)$  denote the natural projections.
In particular, 
$$\left\langle z\,,\,\textup{Exp}_{f}^{\varepsilon^{-1}}\left((1+X)\otimes d_\delta\right)\right\rangle_{\Iw} \in J_E^2 \iff \textup{pr}_0(z) \in H^1_{\textup{f}}(\QQ_p,V)\,.$$
\\
\textup{iii)}
Assume in addition  that 
$\mathrm{pr}_0(z)\notin H^1_{\textup{f}}(G_p,V_f)$ and  define 
\[
L_p(z,s)=\mathcal A (\frak{Log}_{f}(z))(-s)=\frak{Log}_{f}(z)(\chi (\gamma)^s-1). 
\]
Then,
$$
L'_p(z,s)=
-\left\langle z,  F \right\rangle\equiv \Gamma(k/2)\,\left(1-\frac{1}{p}\right)^{-1}{\al_{\textup{FM}}(f)}\left[ d_\delta,\exp^*_{V_f}(\textup{pr}_0(z)) \right]_{V_f}\cdot  s \pmod{s^2},$$
where $[\,,\,]_{V_f}\,:\,\Dst (V_f)\times \Dst (V_f)\rightarrow E$ denotes the canonical duality.
\end{prop}
\begin{proof} i) and ii) are  proved in \cite[Proposition 1.3.7]{Ben14a}; whereas the assertion iii) is proved in \cite[Proposition~2.2.2]{Ben14a}.
\end{proof}

Recall that we have a duality
\[
V_{\mathbf{f}}(\boldsymbol{\chi})\times V_{\mathbf{f}}\rightarrow E(\chi).
\]
We denote by $\frak{Log}_{\mathbf{f}}$ the large logarithm map
$\frak{Log}_{V_{\mathbf{f}}(\boldsymbol{\chi}),\eta}^{\varepsilon}$ for $\eta=d_{\boldsymbol\delta}.$ 
Let 
\[\mathbf z\in H^1_{\Iw}(\DdagrigA (V_{\mathbf{f}}(\boldsymbol{\chi}))).
\]  
We define 
\be\label{eqn_define_L__2022_02_14_17_02}
L_p(\mathbf{z},\kappa,s):=\mathcal {A}\left(\frak{Log}_{\mathbf{f}}(\mathbf{z})\right) (\kappa,-s)=\left (\frak{Log}_{\mathbf{f}}(\mathbf{z})\right ) ((1+p)^{\kappa-k}-1, \chi (\gamma)^{s}-1).
\ee
Note that $L_p(\mathbf{z},\kappa,s)$ is  a locally analytic function in the weight variable $\kappa$ and the cyclotomic variable $s$. Write $\mathbf{z}(k) \in H^1_\Iw(\Ddagrig (V_f))$ for the stalk of $\mathbf{z}$ at the point $(f,a_p(f))$ (corresponding to $\kappa=k$). 
We denote by $\Pi_{\mathbf{f}}\,:\,H^1_{\Iw}(\DdagrigA (V_{\mathbf{f}}))\rightarrow
H^1_{\Iw}(\widetilde{\mathbb D}_{\mathbf{f}})$ and $\mathrm{pr}_{\gamma,\kappa}\,:\,
H^1_{\Iw}(\widetilde{\mathbb D}_{\mathbf{f}})\rightarrow H^1(\widetilde{\mathbb D}_f)$ the natural projections. Recall that  $\frak{I} \subset  E[[\kappa-k, s]]$ denotes the ideal generated by $(\kappa-k)$ and $s$.
\begin{thm}
\label{thm:keyleadingtermformulas}
\item[i)]  Suppose $\mathbf{z}\in H^1_{\Iw}(\DdagrigA (V_{\mathbf{f}})).$  
Then we have the following equality inside $\frak{I}/\frak{I}^2$:
\begin{multline*}\left(1-\frac{1}{p}\right)\Gamma(k/2)^{-1}L_p(\mathbf{z},\kappa,s)\equiv\\ 
\langle \textup{pr}_{\gamma,\kappa} \circ \Pi_{\mathbf{f}}(\mathbf{z}), \Psi_1^*\rangle \cdot s-\frac{\al_{\textup{FM}}(f)}{2}\cdot \langle \textup{pr}_{\gamma,\kappa}\circ \Pi_{\mathbf{f}}(\mathbf{z}),\exp_{\mathbb{D}_f}(d_\delta)\rangle \cdot(\kappa-k)\, \pmod{\frak{I}^2}. \end{multline*}
\item[ii)] Suppose in addition that $\mathrm{pr}_{\gamma,\kappa}\circ \Pi_{\mathbf f}(\mathbf{z})\neq 0.$  Then,
\begin{multline*}\left(1-\frac{1}{p}\right)\Gamma(k/2)^{-1}L_p(\mathbf{z},\kappa,s)\equiv \\
\left(s-\frac{\kappa-k}{2}\right){\al_{\textup{FM}}(f)} [d_\delta\,,\,\exp^*_V(\textup{pr}_0(\mathbf{z}(k)))]_{V_f} \pmod{\frak{J}^2}.
\end{multline*}
\end{thm}
\begin{proof}
\item[i)] By Theorem~\ref{thm:stevensderformula}  together with the fact that 
$\alpha(k)=p^{-1}$ we have 
 \be\label{eqn:1}1-p^{-1}\alpha(\kappa)^{-1} \equiv -\frac{\al_{\textup{FM}}(f)}{2} \cdot(\kappa-k) \pmod{\frak{J}^2}.\ee 

It follows from Theorem~\ref{them:propertiestwovarPRlog}(iii) that
\be\label{eqn_2022_02_14_16_37}
 \mathcal{A}\circ \frak{Log}_{\mathbf{f}}(\mathbf{z})\vert_{s=0}=\Gamma(k/2)\cdot \frac{1-p^{-1}\alpha(x)^{-1}}{1-\alpha(x)}\, \mathcal{A}^{\rm wt}\circ {\rm Log}_{V^*_{\bf f}(\boldsymbol{\chi}),d_{\boldsymbol{\delta}}}({\rm pr}_0(\Pi_{\bf f}(z)))
\ee
which, combined with \eqref{eqn:1} yields
\be\label{eqn_2022_02_14_16_46}
 \mathcal{A}\circ \frak{Log}_{\mathbf{f}}(\mathbf{z})\vert_{s=0}=-\frac{\al_{\textup{FM}}(f)\Gamma(k/2)}{2\left(1-\frac{1}{p} \right)} \langle \textup{pr}_{\gamma,\kappa}\circ \Pi_{\mathbf f}(\mathbf{z}), \exp_{\mathbb{D}_f}(d_\delta)\rangle \cdot (\kappa-k)+O((\kappa-k)^2)\,.
\ee
On the other hand, we have
\begin{align}
    \begin{aligned}
    \label{eqn_2022_02_14_16_51}
    \mathcal{A}\circ \frak{Log}_{\mathbf{f}}(\mathbf{z})\vert_{\kappa=k}&=\mathcal{A}^{\rm cyc}\circ\frak{Log}_{f}(\mathbf{z}(k))\\
    &\stackrel{\rm Prop.~\ref{prop:derofcycloPRmap}(ii)}{=}\left(1-\frac{1}{p}\right)^{-1}\Gamma(k/2)\,\langle \textup{pr}_{\gamma,\kappa}\circ \Pi_{\mathbf{f}}(\mathbf{z}), \Psi_1^*\rangle \cdot s +O(s^2)\,.
    \end{aligned}
\end{align}
The proof of (i) follows on combining \eqref{eqn_2022_02_14_16_46} and \eqref{eqn_2022_02_14_16_51} together with the definition \eqref{eqn_define_L__2022_02_14_17_02} of $L_p(\mathbf{z},\kappa,s)$.


\item[ii)] This follows from the first part, combined with Proposition~\ref{prop:derofcycloPRmap}(iii).
\end{proof}

\section{A $p$-adic Beilinson formula in two-variables}
\label{sec:rubinstyleformula}
We keep previous notation and conventions. Suppose we are given a class $\frak{X} \in H^1(G_{\QQ,S},\overline{V}_\mathbf{f})$ whose image under the natural map 
$H^1(G_{\QQ,S},\overline{V}_\mathbf{f})\ra H^1(G_{\QQ,S},{V}_{f})$ lands in $H^1_\textup{f}(\QQ,V_f)$. We denote by $\frak{d}_{\mathrm{?}}(\frak X)$  the Bockstein normalized partial derivatives for $\textup{?}\in \{{\textup{cyc}},{\textup{wt}}, {\textup{c-wt}}\}.$
To simplify notation, we write $L_p(\frak{X},\kappa,s)$ for the $p$-adic $L$-function
associated to $\res_p(\frak{X})\in H^1_{\Iw}(\DdagrigA (V_{\textup{f}})).$

\begin{prop}
\label{prop:partialderivativespartiallywelldefine}
We have the following equality inside $\frak{I}^2/\frak{I}^3$:
\begin{align*}
\left(1-\frac{1}{p}\right)&\Gamma(k/2)^{-1}L_p(\frak{X},\kappa,s)\\ 
&=\left\langle \frak{d}_{\textup{cyc}}(\frak{X}), \Psi_1^*\right\rangle\cdot s^2-\frac{\al_{\textup{FM}}(f)}{2}\cdot\left\langle \frak{d}_{\textup{wt}}(\frak{X}),
\exp(d_\alpha)\right\rangle\cdot(\kappa-k)^2\\
&\qquad+\left(\left\langle \frak{d}_{\textup{wt}}(\frak{X}), \Psi_1^*\right\rangle-\frac{\al_{\textup{FM}}(f)}{2}\cdot\left\langle \frak{d}_{\textup{cyc}}(\frak{X}),\exp(d_\alpha)\right\rangle\right)\cdot s (\kappa-k)\,. \end{align*}
\end{prop}
In particular, all the three quantities
$\left\langle \frak{d}_{\textup{cyc}}(\frak{X}),\Psi_1^*\right\rangle\,,$ 
$\left\langle \frak{d}_{\textup{wt}}(\frak{X}),\exp(d_\alpha)\right\rangle$ 
and 
$$\left(\left\langle \frak{d}_{\textup{wt}}(\frak{X}), \Psi_1^*\right\rangle-\frac{\al_{\textup{FM}}(f)}{2}\cdot\left\langle \frak{d}_{\textup{cyc}}(\frak{X}),\exp(d_\alpha)\right\rangle\right)$$
are independent of the choices involved in the definitions of $\frak{d}_{\textup{cyc}}(\frak{X})$ and $\frak{d}_{\textup{wt}}(\frak{X}).$
\begin{proof}[Proof of Proposition~\ref{prop:partialderivativespartiallywelldefine}]
It follows from the $A$-linearity of the large regulator map $\frak{Log}_{\mathbf{f}}$ that 
$$\frak{Log}_{\mathbf{f}}(\frak{X})=\frac{\gamma-1}{\log\chi(\gamma)}\cdot\frak{Log}_{\mathbf{f}}(\frak{X}^{(\gamma)}) +\varpi_\kappa\cdot\frak{Log}_{\mathbf{f}}(\frak{X}^{(\kappa)})$$ 
and the proof follows applying the Amice transform on both sides and using Theorem~\ref{thm:keyleadingtermformulas}(i) (with $\mathbf{z}=\frak{X}^{(\gamma)}$ and 
$\mathbf{z}=\frak{X}^{(\kappa)}$). 
\end{proof}

We are now ready to prove one the main results of this paper. 
\begin{thm}
\label{thm:leadingtermformula} Let $\frak{X}\in H^1(G_{\QQ,S}, \overline{V}_{\mathbf{f}})$
and let $[x]=\mathrm{pr}_0(\frak{X})\in H^1(G_{\QQ,S},V_f)$ denote its canonical projection. 
\\\\
\textup{i)} 
We have the following equalities inside $\frak{I}/\frak{I}^2$\,.
\begin{align*}
\left\langle\Psi_2,\partial^{\textup{loc}}_0(d_{\widetilde{\delta}}) \right\rangle\left(1-\frac{1}{p}\right)&\,\Gamma(k/2)^{-1}L_p(\frak{X},\kappa,s)
\\&=\mathcal{L}_{\textup{FM}}(f)\left\langle\Psi_2,\partial^{\textup{loc}}_0(d_{\widetilde{\delta}}) \right\rangle\left[ d_\delta,\exp^*_{V}(\textup{res}_p([x])) \right]_{V_f}\cdot \left(s-\frac{\kappa-k}{2}\right)\\
\\&={\mathcal{L}_{\textup{FM}}(f)}\left[ d_\delta,\exp^*_{V}(\textup{res}_p([x])) \right]_{V_f}\cdot \mathcal{A}\left(\mathbb{H}_{\mathbf{f}}\left(\partial_0(d_{\widetilde\delta}),\partial_0(d_{\widetilde\delta})\right)\right).
\end{align*}
\\\\
\textup{ii)} Suppose that $[x]\in H^1_{\textup{f}}(V_f)$ and denote by
$[x_{\textup{f}}]\in \widetilde{H}^1_{\textup{f}}(V_f)$ the canonical lift of $[x].$ Then,
\begin{align*}\frac{\left(1-\frac{1}{p}\right)}{\Gamma(k/2)}\left\langle\Psi_2,\textup{res}_p([x])\right\rangle\cdot\left\langle\Psi_1,\partial^{\textup{loc}}_0(d_{\widetilde{\delta}}) \right\rangle\cdot L_p(\frak{X},\kappa,s)=\mathcal{A}\left(\frak{Reg}_{\mathbb{H}_{\mathbf{f}}}(\partial_0(d_{\widetilde\delta}),[x_{\textup{f}}])\right)\end{align*}
 as elements of $\frak{I}^2/\frak{I}^3$, where we have put 
 $$\frak{Reg}_{\mathbb{H}_{\mathbf{f}}}(\partial_0(d_{\widetilde\delta}),[x_{\textup{f}}]):=\det\left(\begin{array}{cc} \mathbb{H}_{\mathbf{f}}\left(\partial_0(d_{\widetilde\delta}),\partial_0(d_{\widetilde\delta})\right) &\mathbb{H}_{\mathbf{f}}\left(\partial_0(d_{\widetilde\delta}),[x_{\textup{f}}]\right)\\\\
 \mathbb{H}_{\mathbf{f}}\left([x_{\textup{f}}],\partial_0(d_{\widetilde\delta})\right)&\mathbb{H}_{\mathbf{f}}\left([x_{\textup{f}}],[x_{\textup{f}}]\right)\end{array}  \right)\,.$$
\end{thm}
\begin{proof}[Proof of Theorem~\ref{thm:leadingtermformula}]
We note that (i) trivially holds when $\al_{\textup{FM}}(f)=0$, as both sides in the asserted identity equal zero in this particular case. Indeed, note that by the explicit description of $\mathcal{L}_{\rm FM}(f)$ in the final paragraph of the proof of Theorem~\ref{thm:thepropertiesofthetwovarheight} (see also Equation \eqref{eqn:Linvariantreferee} below), we have $\left\langle\Psi_2,\partial^{\textup{loc}}_0(d_{\widetilde{\delta}}) \right\rangle=0$ whenever $\mathcal{L}_{\rm FM}(f)=0$. In particular, the left hand side of the displayed equation in the statement of Theorem \ref{thm:leadingtermformula}(i) vanishes if $\mathcal{L}_{\rm FM}(f)=0$. So we may assume  without loss of generality for our proof of (i) that  $\al_{\textup{FM}}(f)\neq 0$. In this case, the first asserted equality follows from Theorem~\ref{thm:keyleadingtermformulas}(ii) (on eliminating the redundant  $\left\langle\Psi_2,\partial_0^{\rm loc}(d_{\widetilde{\delta}})  \right\rangle$ factors from both sides using the non-vanishing of $\al_{\textup{FM}}(f)$) and the second equality from \eqref{eqn:Rubin1}.

The proof of (ii) is rather a tedious computation. To ease notation, we set $C=\left(1-\frac{1}{p}\right)\Gamma(k/2)^{-1}$. We also write $\al$ in place of $\al_{\textup{FM}}(f)$. Using (\ref{eqn:Rubin1}), (\ref{eqn:Rubin2}) together with Theorem~\ref{thm:thepropertiesofthetwovarheight}(i) and Corollary~\ref{thm:rubinsformula} we have,
\begin{align}
\begin{aligned}
\label{eqn:functionaleqnforheights1}
\mathcal{A}\left(\frak{Reg}_{\mathbb{H}_{\mathbf{f}}}(\partial_0(d_{\widetilde\delta}),[x_{\textup{f}}])\right)&=\mathcal{A}\circ \det\left(\begin{array}{cc} \mathbb{H}_{\mathbf{f}}\left(\partial_0(d_{\widetilde\delta}),\partial_0(d_{\widetilde\delta})\right) &\mathbb{H}_{\mathbf{f}}\left(\partial_0(d_{\widetilde\delta}),[x_{\textup{f}}]\right)\\\\
 \mathbb{H}_{\mathbf{f}}\left([x_{\textup{f}}],\partial_0(d_{\widetilde\delta})\right)&\mathbb{H}_{\mathbf{f}}\left([x_{\textup{f}}],[x_{\textup{f}}]\right)\end{array}  \right)\\
 &=\left\langle\Psi_2,\partial^{\textup{loc}}_0(d_{\widetilde{\delta}})  \right\rangle\cdot\left\langle \frak{d}_{\textup{cyc}}(\frak{X)}, [x_p^+]\right\rangle\cdot\left(s-\frac{\kappa-k}{2}\right)^2\\
 &\qquad\qquad\qquad\qquad\qquad+\left\langle\Psi_2,[x_p^+] \right\rangle^2\cdot s\cdot(\kappa-k-s)\,,
\end{aligned}
 \end{align}
where we used the identity (that follows from the functional equation for the two-variable height pairing)
\begin{align*}
\notag \mathcal{A}\left(\mathbb{H}_{\mathbf{f}}\left([x_{\textup{f}}],\partial_0(d_{\widetilde\delta})\right)\right)(\kappa,s)&=-\mathcal{A}\left(\mathbb{H}_{\mathbf{f}}\left(\partial_0(d_{\widetilde\delta}),[x_{\textup{f}}]\right)\right)(\kappa,\kappa-s-k)\\
&=-\left\langle\Psi_2,[x_p^+] \right\rangle\cdot(\kappa-k-s)\,.
\end{align*}
We define $\Phi_1:=\Psi_1$ and $\Phi_2:=\Psi_2+\al\cdot\Psi_1$. These two elements constitute a basis of $H^1(\widetilde{\mathbb{D}}_f)$ and enjoy the properties that
\begin{equation}\label{eqn:refereeaskedtolabelthis}
\left\langle \Phi_2,\partial^{\textup{loc}}_0(d_{\widetilde{\delta}}) \right\rangle=0=\left\langle\Phi_1,[x_p^+]\right\rangle\,
\ee
where the first equality follows from the description of the $\al$-invariant via the identity (as we have readily recalled in the proof of Theorem~\ref{thm:thepropertiesofthetwovarheight})
\be\label{eqn:Linvariantreferee} \left\langle\Psi_2,\partial^{\textup{loc}}_0(d_{\widetilde{\delta}}) \right\rangle=-\al\cdot \left\langle\Psi_1,\partial^{\textup{loc}}_0(d_{\widetilde{\delta}}) \right\rangle
\ee
and the second from Corollary~\ref{ref:crucialproperty1and2}. Write 
$$\frak{d}_{\textup{cyc}}(\frak{X})=a_1(\gamma)\Phi_1+a_2(\gamma)\Phi_2, \qquad 
\frak{d}_{\textup{wt}}(\frak{X})=a_1(\kappa)\Phi_1+a_2(\kappa)\Phi_2.
$$ 
Using (\ref{eqn:functionaleqnforheights1}) and (\ref{eqn:refereeaskedtolabelthis}) we have,
\begin{align}
\label{eqn:regulatorsimplifiedformula}\mathcal{A}\left(\frak{Reg}_{\mathbb{H}_{\mathbf{f}}}\left(\partial_0(d_{\widetilde\delta}),[x_{\textup{f}}]\right)\right)
 &=\left\langle\Psi_2,\partial^{\textup{loc}}_0(d_{\widetilde{\delta}}) \right\rangle\cdot a_2(\gamma)\left\langle\Psi_2, [x_p^+]\right\rangle\cdot\left(s-\frac{\kappa-k}{2}\right)^2\\
 \notag&\,\,\,\,\,\,\,+\left\langle\Psi_2,[x_p^+] \right\rangle^2\cdot s\cdot(\kappa-k-s)\,.
 \end{align}
 Set $C_1=\left\langle \Psi_2,[x_p^+]\right\rangle$ and $C_2=C_1\cdot\left\langle\Psi_1,\partial^{\textup{loc}}_0(d_{\widetilde{\delta}})  \right\rangle$. Note that $C_2$ is zero iff $C_1$ is. We have the following identities.\\
 $\bullet$ The equation \eqref{eqn:Rubin2} and the Rubin-style formula (Corollary~\ref{thm:rubinsformula}) for the cyclotomic height pairing $\frak{h}_p$ together show that
 \be\label{eqn:identity1}
 \left\langle\Psi_2,[x_p^+] \right\rangle=a_1(\gamma)\left\langle\Psi_1, \partial^{\textup{loc}}_0(d_{\widetilde\delta}) \right\rangle\,,
 \ee
 since we have
 \begin{align*}
     \left\langle\Psi_2,[x_p^+]\right\rangle\cdot s\stackrel{ \eqref{eqn:Rubin2}}{=}\mathcal A\left(\mathbb{H}_{\mathbf{f}}\left([x_{\textup{f}}], \partial_0 (d_{\widetilde{\delta}})\right)\right)&\stackrel{\rm Cor.~\ref{thm:rubinsformula}}{=} \left\langle \frak{d}_{\rm cyc}(\frak{X}), \partial_0^{\rm loc} (d_{\widetilde{\delta}})\right\rangle\cdot s\\
     &=a_1(\gamma)\left\langle \Phi_1, \partial_0^{\rm loc} (d_{\widetilde{\delta}})\right\rangle\cdot s+a_2(\gamma)\left\langle \Phi_2, \partial_0^{\rm loc} (d_{\widetilde{\delta}})\right\rangle\cdot s\\ 
     &\stackrel{\eqref{eqn:refereeaskedtolabelthis}}{=}a_1(\gamma)\left\langle \Phi_1, \partial_0^{\rm loc} (d_{\widetilde{\delta}})\right\rangle\cdot s\\
     &\stackrel{\Phi_1:=\Psi_1}{=}a_1(\gamma)\left\langle \Psi_1, \partial_0^{\rm loc} (d_{\widetilde{\delta}})\right\rangle\cdot s\,.
 \end{align*}
 \\
 $\bullet$ We have
\begin{align}\notag\frak{h}_\mathbf{f}^{\textup{c-wt}}\left([x_{\textup{f}}],\partial_0(d_{\widetilde\delta})\right)&=-\left\langle \frak{d}_{\textup{wt}}(\frak{X}),\partial^{\textup{loc}}_0(d_{\widetilde{\delta}}) \right\rangle-\frac{1}{2}\left\langle \frak{d}_{\textup{cyc}}(\frak{X}),\partial^{\textup{loc}}_0(d_{\widetilde{\delta}}) \right\rangle\\
\label{eqn:wtdirectionRSformulaforreg}&=-a_1(\kappa)\left\langle\Psi_1,\partial^{\textup{loc}}_0(d_{\widetilde{\delta}}) \right\rangle+\frac{\frak{h}_p\left([x_{\textup{f}}],\partial_0(d_{\widetilde\delta})\right)}{2}
\end{align}
where we used Theorem~\ref{thm:rubinsformula2var} for the first equality and the Rubin-style formula for the cyclotomic height-pairing $\frak{h}_p$  and (\ref{eqn:refereeaskedtolabelthis}) for the second. Furthermore, it follows from (\ref{eqn:Rubin2}) and (\ref{eqn:twovarexpanded1}) that
\begin{align*}
\frak{h}_\mathbf{f}^{\textup{c-wt}}\left([x_{\textup{f}}],\partial_0(d_{\widetilde\delta})\right)&=-\frac{\frak{h}_p\left([x_{\textup{f}}],\partial_0(d_{\widetilde\delta})\right)}{2}=\frac{\left\langle\Psi_2,[x_p^+] \right\rangle}{2}
\end{align*}
This combined with \eqref{eqn:identity1} and \eqref{eqn:wtdirectionRSformulaforreg} yields
\begin{align}
\notag a_1(\kappa)&=-\frac{\left\langle\Psi_2,[x_p^+] \right\rangle}{\left\langle\Psi_1,\partial^{\textup{loc}}_0(d_{\widetilde{\delta}}) \right\rangle}
\label{eqn:a1kappaexplcit}\\&=-a_1(\gamma)\,.
\end{align}
$\bullet$ Using Theorem~\ref{thm:rubinsformula2var} along with the fact that the pairing $\frak{h}_\mathbf{f}^{\textup{c-wt}}$ is skew-symmetric, it follows that
\begin{align*} 0=\frak{h}_\mathbf{f}^{\textup{c-wt}}\left([x_{\textup{f}}],[x_{\textup{f}}]\right)&=-\left\langle \frak{d}_{\textup{wt}}(\frak{X}),[x_p^+]\right\rangle-\frac{1}{2}\left\langle \frak{d}_{\textup{cyc}}(\frak{X}),[x_p^+]\right\rangle\\
&=-a_2(\kappa)\cdot\left\langle\Psi_2,[x_p^+] \right\rangle -a_2(\gamma)\cdot \frac{\left\langle\Psi_2,[x_p^+] \right\rangle}{2}\,,
\end{align*}
where we used the fact that $\left\langle \Phi_1, [x_p^+]\right\rangle=0$ for the second line. We conclude that
\be
\label{eqn:identity3}
\left\langle\Psi_2,[x_p^+] \right\rangle\cdot a_2(\kappa)=-\frac{\left\langle\Psi_2,[x_p^+] \right\rangle}{2}\cdot a_2(\gamma)\,.
\ee
We will now use (\ref{eqn:regulatorsimplifiedformula}) to explicitly compute $\mathcal{A}\left(\frak{Reg}_{\mathbb{H}_{\mathbf{f}}}(\partial_0(d_{\widetilde{\delta}}) ,[x_{\textup{f}}])\right)$. We shall compare the resulting expressions for the coefficients of $s^2$, $s(\kappa-k)$ and $(\kappa-k)^2$ to the left hand side of the asserted equality in ii) via Proposition~\ref{prop:partialderivativespartiallywelldefine}. 

 (1) \emph{The coefficient $A_s$ of $s^2$.}
 This coefficient equals
 \begin{align*}
A_s &=\left\langle\Psi_2,\partial^{\textup{loc}}_0(d_{\widetilde{\delta}}) \right\rangle\cdot a_2(\gamma)\left\langle\Psi_2,[x_p^+] \right\rangle-\left\langle\Psi_2,[x_p^+] \right\rangle^2\\
&=C_1\cdot \left(a_2(\gamma)\left\langle\Psi_2,\partial^{\textup{loc}}_0(d_{\widetilde{\delta}}) \right\rangle-\left\langle\Psi_2,[x_p^+] \right\rangle\right)\\
&=C_1\cdot \left(-\al\cdot a_2(\gamma)\left\langle\Psi_1,\partial^{\textup{loc}}_0(d_{\widetilde{\delta}}) \right\rangle-\left\langle\Psi_2,[x_p^+]\right\rangle\right)\\
&=C_1\cdot \left(-\al\cdot a_2(\gamma)\left\langle\Psi_1,\partial^{\textup{loc}}_0(d_{\widetilde{\delta}})\right\rangle-a_1(\gamma)\cdot \left\langle\Psi_1, \partial^{\textup{loc}}_0(d_{\widetilde{\delta}})\right\rangle\right)\\
&=-C_2 \left(\al\cdot a_2(\gamma)+a_1(\gamma)\right)=C_2\cdot \left\langle \frak{d}_{\textup{cyc}}(\frak{X}),\Psi_1^*\right\rangle\,.
 \end{align*}
Here the third equality follows from (\ref{eqn:Linvariantreferee}), the fourth from (\ref{eqn:identity1}) and the last from (\ref{eqn:refereeaskedtolabelthis}). 

 (2) \emph{The coefficient $A_{s,\kappa}$ of $s\cdot(\kappa-k)$.} 
This quantity equals,
 \begin{align*}
A_{s,\kappa}& =C_1\left(\left\langle\Psi_2,[x_p^+] \right\rangle-\left\langle\Psi_2,\partial^{\textup{loc}}_0(d_{\widetilde{\delta}})  \right\rangle \cdot a_2(\gamma)\right)\\
&=C_2\left({\al} \cdot a_2(\gamma)+a_1(\gamma)\right)\\
&=C_2\left(-\al\cdot a_2(\kappa)+\frac{\al}{2}\cdot a_2(\gamma) -a_1(\kappa)\right)\\
&=C_2\left(\left\langle\frak{d}_{\textup{wt}}(\frak{X}),\Psi_1^*\right\rangle+\frac{\al}{2}\cdot\left\langle\frak{d}_{\textup{cyc}}(\frak{X}),\Psi_2^*\right\rangle\right)\\
&=C_2\left(\left\langle\frak{d}_{\textup{wt}}(\frak{X}),\Psi_1^*\right\rangle-\frac{\al}{2}\left\langle\exp_{\mathbb{D}_f}(d_\alpha),\frak{d}_{\textup{cyc}}(\frak{X})\right\rangle\right)
 \end{align*}
where we used (\ref{eqn:identity1}) for the second equality, (\ref{eqn:a1kappaexplcit}) and (\ref{eqn:identity3}) for the third and \cite[Proposition~1.2.6]{Ben14a} for the last.

(3) \emph{The coefficient ${A}_{\kappa}$ of $(\kappa-k)^2$.} 
This quantity equals,
\begin{align*}
A_{\kappa}& =-\frac{1}{4}\left\langle\Psi_2,\partial^{\textup{loc}}_0(d_{\widetilde{\delta}}) \right\rangle\cdot a_2(\gamma)\left\langle\Psi_2, [x_p^+]\right\rangle\\
&=C_2\cdot  \frac{\al}{2}\cdot a_2(\kappa)\\
&=C_2\cdot\frac{\al}{2}\cdot\left\langle\frak{d}_{\textup{wt}}(\frak{X}),\Psi_2^*\right\rangle\\
&=-C_2\cdot \frac{\al}{2}\cdot \left\langle\exp_{\mathbb{D}}(d_\alpha),\frak{d}_{\textup{wt}}(\frak{X})\right\rangle
 \end{align*}
where we used (\ref{eqn:Linvariantreferee}) and (\ref{eqn:identity3}) for the second equality. 

The proof of the theorem now follows using the computations in the paragraphs (1), (2) and (3) and Proposition~\ref{prop:partialderivativespartiallywelldefine}.
\end{proof}

\begin{cor}[Central critical Rubin-style formula for $p$-adic $L$-functions]
\label{thm:ccrubinstyleformula}
In the situation of Theorem~\ref{thm:leadingtermformula}, let us define
\[
\frak{L}_p(\frak{X},\kappa):=L_p \left (\frak{X}, \kappa, \frac{\kappa-k}{2}\right ).
\]
Assume that $[x]\in H^1_{\textup{f}}(G_{\QQ,S},V_f).$ Then
\begin{align}
\label{eqn:rubinstyleforLfunctions}
\left(1-\frac{1}{p}\right)\Gamma(k/2)^{-1}\left\langle\Psi_2,[x_p^+]\right\rangle\left\langle\Psi_1,\partial^{\textup{loc}}_0(d_{\widetilde{\delta}})\right\rangle\cdot \frac{d^2}{d\kappa^2}& \frak{L}_p(\frak{X},\kappa)\Big{|}_{\kappa=k}=\frac{\left\langle\Psi_2,[x_p^+]\right\rangle^2}{2}\,.\end{align}
\end{cor}
\begin{proof}
This follows from Theorem~\ref{thm:leadingtermformula}(ii) (specialized to $s=\frac{\kappa-k}{2}$), more particularly, see \eqref{eqn:regulatorsimplifiedformula} as part of its proof.
\end{proof}

\begin{prop}
\label{prop:whendoescentralcriticalLvanish}
Suppose that we are in the situation of Corollary~\ref{thm:ccrubinstyleformula} and $[x]\in H^1_{\textup{f}}(G_{\QQ,S},V_f)$.
Then,
$$\textup{ord}_{\kappa=k}\, \frak{L}_p(\frak{X},\kappa)>2 \iff \textup{res}_p([x])=0.$$
\end{prop}
The proof of this assertion is obtained after appropriately modifying the proof of \cite[Lemma 6.1]{venerucciarticle}; we include it here for the convenience of the reader.
\begin{proof} ($\implies$)The assumption forces the right hand side of (\ref{eqn:rubinstyleforLfunctions}) to vanish and it follows that $\textup{res}_p([x])=0$. 

($\impliedby$) The hypothesis that $\textup{res}_p([x])=0$ shows that
$\res_p(\frak{X}) \in \frak{P}\cdot H^1_\Iw(\mathbb{D}^\dagger_{\textup{rig}}(V_{\mathbf{f}})).$

We may therefore write 
$$\res_p(\frak{X})=\frac{\gamma-1}{\log\chi(\gamma)}\cdot\mathbf{z}_\gamma+ \varpi_\kappa\cdot \mathbf{z}_\kappa$$
for some $\mathbf{z}_\gamma, \mathbf{z}_\kappa \in H^1_\Iw(\mathbb{D}^\dagger_{\textup{rig}}(V_{\mathbf{f}}))$ and set 
$$z_\gamma:=\textup{pr}_0(\mathbf{z}_\gamma)\,,\, z_\kappa:=\textup{pr}_0(\mathbf{z}_\kappa) \in H^1(\mathbb{D}^\dagger_{\textup{rig}}(V_{{f}}))\,.$$
 We  have 
\begin{align*}
\frak{L}_p(\frak{X},\kappa)&= L_p(\frak{X},\kappa,s)\Big{|}_{s=\frac{\kappa-k}{2}}\\
\label{eqn:plugin1}&= \left((s-1)\cdot L_p(\mathbf{z}_\gamma,\kappa,s)+(\kappa-k)\cdot L_p(\mathbf{z}_\kappa,\kappa,s)\right)\Big{|}_{s=\frac{\kappa-k}{2}}.
\end{align*}
First assume that ${z}_{?}\in H^1_{\textup{f}}(G_p,V_f),$ where $?\in\{\gamma,\kappa\}.$  Then the projection of $z_{?}$ to $H^1(\widetilde{\mathbb D})$ is $0$ and from Theorem~\ref{thm:keyleadingtermformulas} (i) it follows that  
\begin{equation}
\label{eqn:ordertwovanishingformula1}
L_p(\mathbf{z}_?,\kappa,s)\in \frak{I}^2.
\end{equation}
Now assume that  ${z}_{?}\notin H^1_{\textup{f}}(G_p,V_f).$
Write $C=\left(1-\frac{1}{p}\right)\Gamma(k/2)^{-1}$ for simplicity.
From Theorem~\ref{thm:keyleadingtermformulas}(ii) we have  
\begin{align*}
L_p(\mathbf{z}_{?},\kappa,s)\equiv 
\left(s-\frac{\kappa-k}{2}\right)\al_{\textup{FM}}(f)\cdot[d_\delta,\exp^*_V(z_{?})]_{V_f} \pmod{\frak {I}^2}
\end{align*}
Taking $s=\frac{\kappa-k}{2},$ we obtain that $\frak{L}(\mathbf{z}_{?},\kappa)\equiv 0\pmod
{(\kappa-k)^2}.$ Together with (\ref{eqn:ordertwovanishingformula1}), this implies  that 
$\textup{ord}_{\kappa=k}\, \frak{L}_p(\frak{X},\kappa)>2.$ 
\end{proof}

\section{Main results} 
\label{sec:proofofPRconjecture}

\subsection{$p$-adic $L$-functions}
Let 
\[f=\sum a_nq^n \in S_{k}(\Gamma_0(Np))
\] 
be a elliptic newform of even weight $k$ and level $N$ as before. Let $\varepsilon_{f,N} \in \{\pm 1\}$ denote the eigenvalue of the Fricke involution acting on $f$ and set $\varepsilon_f=(-1)^{{k}/{2}}\varepsilon_{f,N}$. According to \cite[Theorem 3.66]{shimura} the sign of $\varepsilon_f$ agrees with the sign of the functional equation of the Hecke $L$-function $L(f,s)$. We suppose throughout that we have $a_p=p^{\frac{k}{2}-1}$. Recall the $p$-adic $L$-function $L_{p,\alpha}(f,\omega^{k/2},s)$ (where $\omega$ is the Teichm\"uller character) of Amice--V\'elu, Manin and Vi\v{s}ik. We refer the reader to \cite[Section 4.1]{Ben14a} for the precise definition and basic properties of this function. 

Let $\mathbf{f}$ be a family of modular forms with coefficients in $A=\mathcal O(U)$ passing through $f,$ where $U=\overline{B}(k,p^{-r})$. Let $\epsilon$ be a Dirichlet character of order dividing $2$; its choice will be made precise\footnote{For the main purposes of this article, $\epsilon$ shall be the trivial character. Despite this fact, we still chose to include it in our notation to have an easier comparison with Seveso's notation in \cite{sevesoCJM}. } below. 
We shall consider the two-variable Mazur-Kitagawa $p$-adic $L$-function $L_p(\mathbf{f},\epsilon,\kappa,s)$  (as introduced in \cite{sevesoCJM}, Section 5) associated with the family $\mathbf{f}$, for $s$ in the neighborhood $k/2+p\ZZ_p$ of $k/2$.   See also \cite{Bel, Pan} for constructions of related $p$-adic $L$-functions and Remark~\ref{rem:comparepadicLfunctionsBelliachevsSeveso} for a comparison of Seveso's notation with that of \cite{Ben14a,mtt} (according to which our normalizations are made). At each classical point $x$ of weight 
$k_x\in \mathbb Z_{\geqslant 2}\cap \overline{B}(k,p^{-r})$ we have  
\begin{equation}
\label{interpolationmazurkitagawa}
L_p (\mathbf{f},\epsilon, k_x, s)=\lambda (x) L_p(\mathbf{f}_x,\epsilon,s),
\end{equation}
where $\lambda (x)$ is Stevens's interpolation factor; see \cite[Theorem 1.5]{bertolinidarmon05} for its defining property and \cite[Theorem 5.3]{sevesoCJM}. Here $L_p(\mathbf{f}_x,\epsilon,s)$ is in Seveso's notation and its comparison to the $p$-adic $L$-function denoted by $L_{p,\alpha(x)}(f_x,\epsilon\omega^{k/2},s)$ in \cite{Ben14a} is included below.
\begin{rem}
\label{rem:comparepadicLfunctionsBelliachevsSeveso}
Given a classical point $x$ of even weight $k_x\in \ZZ_{\geq 2}$ as above, the $p$-adic $L$-function $L_p(\mathbf{f}_x,\epsilon\omega^{k/2},s)$ of Manin-Vi\v{s}ik and Amice-Velu which is denoted in \cite{Ben14a} by $L_{p,\alpha(x)}(\mathbf{f}_x,\epsilon\omega^{k/2},s)$ and corresponds to the function denoted  in \cite{mtt} by the notation $L_{p}(\mathbf{f}_x,\alpha(x),\epsilon\omega^{k/2-1},s-1)$.  We also note that 
\begin{align*}  L_{p}(\mathbf{f}_x,\alpha(x),\epsilon\omega^{k/2-1},s-1)&=L_{p}(\mathbf{f}_x,\alpha(x),\epsilon\omega^{k/2-1}\chi_{s-1})\\
&=L_{p}(\mathbf{f}_x,\alpha(x),\epsilon\omega^{k/2-s}\chi^{s-1})
 \end{align*}
(still in the notation of \cite{mtt}, so that $\chi_s$ is as in Section 13.2 of loc. cit.) where we make sense of the second line only for integer values of $s$. 

Let $\psi$ denote an arbitrary primitive Dirichlet character. In the notation of \cite{sevesoCJM}, the function $L_{p}(\mathbf{f}_x,\psi,s)$ in the neighborhood $j+p\ZZ_p$ of $j$ corresponds to the $p$-adic $L$-function 
\begin{align*} L_{p}(\mathbf{f}_x,\alpha(x),\psi\chi^{s-1})&=L_{p}(\mathbf{f}_x,\alpha(x),\psi\omega^{j-1}\chi_{s-1})\\
&=L_{p}(\mathbf{f}_x,\alpha(x),\psi\omega^{j-1},{s-1})
\end{align*}
of \cite{mtt}. (Note that Seveso has stated the interpolation property for the Mazur-Kitagawa $p$-adic $L$-function only for characters of degree at most $2$, which is also sufficient for our purposes here.)  In particular, for the choice $\psi=\epsilon$ and in the neighborhood $k/2+p\ZZ_p$ of $k/2$, Seveso's notation compares to those of \cite{Ben14a, mtt} via
\begin{align*} L_{p}(\mathbf{f}_x,\epsilon,s)&=L_{p}(\mathbf{f}_x,\alpha(x),\epsilon\omega^{k/2-1},{s-1})\\
&=L_{p,\alpha(x)}(\mathbf{f}_x,\epsilon\omega^{k/2},{s})\,.
\end{align*}
Note therefore that the interpolation property (\ref{interpolationmazurkitagawa}) reads
\be\label{interpolationmazurkitagawa2}
L_p (\mathbf{f},\epsilon, k_x, s)=\lambda (x) L_{p,\alpha(x)}(\mathbf{f}_x,\epsilon\omega^{k/2},{s})
\ee
for $s$ in the neighborhood $k/2+p\ZZ_p$.
\end{rem}

Let $b_{f}^*$ be a canonical basis of $\textup{Fil}^0\mathbf{D}_{\textup{st}}(V_f^*)$ defined by Kato in \cite[Theorem 12.5]{ka1} and let 
$$[\,,\,]_{V_f}\,:\,\mathbf{D}_{\textup{st}}(V_f)\times \mathbf{D}_{\textup{st}}(V_f)\lra E$$
denote the canonical pairing. 
It follows from the work of Kato \cite[Theorem 16.2]{ka1} (see also \cite[Theorem 4.3.2]{Ben14a}) that there exists an element $[z_{f,\Iw}^{\textup{BK}}] \in H^1_\Iw(G_{\QQ,S},V_f)$ such that 
\begin{equation}
\label{interpolationkato}
L_p( [z_{f,\Iw}^{\textup{BK}}],s)=L_{p,\alpha}(f,\omega^{k/2},s+k/2)\,[e_\alpha, b^*_f]\,.
\end{equation}

Hansen~\cite{hansenarxiv}, Wang~\cite{wangarxiv} and the authors of the present paper in~\cite{BB_CriticalKato1} independently gave a construction of an element $\frak{Z}^{\textup{BK}}_{\mathbf{f}}\in H^1_{\Iw}(G_{\QQ,S},V_{\mathbf{f}})$ interpolating Kato's zeta elements. On the other hand, Colmez--Wang~\cite{ColmezWang} and Nakamura~\cite{NakamuraArXiv} constructed a big zeta element over the universal deformation ring. In what follows, we shall work with the element constructed in \cite{BB_CriticalKato1}\footnote{Our $\frak{Z}^{\textup{BK}}_{\mathbf{f}}$ here coincides with the projection of $\mathrm{Tw}_{\frac{k}{2}}({\mathbb{BK}}^{[\mathcal X]}_{N}(j,\xi))$ in \cite[Theorem A]{BB_CriticalKato1} to the $(-1)^{\frac{k}{2}}$-eigenspace for the action of complex conjugation, with a suitable choice of a pair $(j,\xi)$ (which, among other things, depends on the sign of $(-1)^{\frac{k}{2}}$) and $U$ shrank appropriately, as in op. cit.}. Let us put
\[
L_p(\frak{Z}^{\textup{BK}}_{\mathbf{f}},\kappa,s)=\mathcal A(\frak{Log}_{\mathbf{f}}(\frak{Z}^{\textup{BK}}_{\mathbf{f}}))(\kappa,-s).
\]

Then at each classical point $x$ of even weight $k_x\in \mathbb{Z}_{\geqslant 2}\cap \overline{B}(k,p^{-r})$ we have (see \cite[Theorem A]{BB_CriticalKato1}, as well as the proof of \cite{hansenarxiv}, Proposition 4.2.2)
\be\label{eqn:specialize2varkatopadicL}
L_p(\frak{Z}^{\textup{BK}}_{\mathbf f},k_x,s)=c(x)\mathscr{A}_N(k_x,s)L_{p,\alpha(x)}(\mathbf{f}_x,\omega^{k/2},s+k/2)
\ee
for some non-zero constant $c(x)\in E_x$  and $\mathscr{A}_N(\kappa,s)$ is the two-variable Amice-transform of the (twisted) bad Euler factors ${\rm Tw}_{\frac{k}{2}}\mathscr{E}_N$, where $\mathscr{E}_N$ is given as in \cite[Equation (64)]{BB_CriticalKato1} and ${\rm Tw}_{\frac{k}{2}}:\LL\to \LL$ is the twisting morphism given by $\gamma\mapsto \chi^{\frac{k}{2}}(\gamma)\gamma$ on the group-like elements. More precisely,
\[
\mathscr{A}_N(\kappa,s):=\prod_{\ell \mid N}(1-a_\ell({\mathbf f})\ell^{-\frac{k}{2}}\left <\ell\right >^{-s})\,.
\]
The following proposition is an immediate consequence of \cite[Theorem A.(ii).d]{BB_CriticalKato1}. We record its proof here for the sake of completeness.
 \begin{prop}
\label{thm:BKvscentralcriticalL}
We have 
$$L_p(\frak{Z}_{\mathbf{f}}^{\textup{BK}},\kappa,s)\big{|}_{s=\frac{\kappa-k}{2}}=a(\kappa)\cdot L_p (\mathbf{f},\mathds{1}, \kappa, \kappa/2) $$
where $a(\kappa)$  is analytic and non-vanishing in a neighborhood of $\kappa=k$ and $\mathds{1}$ is the trivial character.
\end{prop}
\begin{proof}
Let us define
$$\textup{Err}(\kappa,s):=\frac{L_p(\frak{Z}_{\mathbf{f}}^{\textup{BK}},\kappa,s)}{L_p(\mathbf{f},\mathds{1},\kappa,s+k/2)}\,.$$  
{\blue Using \eqref{interpolationmazurkitagawa2} with $\epsilon=\mathds{1}$ together with \eqref{eqn:specialize2varkatopadicL} and relying on the fact that we have, for every classical $x \in U$ of non-critical slope, $L_{p,\alpha(x)}(\mathbf{f}_x,\omega^{k/2},s+k/2)\neq 0$ as an element of the integral domain $\mathcal{A}^{\rm cyc}(\mathcal{H}_E)$
(to eliminate common terms), we see that
\begin{align*}\mathscr{A}_N(k_x,s)^{-1}\textup{Err}(k_x,s)=\frac{c(x)}{\lambda(x)}
\end{align*}
for every $x$ as above. We remind the reader that $\mathscr{A}_N(k_x,s)$ is a non-zero element of $\mathcal{A}^{\rm cyc}(\mathcal{H}_E)$ (c.f. \cite{BB_CriticalKato1}, Lemma 6.13) and in fact $\mathscr{A}_N(k_x,j)\neq 0$ for any such $x$ and integer $j\neq \dfrac{k_x-k}{2}-1$. Since classical points of non-critical slope are dense in $U$, it follows that the meromorphic function $\mathscr{A}_N(\kappa,s)^{-1}\textup{Err}(\kappa,s)$ does not depend on the cyclotomic variable $s$, i.e. it is only a function of the weight-variable $\kappa$. We remark that $\mathscr{A}_N(\kappa,s)^{-1}\textup{Err}(\kappa,s)$ is a meromorphic function since it is the product of two functions which are both manifestly so.  In particular, 
$$\mathscr{A}_N(\kappa,s)^{-1}\textup{Err}(\kappa,s)=:\textup{Err}_0(\kappa)=\mathscr{A}_N(\kappa,{k}/{2}-1)^{-1}\textup{Err}(\kappa,{k}/{2}-1)\,.$$ 
Note that $\textup{Err}_0(\kappa)$ is a meromorphic function (being the product of two meromorphic functions) and non-vanishing at every classical point of non-critical slope in a neighborhood of $\kappa=k$ (since neither $\mathscr{A}_N(\kappa,s)^{-1}$ nor $\textup{Err}(\kappa,s)$ is identically zero as a function of $s$ for classical $\kappa$); in particular at $k$. The proof is complete on setting $a(\kappa):=\mathscr{A}_N(\kappa,\frac{\kappa-k}{2})\,\textup{Err}_0(\kappa)$ and passing to a suitably small neighbourhood of $\kappa=k$.}
\end{proof}
 Let $\epsilon$ be an auxiliary quadratic or trivial character verifying 
\be\label{eqn:choiceofpsi}
\epsilon(-N)=-\varepsilon_f \hbox{ \,\,\,\, and \,\,\,\,} \epsilon(p)=1\,.
\ee



Let $L$ be an algebraic number field containing all $a_n$. Let $E \supset \QQ_{p^2}$ denote its completion at a fixed (arbitrary) prime above $p$. Let the $\mathcal{M}_{k-2}{/\QQ}$ denote the Iovita-Spiess Chow motive of weight $k$ modular forms (whose $p$-adic realisation affords representations associated to cusp forms which are new at $pN^-$). There exists a map (via Faltings' comparison theorem and Coleman)
$$\log\Phi^{\textup{AJ}}:  \textup{CH}^{{k}/{2}}(\mathcal{M}_{k-2}\otimes H) \lra M_k(\Gamma,E)^*,$$
(which is essentially identical to the $p$-adic Abel--Jacobi-map) where $\Gamma$ is a certain `congruence subgroup' of a suitably chosen quaternion algebra\footnote{All this is made explicit in \cite[Section 5.3.2]{sevesoCJM}; see also \cite{sevesocrelle}.}, $H$ is a certain extension of $K$ and $M_k(\Gamma,E)^*$ is the dual of the space of rigid analytic modular forms for $\Gamma$. Let
\[
\left.  \frak{L}_p(\kappa)=L_p (\mathbf{f},\epsilon, \kappa, s)\right |_{s={\kappa}/{2}}
\]
denote the restriction of $L_p(\mathbf{f},\epsilon,\kappa,s)$ on the central critical line. 
The following result is \cite[Theorem 6.1]{sevesoCJM}. 

\begin{thm}[Seveso] \label{thm:sevesomainmazurKitag} $\,$\\
\textup{i)} The $p$-adic $L$-function $\frak{L}_p(\kappa)$ vanishes at $\kappa=k$ at least of order $2$.\\
\textup{ii)} There exists a Heegner cycle 
$$y^\epsilon\in \textup{CH}^{{k}/{2}}(\mathcal{M}_{k-2}\otimes H)$$ 
in the Chow group of codimension-$k/2$ cycles and an element $t_f \in L^\times$ such that 
$$\frac{d^2 \frak{L}_p}{d\kappa^2}\Big{|}_{\kappa=k}=t_f\cdot \left(\log\Phi^{\textup{AJ}}(y^\epsilon)\big{|}_{f^{\textup{rig}}}\right)^2\,.$$
Here $f^{\textup{rig}} \in M_{k}(\Gamma,E)$ is the rigid analytic modular form associated (via the Cerednik--Drinfeld uniformization) to the Jacquet--Langlands correspondent of $f$.
\end{thm}
When the sign of the functional equation is $-1$, the trivial character verifies (\ref{eqn:choiceofpsi}). We shall eventually place ourselves in a situation where $r_{\textup{an}}(f)=1$ and we will take $\epsilon$ to be the trivial character.
 
\subsection{Exceptional zeros of $p$-non-ordinary $p$-adic $L$-functions}
Our goal in this section is to give a proof of Theorem~A of the introduction. In the remainder of this paper, we assume that  $f=\underset{n=1}{\overset{\infty}\sum}a_nq^n$ is an  elliptic newform of even weight $k$ for $\Gamma_0(Np)$ such that $a_p=p^{k/2-1}$ and  $r_{\textup{an}}(f)>0$.  Let $\left[z_f^{\textup{BK}}\right]=\mathrm{pr}_0\left(\left[z_{f,\Iw}^{\textup{BK}}\right]\right)\in H^1 (G_{\QQ,S},V_f)$ denote the Beilinson-Kato element at the ground level.  Kato's explicit reciprocity law and our hypothesis on $r_{\textup{an}}(f)$ imply that  $\left[z_f^{\textup{BK}}\right]\in H^1_{\textup{f}}(\QQ,V_f)$. We denote by  $\left[\frak{z}^{\textup{BK}}_f\right] \in \widetilde{H}^1_{\textup{f}}(V_f)$ the canonical lift of $[z^{\textup{BK}}_f]$ with respect to the splitting $\textup{spl}$ of Proposition~\ref{prop:splitting}.

\begin{thm}
\label{tim:main}  We have,
$$ \frac{\left(1-\frac{1}{p} \right)\cdot\left\langle\Psi_1,\partial^{\textup{loc}}_0(d_{\widetilde{\delta}})\right\rangle \cdot\left\langle\Psi_2,\res_p[z^{\textup{BK}}_f]\right\rangle }{2\cdot\Gamma(k/2)}\cdot\frac{d^2}{ds^2}\,L_{p,\alpha}(f,\omega^{k/2},s)\big{|}_{s=k/2}=\frac{\frak{Reg}_{\frak h_{p}}\left(\partial_0(d_{\widetilde{\delta}}),\left[\frak{z}_f^{\textup{BK}}\right]\right)}{\left [e_\alpha,b_f^*\right ]_{V_f}}\,\,.$$
 \end{thm}
\begin{proof} 
It follows from \cite[Equation (68)]{BB_CriticalKato1} (which is a more precise form of Theorem A.(i) in op. cit. concerning the center $k$ of the affinoid disc ${\rm Sp}(A)$)  that 
\begin{equation}
\label{eqn_specialize_big_Kato_at_centre_precise_68_BBCK1A}
\psi_k(\frak{Z}^{\textup{BK}}_{\mathbf{f}})= {\rm Tw}_{\frac{k}{2}} \mathscr{E}_N(k) [z_{f,\Iw}^{\textup{BK}}],
\end{equation} 
where we recall that
$$ {\rm Tw}_{\frac{k}{2}} \mathscr{E}_N(k):=\prod_{\ell\mid N} (1-a_\ell(f)\ell^{-\frac{k}{2}}\sigma_\ell^{-1})$$
is the product of (twisted) bad Euler factors away from $p$, where $\sigma_\ell\in \Gamma$ is the unique element with $\chi(\sigma_\ell)=1$. Note that the Amice transform of ${\rm Tw}_{\frac{k}{2}} \mathscr{E}_N(k)$ is then 
$$ \mathscr{A}_N(k,s)=\prod_{\ell\mid N} (1-a_\ell(f)\ell^{-\frac{k}{2}}\langle \ell\rangle^{-s})\,.$$
The identity \eqref{eqn_specialize_big_Kato_at_centre_precise_68_BBCK1A} in turn tells us that 
\begin{equation}
\label{eqn_thm_main_new_eqn_1}
{\rm pr}_0\circ \psi_k(\frak{Z}^{\textup{BK}}_{\mathbf{f}})= \mathds{1}\left({\rm Tw}_{\frac{k}{2}} \mathscr{E}_N(k)\right) [z_{f}^{\textup{BK}}],
\end{equation} 
where we note that 
$$ \mathds{1}\left({\rm Tw}_{\frac{k}{2}} \mathscr{E}_N(k)\right)=\mathscr{A}_N(k,0)=\prod_{\ell\mid N} (1-a_\ell(f)\ell^{-\frac{k}{2}})\in E^\times$$ 
by the weight-monodromy conjecture for modular forms  (c.f. \cite[Corollary 2]{SaitoWeightMonodromySupplement}; see also \cite{BB_CriticalKato1}, Lemma 6.13). It follows from Theorem~\ref{them:propertiestwovarPRlog}(i) together with \eqref{eqn_specialize_big_Kato_at_centre_precise_68_BBCK1A} and \eqref{interpolationkato} that
\be\label{eqn:descentfrom2varto1}
L_p(\frak{Z}^{\textup{BK}}_\mathbf{f},k,s)=\mathscr{A}_N(k)\,L_p([z_{f,\Iw}^{\textup{BK}}],s)= \mathscr{A}_N(k)\,L_{p,\alpha}(f,\omega^{k/2},s+k/2)\,[e_\alpha, b^*_{f}]\,.
\ee
Set $\Omega_p:=\left(1-\frac{1}{p}\right)\Gamma(k/2)^{-1}\left\langle\Psi_2,\textup{res}_p([z^{\textup{BK}}_f])\right\rangle\cdot\left\langle\Psi_1,\partial^{\textup{loc}}_0(d_{\widetilde{\delta}})\right\rangle$, so that we have 
\begin{align*}
\Omega_p\cdot \,&\mathscr{A}_N(k,0)\cdot  L_p(\frak{Z}^{\textup{BK}}_\mathbf{f},k,s)  \\
&\equiv \mathscr{A}_N(k,0)^2\cdot \mathcal{A}\circ\det\left(\begin{array}{cc} \mathbb{H}_\mathbf{f}\left(\partial_0(d_{\widetilde{\delta}}),\partial_0(d_{\widetilde{\delta}})\right) & \mathbb{H}_\mathbf{f}\left(\partial_0(d_{\widetilde{\delta}}),[\frak{z}_f^{\textup{BK}}]\right)\\\\  
\mathbb{H}_\mathbf{f}\left([\frak{z}_f^{\textup{BK}}],\partial_0(d_{\widetilde{\delta}})\right) &  \mathbb{H}_\mathbf{f}\left([\frak{z}_f^{\textup{BK}}],[\frak{z}_f^{\textup{BK}}]\right)\end{array}\right)\Bigg{|}_{\kappa=k}\\\\
&\equiv
\mathscr{A}_N(k,0)^2\cdot \det\left(\begin{array}{cc} -\frak{h}_p\left(\partial_0(d_{\widetilde{\delta}}),\partial_0(d_{\widetilde{\delta}})\right) & -\frak{h}_p\left(\partial_0(d_{\widetilde{\delta}}),[\frak{z}_f^{\textup{BK}}]\right)\\\\  -\frak{h}_p\left([\frak{z}_f^{\textup{BK}}],\partial_0(d_{\widetilde{\delta}})\right) &  -\frak{h}_p\left([\frak{z}_f^{\textup{BK}}],[\frak{z}_f^{\textup{BK}}]\right)\end{array}\right)\cdot s^2 \mod s^3\,\\\\
&=\mathscr{A}_N(k,0)^2\cdot \frak{Reg}_{\frak h_{p}}\left(\partial_0(d_{\widetilde{\delta}}),\left[\frak{z}_f^{\textup{BK}}\right]\right),
\end{align*}
where the first equality follows from  Theorem~\ref{thm:leadingtermformula} and \eqref{eqn_thm_main_new_eqn_1}, whereas the second from \eqref{eqn:twovarexpanded1}. The proof now follows combining this calculation with \eqref{eqn:descentfrom2varto1} and the fact that $\mathscr{A}_N(k,0)\neq 0$.
\end{proof}
\begin{cor} Assuming the non-vanishing $\mathcal L_{\textup{FM}}(f)$ we have
\begin{multline*}\left (1-\frac{1}{p}\right )\frac{[e_{\alpha},b_f^*]_{V_f}}{2\Gamma(k/2)} \langle \Psi_2, \res_p([z_f^{\textup{BK}}])\rangle \cdot \frac{d^2}{ds^2}  L_{p,\alpha}(f,\omega^{k/2},s)\big{|}_{s=k/2}=\\
\mathcal L_{\textup{FM}}(f)\cdot h^{\textup{Nek}}([z^{\textup{BK}}_f], [z^{\textup{BK}}_f]).
\end{multline*}
Here $h^{\textup{Nek}}$ denotes Nekov\'a\v r's $p$-adic height in \cite{Ne92} associated to the splitting of the Hodge filtration of $\Dst (V_f)$
induced  by $D.$ 
\end{cor}
\begin{proof} This formula is a formal consequence of the computation of   
$\frak{h}_p$ in terms of  $h^{\textup{Nek}}$ 
 (see, for example, the proof of \cite[Proposition~11.4.9]{Ne06} for the ordinary case). By \cite[Proposition 20]{Ben14b}, the pairing  $h^{\textup{Nek}}$ coincides with the pairing 
$h^{\spl}_{V_f,D}$ constructed in \cite{Ben14b}. Theorem~11 and Lemma~10
of \cite{Ben14b} together give
\[
h^{\textup{Nek}}([z_f^{\textup{BK}}],[z_f^{\textup{BK}}])=
\frak{h}_p([z_f^{\textup{BK}}],[z_f^{\textup{BK}}])+\frac{\langle \Psi_2, \res_p  ([z_f^{\textup{BK}}])\rangle^2}
{\langle \Psi_2,\partial^{\textup{loc}}_0(d_{\widetilde{\delta}})  \rangle}.
\]
Using Theorem~\ref{thm:thepropertiesofthetwovarheight}, we have
\[
\frak{Reg}_{\frak h_{p}}\left(\partial_0(d_{\widetilde{\delta}}),\left[\frak{z}_f^{\textup{BK}}\right]\right)=-\frak{h}_p([z_f^{\textup{BK}}],[z_f^{\textup{BK}}]) \langle \Psi_2,\partial^{\textup{loc}}_0(d_{\widetilde{\delta}})\rangle-
 \langle \Psi_2,\res_p  ([z_f^{\textup{BK}}])\rangle^2.
\]
Therefore
\[
\frac{\frak{Reg}_{\frak h_{p}}\left(\partial_0(d_{\widetilde{\delta}}),\left[\frak{z}_f^{\textup{BK}}\right]\right)}{\langle \Psi_1,\partial^{\textup{loc}}_0(d_{\widetilde{\delta}})\rangle }=-h^{\textup{Nek}}([z_f^{\textup{BK}}],[z_f^{\textup{BK}}])\cdot \frac{\langle \Psi_2,\partial^{\textup{loc}}_0(d_{\widetilde{\delta}})\rangle}{\langle \Psi_1,\partial^{\textup{loc}}_0(d_{\widetilde{\delta}})\rangle }=
\mathcal{L}_{\textup{FM}}(f)\cdot h^{\textup{Nek}}([z_f^{\textup{BK}}],[z_f^{\textup{BK}}]).
\]
Now the Corollary follows from Theorem~\ref{tim:main}.
\end{proof}

\begin{rem} Theorem~\ref{tim:main} can be  also deduced from the existence of Beilinson--Kato
classes $[z^{\textup{BK}}_{f,\Iw}]$ and the one variable analog of Theorem~\ref{thm:leadingtermformula}, without any appeal to deformations of $V_f$ besides the cyclotomic deformation $\overline{V}_f$.  
 \end{rem}

\subsection{On the conjecture of Perrin-Riou}
\label{subsec:PRconjproved}
We keep notation and conventions of the  previous subsection. Assume, in addition, that $\varepsilon_f=-1$ and $\epsilon$ is chosen as the trivial character. Note that the condition (\ref{eqn:choiceofpsi}) holds.  


\begin{thm}
\label{thm:BKvsHeegner}
We have the following comparison between the Heegner cycle and the Beilinson-Kato element:
\begin{align*}t_f \cdot\left(\log\Phi^{\textup{AJ}}(y^\epsilon)\big{|}_{f^{\textup{rig}}}\right)^2&\left\langle\Psi_2,\, \textup{res}_p([z_{f}^{\textup{BK}}])\right\rangle=
\\&{a(k)}\cdot {\mathscr{A}_{N}(k,0)}\cdot \left(1-\frac{1}{p}\right)^{-1}\Gamma(k/2)\cdot\frac{\left\langle\Psi_2,\textup{res}_p([z_{f}^{\textup{BK}}])\right\rangle^2}{2\left\langle\Psi_1,\partial^{\textup{loc}}_0(d_{\widetilde{\delta}})\right\rangle}\,.\end{align*}
\end{thm}
\begin{proof}
We have the following chain of equalities:
\begin{align}
\notag t_f\cdot \left(\log\Phi^{\textup{AJ}}(y^\epsilon)\big{|}_{f^{\textup{rig}}}\right)^2\cdot&\left\langle\Psi_2,\,\textup{res}_p([z_{f}^{\textup{BK}}])\right\rangle\stackrel{\rm Thm~\ref{thm:sevesomainmazurKitag}}{=}\frac{d^2 \frak{L}_p}{d\kappa^2}\Big{|}_{\kappa=k}\cdot\left\langle\Psi_2,\,\textup{res}_p([z_{f}^{\textup{BK}}])\right\rangle\\
\notag&\stackrel{\rm Prop.~\ref{thm:BKvscentralcriticalL}}{=}{a(k)}\cdot\frac{d^2}{d\kappa^2} \frak{L}_p(\frak{Z}_{\mathbf{f}}^{\textup{BK}},\kappa)\Big{|}_{\kappa=k}\cdot\left\langle\Psi_2,\,\textup{res}_p([z_{f}^{\textup{BK}}])\right\rangle\\
\label{eqn:ccrubinstyleformula}&={a(k)}\cdot{\mathscr{A}_{N}(k,0)}\cdot\left(1-\frac{1}{p}\right)^{-1}\Gamma(k/2)\cdot\frac{\left\langle\Psi_2,\textup{res}_p([z_{f}^{\textup{BK}}])\right\rangle^2}{2\left\langle\Psi_1,\partial^{\textup{loc}}_0(d_{\widetilde{\delta}})\right\rangle} \,.
\end{align} 
Here, Equation \eqref{eqn:ccrubinstyleformula} follows from Corollary~\ref{thm:ccrubinstyleformula} and the fact that we have $\mathrm{pr}_{\gamma,\kappa}(\frak{Z}_{\mathbf{f}}^{\textup{BK}})=\mathscr{A}_{N}(k,0)[z^{\textup{BK}}_f]$ (c.f. Equation \eqref{eqn_thm_main_new_eqn_1}). 
\end{proof}

The following theorem gives a partial answer to a question of Perrin-Riou.
\begin{thm}
\label{thm:PRconj}
Suppose that $\varepsilon_f=-1$ and $\epsilon$ is chosen as the trivial character. Assume that $\log\Phi^{\textup{AJ}}(y^\epsilon)\big{|}_{f^{\textup{rig}}}$ is non-vanishing. Then the restriction of the Beilinson--Kato class $\textup{res}_p\left([z_{f}^{\textup{BK}}]\right)$ at the prime $p$ does not vanish. 
\end{thm}
\begin{proof}
It follows from Theorem~\ref{thm:sevesomainmazurKitag}(ii) that $\textup{ord}_{\kappa=k}\,\frak{L}_p(\kappa)=2$. Proposition~\ref{prop:whendoescentralcriticalLvanish} in turn implies that $\textup{res}_p([z_{f}^{\textup{BK}}])\neq0$, as desired. 
\end{proof}

\begin{cor}
\label{cor:orderofvanishingofpadicL} 
Suppose that the $p$-adic Abel--Jacobi image $\log\Phi^{\textup{AJ}}(y^\epsilon)\big{|}_{f^{\textup{rig}}}$ of the Heegner cycle $y^\epsilon$ is non-trivial. Assume further that the (cyclotomic) $p$-adic height pairing $\frak{h}_p$ is non-degenerate. Then,
$$\textup{ord}_{s=\frac{k}{2}}\, L_p(f,s) =2.$$ 
\end{cor}
\begin{proof}
It follows from Theorem~\ref{thm:PRconj} that $\left\langle\Psi_2,\res_p([z^{\textup{BK}}_f])\right\rangle$ and therefore also $\Omega_p$ is non-zero. Theorem~\ref{thm:PRconj} also shows that the Beilinson-Kato class $\left[{z}_f^{\textup{BK}}\right] \in H^1_\textup{f}(V_f)$ is non-trivial, so that the pair $\left\{\partial_0(d_{\widetilde{\delta}}), \left[\frak{z}_f^{\textup{BK}}\right]\right\}$ is linearly independent. The proof follows by our assumption that  $\frak{h}_p$ is non-degenerate.
\end{proof}

\begin{rem} If $\textup{res}_p([z_{f}^{\textup{BK}}]) \neq 0,$ 
we infer using Theorem~\ref{thm:BKvsHeegner} that
\begin{align*}t_f \cdot\left(\log\Phi^{\textup{AJ}}(y^\epsilon)\big{|}_{f^{\textup{rig}}}\right)^2={a(k)}\cdot {\mathscr{A}_{N}(k,0)}\cdot \left(1-\frac{1}{p}\right)^{-1}\Gamma(k/2)\cdot\frac{\left\langle\Psi_2,\textup{res}_p([z_{f}^{\textup{BK}}])\right\rangle}{2\left\langle\Psi_1,\partial^{\textup{loc}}_0(d_{\widetilde{\delta}})\right\rangle}\,.\end{align*}
\end{rem}

\begin{rem}
\label{rem:grosszagierimpliesnonvanheeg}
When the order of vanishing $r_{\textup{an}}(f)$ of the Hecke $L$-function $L(f,s)$ equals $1$ (and only then) the hypothesis of Theorem~\ref{thm:PRconj} is expected to hold true.
Indeed, a suitable extension of Gross--Zagier--Zhang formula on Shimura curves shows that the Heegner cycle $y^{\epsilon}$ is non-torsion if $r_{\textup{an}}(f)=1$. The desired formula is available in the literature for weight two forms and for Heegner cycles on classical modular curves for higher weight forms by the celebrated works of Zhang~\cite{zhanghigherweight,zhangshimuraweight2}. The experts\footnote{The second named author would like to thank Daniel Disegni, Victor Rotger, Ariel Shnidman, Xinyi Yuan and Wei Zhang for correspondences on this matter.} seem to believe that the current technology would be sufficient to extend these results to our case of interest. Furthermore, the $p$-adic Abel--Jacobi map $\Phi^{\textup{AJ}}$ is also expected to be injective in this set up. When the weight of the eigenform equals to $2$, this follows from Kummer theory. 
\end{rem}

\begin{rem}
As we have already pointed out in Remark~\ref{rem:grosszagierimpliesnonvanheeg}, a suitable Gross--Zagier--Zhang formula together with the injectivity of the Abel--Jacobi map and the non-degeneracy of Beilinson's height pairing (none of which is currently known in the required level of generality) would show that the non-vanishing of the Abel--Jacobi image of the Heegner cycle $y^\epsilon$ is equivalent to asking that the Hecke $L$-function associated to $f$ vanishes at $s=k/2$ to exact order $1$. If this is the case, one may prove that $H^1_\textup{f}(\QQ,V_f)$ is one-dimensional and the requirement that $\frak{h}_p$ be non-degenerate is equivalent to asking that it is non-zero.
\end{rem}


{\scriptsize
\bibliographystyle{halpha}
\bibliography{references}

\bibliographystyle{style}

\end{document}